\documentclass[10pt]{amsart}
\usepackage{mathrsfs}
\usepackage{amsfonts}
\usepackage{amssymb}
\usepackage{amssymb,mathrsfs,graphicx,enumerate}
\usepackage{amsmath,amsfonts,amssymb,amscd,amsthm,bbm}
\usepackage[all]{xy}
\usepackage{extpfeil}
\usepackage{graphicx,colortbl}
\usepackage{float}
\numberwithin{equation}{section}

\title[Asymptotic stability of the Kuramoto-Sakaguchi equation with inertia]{Asymptotic stability of the phase-homogeneous solution to the Kuramoto-Sakaguchi equation with inertia}

\author[Choi]{Young-Pil Choi}
\address[Young-Pil Choi]{\newline Department of Mathematics and Institute of Applied Mathematics \newline Inha University, Incheon 22212, Korea (Repbulic of)}
\email{ypchoi@inha.ac.kr}

\author[Ha]{Seung-Yeal Ha}
\address[Seung-Yeal Ha]{\newline Department of Mathematical Sciences and Research Institute of Mathematics\newline
    Seoul National University, Seoul 08826 and \newline
    Korea Institute for Advanced Study, Hoegiro 85, Seoul, 02455, Korea (Republic of)}
\email{syha@snu.ac.kr}

\author[Xiao]{Qinghua Xiao}
\address[Qinghua Xiao]{\newline Wuhan Institute of Physics and Mathematics, the Chinese Academy of Sciences \newline Wuhan 430071, People's Republic of China}
\email{xiaoqh@wipm.ac.cn}

\author[Zhang]{Yinglong Zhang}
\address[Yinglong Zhang]{\newline Department of Mathematical Sciences and \newline Industrial and Mathematical Data Analytics Research Center\newline Seoul National University, Seoul 08826, Korea (Repbulic of)}
\email{yinglongzhang@amss.ac.cn}

\newtheorem{theorem}{Theorem}[section]
\newtheorem{lemma}{Lemma}[section]

\newtheorem{proposition}{Proposition}[section]
\newtheorem{remark}{Remark}[section]
\newtheorem{definition}{Definition}[section]

\newcommand{\bbr}{\mathbb{R}}
\newcommand{\bbt}{\mathbb{T}}
\newcommand{\bbp}{\mathbb{P}}
\newcommand{\bbi}{\mathbb{I}}
\newcommand{\bbe}{\mathbb{E}}

\newcommand{\mci}{\mathcal{I}}

\newcommand{\mfa}{\mathfrak{a}}
\newcommand{\mfb}{\mathfrak{b}}

\begin{document}



\thanks{\textbf{Acknowledgment.} The work of Y.-P. Choi is supported by National Research Foundation of Korea(NRF) grant funded by the Korea government(MSIP) (No. 2017R1C1B2012918 and 2017R1A4A1014735) and POSCO Science Fellowship of POSCO TJ Park Foundation. The work of S.-Y. Ha is supported by the National Research Foundation of Korea(NRF-2017R1A2B2001864). The work of Q. Xiao was supported by grants from Youth Innovation Promotion Association and the National Natural Science Foundation of China \#11501556. The work of Y. Zhang is supported by the grant MSIP-2017R1A5A1015626}

\begin{abstract}
We present the global-in-time existence of strong solutions and its large-time behavior for the Kuramoto-Sakaguchi equation with inertia. The equation describes the evolution of the probability density function for a large ensemble of Kuramoto oscillators under the effects of inertia and stochastic noises. We consider a perturbative framework around the equilibrium, which is a Maxwellian type, and use the classical energy method together with our careful analysis on the macro-micro decomposition. We establish the global-in-time existence and uniqueness of strong solutions when the initial data are sufficiently regular, not necessarily close to the equilibrium, and the noise strength is also large enough. For the large-time behavior, we show the exponential decay of solutions towards the equilibrium under the same assumptions as those for the global regularity of solutions.
\end{abstract}

\maketitle


%
%
%
%
\section{Introduction} \label{sec:1}
\setcounter{equation}{0}
Synchronization phenomenon is ubiquitous in an ensemble of weakly coupled oscillators, e.g., hand clapping in opera and musical halls, flashing of fireflies and heart beating of pacemaker cells, etc \cite{A-B,BB,Erm,P-R-K,Str,Wi1,Wi2}. 
In this paper, we consider a large-time dynamics of infinitely many coupled Kuramoto oscillators with inertia in the presence of stochastic noises. Let  $(\theta^i_t, \omega_t^i)$ be the phase-frequency processes of the $i$-th Kuramoto oscillator with a natural frequency $\nu_i$, which is assumed to be a random variable extracted from a given distribution function $g = g(\nu) \geq 0$, in the presence of inertia and stochastic noise effect. Then, it is well known that the phase dynamics is governed by the system of Langevin type equations \cite{A-B, Ku, Saka88}:
\begin{align}
\begin{aligned} \label{B-1}
d\theta^i_t &= \omega^i_t dt, \quad t  > 0,~~i = 1, \cdots, N, \\
d\omega^i_t &= \frac{1}{m} \Big[  -\omega^i_t + \nu^i +   \frac{\kappa}{N}\sum_{j=1}^{N}\sin{(\theta^j_t - \theta^i_t)}  \Big] dt + \frac{\sqrt{2\sigma} }{m} dB^i_t,
\end{aligned}
\end{align}
where $m, \kappa$, and $\sigma$ are non-negative constants representing the strength of inertia, coupling strength and the noise strength, respectively, and the white noise process $dB^i_t$ satisfies the following mean zero and covariance relations:
\begin{equation*} \label{B-2}
\bbe[dB^i_t] = 0, \qquad \bbe[dB^i_s dB^j_t] = \delta_{ij} \delta(t-s), \quad 1 \leq i, j \leq N, \quad t, s > 0.
\end{equation*}
Note that in the absence of noise effect, i.e., $\sigma = 0$, system \eqref{B-1} formally becomes the inertial Kuramoto model:
\begin{equation}\label{Ku_wod}
m \frac{d^2{\theta}^i_t}{dt^2} = -\frac{d{\theta}^i_t}{dt} + \nu^i +  \frac{\kappa}{N} \sum_{j=1}^N  \sin(\theta^j_t - \theta^i_t).
\end{equation}

The system \eqref{Ku_wod} has been well studied, and it still remains a popular subject in nonlinear dynamics. For the system \eqref{Ku_wod}, the phase transition and hysteresis phenomena are studied in \cite{T-L-O1, T-L-O2} by analyzing the order parameter $r^\infty$ which is defined through the following relation:
\[
r^N(t) e^{\sqrt{-1}\varphi_N(t)} := \frac1N \sum_{j=1}^N e^{\sqrt{-1} \theta_j(t)}, \quad r^\infty := \lim_{t \to +\infty} \lim_{N \to +\infty} r^N(t).
\]

More precisely, no matter what the initial condition is, the continuous/discontinous phase transitions from the disordered state ($r^\infty = 0$) to the partially ordered states ($0 < r^\infty  \leq 1$) and hysteresis phenomena are observed in \cite{BM16, T-L-O1, T-L-O2}. The asymptotic phase/frequency synchronization and stability estimates are rigorously obtained in \cite{C-H-L-X-Y,C-H-N,C-H-Y1,H-L} under certain assumptions on the initial configurations, and the rigorous mean-field limit is provided in \cite{C-H-Y2}. Very recently, the asymptotic frequency synchronization estimate for the identical oscillators case is improved in \cite{C-H-M}. For the Kuramoto model, i.e., system \eqref{Ku_wod} with $m = 0$, similar works are done in \cite{C-C-H-K-K, C-H-J-K}. The assumptions on the initial phases are totally removed in \cite{H-K-R}, when the coupling strength $\kappa$ is large enough.

As the number of oscillators $N$ tends to infinity, i.e., $N \to +\infty$, we can derive the kinetic equation for $F$, namely {\it the inertial Kuramoto-Sakaguchi equation} using the standard BBGKY hierarchy or propagation of chaos \cite{A-B,H-K-P-Z}. Let $F=F(\theta, \omega,\nu,t)$ be an one-oscillator probability density function at phase $\theta$ with frequency $\omega$, and natural frequency $\nu$. Then, its phase-space-temporal evolution is governed by the following Vlasov-Fokker-Planck type equation:
\begin{align}
\begin{aligned} \label{FP}
& \partial_t F + \partial_{\theta} \big(\omega F \big) + \partial_\omega \big(\mathcal {A}[F]F \big) = \frac{\sigma}{m^2}\partial^2_\omega F, \quad (\theta, \omega, \nu, t) \in \mathbb{T} \times \mathbb{R}^2 \times \mathbb{R}_+,  \\
& \mathcal {A}[F](\theta, \omega, \nu, t):=  \frac{1}{m} \Big( - \omega + \nu + \kappa \underbrace{\int_{\mathbb{T} \times \mathbb{R}^2} \sin{(\theta_{*}-\theta)} F(\theta_{*},\omega_{*},\nu_{*},t) \,d\theta_{*}dw_{*}d\nu_{*}}_{=: \mathcal {S}[F]} \Big),
\end{aligned}
\end{align}
subject to the constrained initial data $F^{in}$:
\begin{align}
\begin{aligned} \label{ini}
& F(\theta, \omega, \nu, 0) = F^{in}(\theta, \omega, \nu), \quad (\theta, \omega, \nu) \in \bbt \times \bbr^2,  \\
&  \int_{\mathbb{T} \times \mathbb{R}} F^{in}(\theta,\omega,\nu)\,d\theta d\omega = g(\nu), \quad \int_{\mathbb{T} \times \mathbb{R}^2} F^{in}(\theta,\omega,\nu)\,d\theta d\omega d\nu = 1.
\end{aligned}
\end{align}

In this paper, we are interested in the global-in-time existence and uniqueness of strong solutions to the Cauchy problem \eqref{FP}-\eqref{ini}, and its large-time behavior. Note that the following Maxwellian type function $M = M(\omega, \nu)$ is a stationary homogeneous solution to \eqref{FP} (see Section \ref{sec:3.1} for more details):
\begin{equation} \label{A-1}
M := \frac{1}{2\pi} \sqrt{\frac{m}{2\pi \sigma}} \exp\Big[ - \frac{m}{2\sigma}(\omega - \nu)^2 \Big ] g(\nu).
\end{equation}
Motivated from the above observation, we reformulate equation \eqref{FP} in terms of perturbation $f$:
\begin{equation*} 
F = M + \sqrt{M} f.
\end{equation*}
Then, by a straightforward computation, we find that the perturbation $f$ satisfies
\begin{align}
\begin{aligned} \label{A-3}
& \partial_t f + \omega \partial_{\theta} f = {\mathcal L} f + {\mathcal N}(f, f), \\
& f(\theta, \omega, \nu, 0) = \frac{1}{\sqrt{M}} \big( F^{in} - M \big) =: f^{in},
\end{aligned}
\end{align}
where the linear operator $\mathcal{L} = {\mathcal L}_0 + \mathcal{L}_1$, and the nonlinear operator  ${\mathcal N}$ are given as follows:
\begin{align*}
\begin{aligned} 
{\mathcal L}_0 f &:= \frac{\sigma}{m^2}  \frac{1}{\sqrt{M}} \partial_{\omega} \Big[  M \partial_{\omega} \Big( \frac{f}{\sqrt{M}}  \Big) \Big], \cr
 {\mathcal L}_1 f &:=\frac{\kappa}{\sigma} (\omega - \nu) \sqrt{M} \mathcal{S}[\sqrt{M} f], \quad \mbox{and}\\
{\mathcal N}(f,f) &:= \frac{\kappa}{2\sigma} \mathcal {S}[\sqrt{M} f] (\omega - \nu) f -  \frac{\kappa}{m} \mathcal{S}[\sqrt{M} f] \partial_{\omega}f.
\end{aligned}
\end{align*}
It is worth mentioning that there are interesting works on the stability and asymptotic dynamics of solutions for the original continuum Kuramoto/Kuramoto-Sakaguchi equations; stability of the coherent and incoherent states are investigated in \cite{BCM,FGC16, GLP14, H-K-M-P, H-X1, H-X2, M-S3}, very recently, the stability and bifurcation for the continuum Kuramoto model are studied in \cite{Diet16,DFG16}, inspired by the proof of Landau damping for Vlasov-Poisson equation \cite{MV11} and Vlasov-HMF equation \cite{FR16}.
Despite these fruitful developments on the existence theory and stability estimates for the continuum Kuramoto/Kuramoto-Sakaguchi equations, to the best of the authors' knowledge, the global existence of solutions and its asymptotic dynamics for the inertial Kuramoto-Sakaguchi equation \eqref{FP} has not been studied so far.

It should be noted that the stationary homogeneous solution $M$ defined in \eqref{A-1} is not a global Maxweillian, and this leads to the fact that our main equation \eqref{FP} does not have the same properties of the classical linear Vlasov-Fokker-Planck equation. To be more specific, the linear operator $\mathcal{L}$ does not satisfy the coercivity estimate;  only ${\mathcal L}_0$ satisfies the coercivity estimate (see Lemma \ref{L3.2} (ii) for details). That is why we split the linear part $\mathcal{L}$ of the equation \eqref{A-3} into two linear parts ${\mathcal L}_0f$ and ${\mathcal L}_1f$. For that reason, a more delicate analysis is required for the energy estimates. We introduce the projection operator
\begin{align*}
\begin{aligned} 
& \bbp: L_{\omega, \nu}^{2} \rightarrow span\{ \chi_0, \chi_1 \},~~ \mbox{where}~ \chi_0 := \sqrt{2\pi} \sqrt{M} \quad \mbox{and} \quad \chi_1 := \sqrt{(2\pi m)/\sigma} (\omega - \nu) \sqrt{M}. 
\end{aligned}
\end{align*}
Then, by using the decomposition $f = \bbp f + (\bbi - \bbp)f$ and the estimates of $\bbp f$ and $(\bbi - \bbp)f$ separately, we get the coercivity estimate of $\mathcal{L}_0$ and control the bad terms produced by linear/nonlinear operators $\mathcal{L}_1 f$ and $\mathcal{N}(f,f)$. Our careful analysis of that enables us to obtain the uniform bound of the Sobolev norm $\|f\|_{H^s}$ in time, by which we show the global-in-time existence and uniqueness of strong solutions to the equation \eqref{A-3} without any smallness assumption on the initial data of the perturbation $f$. Now, we state the main results of the current work.
\begin{theorem} \label{T-I} Let $s \geq 1$. Suppose that the initial data $f^{in} \in H^{s}(\bbt \times \bbr^2)$ satisfying $F^{in} = M + \sqrt{M} f^{in} \geq 0$. We also assume that the strength of the noise $\sigma$ satisfies
\[
\sigma \geq C \max\left\{ m\kappa^2, \kappa, \frac1m, m\|g\|_{\nu}^{2} \right\}, \quad \mbox{where} \quad \|g\|_{\nu}^2 := \int_{\bbr} (1+\nu^2)g(\nu) \,d\nu < \infty,
\]
for sufficiently large $C>0$. Then, the Cauchy problem \eqref{A-3} admit a unique global solution $f(\theta, \omega, \nu, t) \in \mathcal{C}\big([0, \infty); H^{s}(\bbt \times \bbr^2) \big)$ such that
\begin{align*}
\begin{aligned}
F = M + \sqrt{M}f \geq 0 \quad \text{and} \quad\|f(t)\|_{H^{s}(\bbt \times \bbr^2)} \leq e^{-\bar{C} t} \|f^{in}\|_{H^{s}}, \quad t > 0,
\end{aligned}
\end{align*}
where $\bar{C}$ is a positive constant depending on $m$.
\end{theorem}

\begin{remark} The smallness assumption on the initial data $f^{in}$ is not required for the global existence of solutions. Thus our result is more stronger than the asymptotic stability estimate of solutions near equilibrium.
\end{remark}

The rest of this paper is organized as follows. In Section \ref{sec:2}, we briefly discuss the basic properties of the inertial Kuramoto-Sakaguchi equation, and we also introduce macroscopic observables and derive local balanced laws for mass, momentum and energy density.  In Section \ref{sec:3}, we derive the stationary homogeneous solution to \eqref{FP}, which is the Maxwellian type given in \eqref{A-1}. We also present the coercivity estimate of the linear operator $\mathcal{L}_0$ based on the macro-micro decomposition through the projection operator $\bbp$. In Section \ref{sec:4}, we study the existence and uniqueness of local-in-time strong solutions to \eqref{A-3}. Then, in Section \ref{sec:5}, by obtaining  uniform a priori estimates, we have the existence and uniqueness of global-in-time strong solutions and the exponential decay towards the equilibrium under certain conditions given in Theorem \ref{T-I}.  Finally, Section \ref{sec:6} is devoted to the brief summary of the present work and future directions. In the sequel, we list several simplified notation to be used throughout the paper. \newline

\noindent {\bf Notation}: Given any functions $h_1 = h_1(\theta, t)$ and $h_2 = h_2(\theta,\omega,\nu,t)$ defined on $\bbt \times \bbr_+$ and $\bbt \times \bbr^2 \times \bbr_+$, respectively.
\begin{enumerate}
\item We denote by $L_{\theta}^{2}$, $L_{\omega}^{2}$, $L_{\omega, \nu}^{2}$, and $L_{\theta,\omega, \nu}^{2}$ the usual Lebesgue space $L^2(\bbt)$, $L^{2}(\bbr)$, $L^{2}(\bbr^2)$, and $L^{2}(\bbt \times \bbr^2)$, respectively, and $\|\cdot\|_{L^2}$ represents the norm $\|\cdot\|_{L_{\theta, \omega, \nu}^{2}}$ for function $h_2$ and the norm $\|\cdot\|_{L_{\theta}^{2}}$ for function $h_1$ if there is no confusion. \newline

\item Let $k$ and $\ell$ be non-negative integers. $H^k$ denotes the $k$-th order $L^2$ Sobolev space. $\mathcal{C}^\ell([0,T]; H^k)$ denotes the set of $k$-times continuously differentiable functions from an interval $[0,T] \subset \mathbb{R}$ into Banach space $H^k$. \newline

\item For $\alpha = 1 + (m/\sigma)(\omega - \nu)^2$, $\beta = \sigma/m$, we denote by $L_{\mu}^{2}$ and $L_{\theta, \mu}^{2}$ the weighted Lebesgue space $L_{\omega, \nu}^{2}$ and $L_{\theta, \omega, \nu}^{2}$, respectively, and we set
\[ \|h_i\|_{L_\mu^2}^{2} := \int_{\bbr^2} \alpha |h_i|^2 + \beta |\partial_{\omega} h_i|^2 \,d\omega d\nu, \quad \|h_i\|_{\mu}^2 := \int_{\mathbb{T}} \|h_i\|_{L_{\mu}^2}^2 \,d\theta. \]

\item $f_1 \lesssim f_2$ represents that there exists a positive constant $c>0$ such that $f_1 \leq c f_2$; $f_1 \simeq f_2$ means that there exists a constant $c>0$ such that $c^{-1}f_1 \leq f_2 \leq cf_1$. We also denote by $C$ a generic positive constant independent of $t$.
\end{enumerate}

%
%
%
%
\section{Preliminaries} \label{sec:2}
In this section, we study basic properties of the equation \eqref{FP}. We then introduce macroscopic observables and derive local balanced laws for mass, momentum, and energy density.


\begin{lemma} \label{L2.1}
For a given $T \in (0, \infty]$, let $F =  F(\theta, \omega, \nu, t)$ be a global classical solution to \eqref{FP} - \eqref{ini} in the time-interval $[0, T)$ with initial datum $F^{in}$. Then, we have
\[
F(\theta, \omega,t)\geq0, \quad  \int_{\mathbb{T} \times \mathbb{R}} F(\theta,\omega,\nu, t)\,d\theta d\omega = g(\nu), \quad \mbox{and} \quad \int_{\bbt \times \bbr^2} F(\theta, \omega, \nu, t)\,d\theta d\omega d\nu = 1.
\]
\end{lemma}

\begin{proof}
It is easy to get that
\begin{align*}
\begin{aligned}
\frac{d}{dt}  \int_{\mathbb{T} \times \mathbb{R}} F(\theta,\omega,\nu, t)\,d\theta d\omega = \frac{d}{dt} \int_{\bbt \times \bbr^2} F(\theta, \omega, \nu, t) \,d\theta d\omega d\nu = 0.
\end{aligned}
\end{align*}
Thus we obtain
\begin{align*}
\begin{aligned}
&  \int_{\mathbb{T} \times \mathbb{R}} F(\theta,\omega,\nu, t)\,d\theta d\omega =  \int_{\mathbb{T} \times \mathbb{R}} F^{in}(\theta,\omega,\nu)\,d\theta d\omega = g(\nu), \\
& \int_{\bbt \times \bbr^2} F(\theta, \omega, \nu, t) \,d\theta d\omega d\nu = \int_{\bbt \times \bbr^2} F^{in}(\theta, \omega, \nu) \,d\theta d\omega d\nu = 1.
\end{aligned}
\end{align*}
The remaining thing is to prove the non-negativity of $F$. For this, we set $F = e^{t/m}u$, then the new variable $u$ satisfies 
\begin{align} \label{e2-3}
\begin{aligned}
\partial_t u = \frac{\sigma}{m^2}\partial^2_\omega u - \omega \partial_{\theta} u - \frac{1}{m} \Big( - \omega + \nu + \kappa \mathcal {S}[e^{\frac{t}{m}} u] \Big) \partial_{\omega}u.
\end{aligned}
\end{align}
Now we use the strong maximum principle or Feynman-Kac's formula to see that the solution of \eqref{e2-3} can not attain a negative minimum if $t > 0$. Thus, by using initial assumption $u|_{t= 0} = F^{in} \geq 0$, we get
\[
u(\theta, \omega, \nu, t) \geq 0 \quad \text{for} \quad t \geq 0,
\]
which yields the non-negativity of $F$.
\end{proof}

\subsection{Local balanced laws} \label{sec:2.2}
For simplicity of presentation, we consider the case of identical oscillators, i.e., $g(\nu) = \delta$, where $\delta$ is the Dirac Delta function whose mass is concentrated at $\nu = 0$. In this case, system \eqref{FP} becomes
\begin{align}
\begin{aligned} \label{B-3-1}
& \partial_t F + \partial_{\theta} \big(\omega F \big) + \partial_\omega \big(\mathcal {A}[F]F \big) = \frac{\sigma}{m^2}\partial^2_\omega F, \quad (\theta, \omega, t)
 \in \mathbb{T} \times \mathbb{R}  \times \mathbb{R}_+, \\
& \mathcal {A}[F](\theta, \omega, t) = \frac{1}{m} \Big( - \omega +  \kappa \mathcal {S}[F] \Big),
\end{aligned}
\end{align}
where $F = F(\theta, \omega, t)$ depends only on $\theta$, $\omega$, and $t$. We first introduce generalized macroscopic observables:
\begin{align}
\begin{aligned} \label{B-4}
& \rho := \int_{\bbr} F \,d\omega:~\mbox{local mass density}, \\
& \rho u :=  \int_{\bbr} \omega F \,d\omega:~\mbox{local momentum density}, \\
& \rho \Big( e + \frac{1}{2} u^2 \Big)  :=  \int_{\bbr} \frac{\omega^2}{2} F d\omega:~\mbox{local energy density}, \\
& p :=  \int_{\bbr} |\omega - u|^2  F \,d\omega~:\mbox{pressure}, \quad \mbox{and}  \\
& q:=   \int_{\bbr} \frac{(\omega- u)^3}{2}  F \,d\omega    :~\mbox{heat flux}.
\end{aligned}
\end{align}
Next, we formally derive a system of local conservation laws for observable variables defined in \eqref{B-4}. For notational simplicity, we suppress $t$-dependance in $F, \rho$, and $u$:
\[ F(\theta, \omega) := F(\theta, \omega, t), \quad \rho(\theta) := \rho(\theta, t), \quad \mbox{and} \quad u(\theta) := u(\theta, t). \]
\begin{lemma}
Let $F = F(\theta, \omega, t)$ be a global classical solution to \eqref{B-3-1} decaying to zero sufficiently fast as $|w| \to \infty$. Then, the local observables defined in \eqref{B-4} satisfy a  system of local balanced laws:
\begin{align}
\begin{aligned} \label{B-4-1}
& \partial_t \rho + \partial_\theta (\rho u) = 0, \\
& \partial_t (\rho u) + \partial_\theta (\rho u^2 + p) =  \frac{1}{m} \Big( -\rho u + \kappa \rho \int_{\bbt} \sin(\theta_* - \theta) \rho(\theta_*) \,d\theta_* \Big), \\
& \partial_t \Big[ \rho \Big(e + \frac{1}{2} |u|^2 \Big) \Big] + \partial_\theta  \Big( q +  \frac{3pu}{2} +  \frac{\rho}{2} u^3     \Big)  \\
& \hspace{3cm} =  \frac{\sigma}{m^2} \rho -  \frac{1}{m} \Big[ p + \rho u^2 - \kappa \rho u \int_{\bbt} \sin(\theta_* - \theta) \rho(\theta_*) \,d\theta_* \Big].
\end{aligned}
\end{align}
\end{lemma}

\begin{proof}
\noindent $\bullet$~(Local conservation of mass): We integrate \eqref{B-3-1} over $\omega$ to obtain
\begin{equation} \label{B-5}
\partial_t \rho + \partial_\theta (\rho u) = 0.
\end{equation}

\vspace{0.2cm}

\noindent $\bullet$~(Local balanced law of momentum): We multiply $\omega$ to \eqref{B-3-1} to get
\[ \partial_t (\omega F) + \partial_\theta (\omega^2 F)
+ \partial_{\omega} \Big (\omega  \mathcal {A}[F] F - \frac{\sigma}{m^2} \omega \partial_\omega F + \frac{\sigma}{m^2} F \Big) =  \mathcal {A}[F] F.  \]
We integrate the above relation with respect to $\omega$ to obtain
\begin{equation} \label{B-6}
\partial_t (\rho u) + \partial_\theta \int_{\bbr} \omega^2 F \,d\omega  = \int_{\bbr}  \mathcal {A}[F] F\, d\omega.
\end{equation}
Note that
\begin{align}
\begin{aligned} \label{B-7}
& \int_{\bbr} \omega^2 F \,d\omega  = \int_{\bbr} ( |\omega - u|^2 + |u|^2 + 2u(\omega - u) ) F\, d\omega  = p + \rho u^2, \\
& \int_{\bbr}  \mathcal {A}[F] F\, d\omega = \frac{1}{m} \Big( -\rho u + \kappa \int_{\bbt \times \bbr^2} \sin(\theta_* - \theta) F(\theta_*, \omega_*) F(\theta, \omega)  \,d\omega_* d\omega d\theta_*  \Big) \\
& \hspace{2cm} =  \frac{1}{m} \Big( -\rho u + \kappa \int_{\bbt} \sin(\theta_* - \theta) \rho(\theta_*) \rho(\theta) \,d\theta_* \Big).
\end{aligned}
\end{align}
Finally, we combine \eqref{B-6} and \eqref{B-7} to derive a local balanced law for momentum:
\begin{equation} \label{B-8}
\partial_t (\rho u) + \partial_\theta (\rho u^2 + p) =  \frac{1}{m} \Big( -\rho u + \kappa \rho \int_{\bbt} \sin(\theta_* - \theta) \rho(\theta_*) \,d\theta_* \Big).
\end{equation}

\vspace{0.2cm}

\noindent $\bullet$~(Local balanced law of energy): We multiply $\omega^2/2$ to \eqref{B-3-1} to get
\begin{equation} \label{B-9}
\partial_t \Big( \frac{\omega^2}{2} F \Big) + \partial_\theta \Big(\frac{\omega^3}{2} F \Big)  + \partial_\omega \Big[  \frac{\omega^2}{2} {\mathcal A}[F] F + \frac{\sigma}{m^2} \Big( \omega F - \frac{\omega^2}{2} \partial_\omega F \Big)               \Big] = \omega {\mathcal A}[F] F +   \frac{\sigma}{m^2} F.
\end{equation}
We then integrate \eqref{B-9} with respect to $\omega$ to obtain
\begin{equation} \label{B-10}
\partial_t \Big[ \rho \Big(e + \frac{1}{2} |u|^2 \Big) \Big]  + \partial_\theta \Big( \int_\bbr  \frac{\omega^3}{2} F \,d\omega \Big) = \frac{\sigma}{m^2} \rho + \int_{\bbr}  \omega {\mathcal A}[F] F \,d\omega.
\end{equation}
We use the identity
\[ \frac{\omega^3}{2} = \frac{1}{2}(\omega - u + u)^3 = \frac{1}{2} (\omega - u)^3 + \frac{3}{2} (\omega - u)^2 u + \frac{3}{2} (\omega - u) u^2 + \frac{1}{2} u^3 \]
to rewrite
\begin{equation} \label{B-11}
\int_\bbr  \frac{\omega^3}{2} F \,d\omega = q +  \frac{3pu}{2} +  \frac{\rho}{2} u^3.
\end{equation}
On the other hand, we also have
\begin{align}
\begin{aligned} \label{B-12}
\int_{\bbr}  \omega {\mathcal A}[F] F d\omega & = \frac{1}{m} \int_{\bbt} \Big[ -\omega^2 F + \kappa \omega F {\mathcal S}[F]   \Big]  d\omega =  \frac{1}{m} \Big( -p - \rho u^2  + \kappa \int_{\bbr} \omega F {\mathcal S}[F]  \,d\omega \Big)   \\
&= \frac{1}{m} \Big[ -p - \rho u^2 + \kappa \int_{\bbt} \sin(\theta_* - \theta) (\rho u)(\theta) \rho(\theta_*) \,d\theta_* \Big],
\end{aligned}
\end{align}
where in the second relation of \eqref{B-12}, we used the identity $\eqref{B-7}$. \newline

In \eqref{B-10}, we finally combine \eqref{B-11} and \eqref{B-12} to derive a balanced law for energy:
\begin{equation} \label{B-13}
\partial_t \Big[ \rho \Big(e + \frac{1}{2} |u|^2 \Big) \Big] + \partial_\theta  \Big( q +  \frac{3pu}{2} +  \frac{\rho}{2} u^3     \Big) =  \frac{\sigma}{m^2} \rho -  \frac{1}{m} \Big( p + \rho u^2 - \kappa  \rho u \int_{\bbt} \sin(\theta_* - \theta)  \rho(\theta_*) \,d\theta_* \Big).
\end{equation}
Now we collect all estimates \eqref{B-5}, \eqref{B-8}, and \eqref{B-13} to derive local balance laws.
\end{proof}

\begin{remark}
Note that the local conservation laws \eqref{B-4-1} are not a closed system. However, if we take $F = \rho e^{-\omega^2}$ to be a global Maxwellian, it is easy to see that the heat flux $q = 0$, and we obtain a hydrodynamic model for $\rho, \rho u$ and $\rho e$:
\begin{align*}
\begin{aligned}
& \partial_t \rho + \partial_\theta (\rho u) = 0, \\
& \partial_t (\rho u) + \partial_\theta (\rho u^2 + p) =  \frac{1}{m} \Big( -\rho u + \kappa \rho \int_{\bbt} \sin(\theta_* - \theta) \rho(\theta_*) \,d\theta_* \Big), \\
& \partial_t \Big[ \rho \Big( e + \frac{1}{2} |u|^2 \Big) \Big] + \partial_\theta  \Big( \frac{3pu}{2} +  \frac{\rho}{2} u^3 \Big) =   \frac{\sigma}{m^2} \rho -  \frac{1}{m} \Big( p + \rho u^2 - \kappa \rho u \int_{\bbt} \sin(\theta_* - \theta) \rho(\theta_*) \,d\theta_* \Big).
\end{aligned}
\end{align*}
\end{remark}
Note that the macroscopic system \eqref{B-4-1} is exactly the Euler system with synchronization dissipations. For $t \geq 0$, we set
\[ M_0(t) := \int_\bbt \rho(\theta, t) \,d\theta \quad \mbox{and} \quad M_1(t) :=  \int_\bbt (\rho u)(\theta, t)\, d\theta.   \]
\begin{lemma}
Let $(\rho, \rho u)$ be a global classical solution to \eqref{B-4-1}. Then, we have
\[ M_0(t) = M_0(0)\quad \mbox{and} \quad M_1(t) = M_1(0) e^{-\frac{t}{m}}, \quad t \geq 0.   \]
\end{lemma}

\begin{proof} (i)~We integrate the continuity equation $\eqref{B-4-1}_1$ with respect to $\theta$ to get a conservation of mass. \newline

\noindent (ii)~We integrate the equation $\eqref{B-4-1}_2$ with respect to $\theta$ to get
$$\begin{aligned}
\frac{d}{dt} \int_\bbt (\rho u)(\theta, t) \,d\theta &=  -\frac{1}{m} \int_{\bbt} (\rho u)(\theta, t) \,d\theta + \frac{\kappa}{m} \int_{\bbt^2} \sin(\theta_* - \theta) \rho(\theta_*,t) \rho(\theta, t)\, d\theta_* d\theta \cr
&= -\frac{1}{m} \int_{\bbt} (\rho u)(\theta, t) \,d\theta.
\end{aligned}$$
This yields the desired exponential decay of the first moment.
\end{proof}

%
%
%
%
\section{Phase-homogeneous solution and coercivity of $\mathcal{L}_0$} \label{sec:3}
\setcounter{equation}{0}
In this section, we study the derivation of the stationary homogeneous solution and coercivity of the linear operator $\mathcal{L}_0$.

\subsection{Phase-homogeneous solution} \label{sec:3.1}
In this subsection, we study the nonlinear stability setting of the Kuramoto-Sakaguchi equation near a phase-homogeneous solution $M$. We first rewrite the equation \eqref{FP} as follows.
\begin{equation} \label{C-1}
 \partial_t F + \omega \partial_{\theta} F = {\mathcal C}(F), \quad {\mathcal C}(F):=  -\partial_\omega \big(\mathcal {A}[F]F \big) +  \frac{\sigma}{m^2}\partial^2_\omega F.
\end{equation}
Note that ${\mathcal C}(F)$ in \eqref{C-1}  plays a role of the collision operator for the Boltzmann equation.

\begin{definition}
Let $F = F(\omega, \nu)$ be a stationary phase-homogeneous solution of \eqref{C-1}, if the L.H.S. of \eqref{C-1} vanishes, i.e., $F$ satisfies
\begin{equation} \label{C-2}
{\mathcal C}(F) = -\partial_\omega \big(\mathcal {A}[F]F \big) +  \frac{\sigma}{m^2}\partial^2_\omega F = 0.
\end{equation}
\end{definition}
\begin{remark}
By the analogy with the Boltzmann equation, our stationary phase-homogeneous solution corresponds to the Maxwellian type \eqref{A-1}.
\end{remark}

\vspace{0.2cm}

Next, we look for stationary ${\mathcal C}^2$ phase-homogeneous solutions $F_e = F_e(\omega, \nu)$  satisfying the zero far-field boundary and normalization conditions:
\begin{eqnarray*}
&& \lim_{|\omega| \to \infty} F_e(\omega, \nu, t) = 0, \quad \lim_{|\omega| \to \infty} \partial_\omega F_e(\omega, \nu, t) = 0, \cr
&& \int_{\bbt \times \bbr} F_e(\omega, \nu, t) \,d\theta d\omega = g(\nu) \quad  \mbox{for each $\nu \in \bbr$}.
\end{eqnarray*}
We integrate \eqref{C-2} with respect to $\omega$ to find
\[  -\mathcal {A}[F_e]F_e + \frac{\sigma}{m^2}\partial_\omega F_e = 0.    \]
This together with ${\mathcal S}[F_e] = 0$ yields
\[ \partial_\omega \ln F_e= -\frac{m}{\sigma}(\omega - \nu), \quad \mbox{i.e.,} \quad F_e = C e^{ -\frac{m}{2\sigma}(\omega - \nu)^2}.  \]
By the normalization condition
\[ \int_{\bbt \times \bbr} F_e(\omega, \nu, t) \,d\theta d\omega = g(\nu) \quad \mbox{for each $\nu$},  \]
 we have
\[ F_e(\omega, \nu) = \frac{1}{2\pi} \sqrt{\frac{m}{2\pi \sigma}} \exp\left(- \frac{m}{2\sigma}(\omega - \nu)^2\right) g(\nu).     \]
Note that the homogeneous solution $F_e$ exactly coincides with $M$ introduced in \eqref{A-1}.
\subsection{Equation for perturbation} \label{sec:3.2}
Recall that one of our purposes is to study the asymptotic stability of a phase homogeneous solution $M$. Thus, it is convenient to work with a perturbation $f = f(\theta, \omega, \nu, t)$ defined by the relation:
\begin{equation} \label{B-20}
F = M + \sqrt{M} f,
\end{equation}
Note that the following identities hold:
\begin{align}
\begin{aligned} \label{B-21}
& {\mathcal S}[M] = 0, \quad   {\mathcal A}[M] = -\frac{1}{m} (\omega - \nu), \quad \frac{\sigma}{m^2} \partial_{\omega} M + \frac{1}{m} (\omega - \nu)M = 0,  \quad \mbox{and}\\
& {\mathcal A}[M + \sqrt{M} f]  = -\frac{1}{m} (\omega - \nu) + \frac{\kappa}{m} {\mathcal S}[\sqrt{M} f].
\end{aligned}
\end{align}
We substitute the ansatz \eqref{B-20} into \eqref{C-1} to see
\begin{align}
\begin{aligned} \label{B-22}
& \partial_t F + \omega \partial_{\theta} F = \sqrt{M} (\partial_t f + \omega \partial_\theta f), \\
& {\mathcal C}(F) = \partial_\omega \Big[  -{\mathcal A}[M + \sqrt{M} f] (M + \sqrt{M} f) + \frac{\sigma}{m^2} \partial_\omega (M + \sqrt{M} f) \Big]  \\
& \hspace{1cm} = \partial_\omega \Big[ - \frac{\kappa}{m} {\mathcal S}[\sqrt{M} f] ( M  + \sqrt{M}f ) + \frac{1}{m} (\omega - \nu) \sqrt{M} f + \frac{\sigma}{m^2} \partial_\omega (\sqrt{M} f) \Big] \\
& \hspace{1cm} = - \frac{\kappa}{m} {\mathcal S}[\sqrt{M} f]  (\partial_\omega M + \partial_\omega (\sqrt{M}f )) + \partial_\omega  \Big[  \frac{1}{m} (\omega - \nu) \sqrt{M} f + \frac{\sigma}{m^2} \partial_\omega (\sqrt{M} f) \Big],
\end{aligned}
\end{align}
where we used the fact that $ {\mathcal S}[\sqrt{M} f]$ depends only on $\theta$ and $t$. \newline

On the other hand, we use $\partial_\omega M = -(m/\sigma)(\omega -\nu) M$ in \eqref{B-21} to see
\begin{align}
\begin{aligned} \label{B-23}
& \partial_\omega M + \partial_\omega (\sqrt{M}f ) =\partial_\omega M  + \frac{\partial_\omega M}{2\sqrt{M}} f + \sqrt{M} \partial_\omega f  \\
& \hspace{2.75cm} = -\frac{m}{\sigma} (\omega -\nu) M -\frac{m}{2\sigma} (\omega - \nu) \sqrt{M} f + \sqrt{M} \partial_\omega f, \\
& \frac{1}{m} (\omega - \nu) \sqrt{M} f + \frac{\sigma}{m^2} \partial_\omega (\sqrt{M} f)  \\
& \hspace{2.5cm} = \frac{1}{2m} (\omega - \nu) \sqrt{M} f + \frac{\sigma}{m^2} \sqrt{M} \partial_\omega f =  \frac{1}{2m} (\omega - \nu) M  \frac{f}{\sqrt{M}} + \frac{\sigma}{m^2} \sqrt{M} \partial_\omega f  \\
& \hspace{2.5cm} = \frac{\sigma}{m^2} \Big[  \sqrt{M} \partial_\omega f  - (\partial_\omega \sqrt{M} ) f  \Big] =  \frac{\sigma}{m^2} M \partial_\omega \Big(  \frac{f}{\sqrt{M}} \Big).
\end{aligned}
\end{align}
We combine \eqref{B-22} and \eqref{B-23} to obtain
\begin{align}
\begin{aligned} \label{B-24}
{\mathcal C}(F) &=  \frac{\kappa}{\sigma} {\mathcal S}[\sqrt{M} f] (\omega -\nu) M + \frac{\kappa}{2\sigma} {\mathcal S}[\sqrt{M} f] (\omega - \nu) \sqrt{M} f -\frac{\kappa}{m}{\mathcal S}[\sqrt{M} f] \sqrt{M} \partial_\omega f  \\
&+  \frac{\sigma}{m^2} \partial_\omega \Big[  M \partial_\omega \Big(  \frac{f}{\sqrt{M}} \Big)          \Big].
\end{aligned}
\end{align}
Finally, we combine $\eqref{B-22}_1$ and \eqref{B-24} to obtain a quasilinear equation for $f$:
$$\begin{aligned}
 \partial_t f + \omega \partial_{\theta} f &=  \frac{\sigma}{m^2} \frac{1}{\sqrt{M}} \partial_{\omega} \Big( M \partial_\omega \big(  \frac{f}{\sqrt{M}}  \Big) \Big)+  \frac{\kappa}{\sigma} \mathcal{S}[ \sqrt{M} f] (\omega - \nu) \sqrt{M}  \\
& \hspace{.5cm}  + \frac{\kappa}{2\sigma} \mathcal {S}[ \sqrt{M}f] (\omega - \nu) f -\frac{\kappa}{m} \mathcal{S}[\sqrt{M}f] \partial_{\omega}f \\
&= {\mathcal L}_0 f + {\mathcal L}_1 f + {\mathcal N}(f,f).
\end{aligned}$$
In next lemma, we present basic estimates on the linear operator ${\mathcal L}_0$.
\begin{lemma}\label{L3.1}
The linear operator ${\mathcal L}_0$ is self-adjoint with respect to $L_{\omega, \nu}^2$-inner product, and it satisfies
\begin{eqnarray*}
&&(i)~\langle -{\mathcal L}_0 f, f \rangle = \frac{\sigma}{m^2} \int_{\mathbb{R}^2} \Big| \partial_{\omega}\Big( \frac{1}{\sqrt{M}} f \Big) \Big|^2 M \,d\omega d\nu. \\
&&(ii)~\mbox{Ker}\,{\mathcal L}_0=  \big\{ h(\nu)\sqrt{M}  ~\big|~ h(\nu)\sqrt{M} \in L_{\omega, \nu}^{2}\big\}, \\
&& \hspace{0.6cm} \mbox{Range}\, {\mathcal L}_0 = \big\{ h(\nu)\sqrt{M} ~\big|~ h(\nu)\sqrt{M} \in L_{\omega, \nu}^{2}\big\}^{\perp},
\end{eqnarray*}
where $\langle \cdot, \cdot\rangle$ denotes the inner product in the Hilbert space $L^2_{\omega,\nu}(\mathbb{R}^2)$.
\end{lemma}
\begin{proof}
(i)~For $f, \bar{f} \in L^2_{\omega,\nu}(\mathbb{R}^2)$, we have
\begin{align}
\begin{aligned} \label{B-26}
\langle {\mathcal L}_0 f, \bar{f} \rangle &= \int_{\bbr^2} ({\mathcal L}_0 f)  \bar{f} \,d\omega d\nu =    \frac{\sigma}{m^2} \int_{\bbr^2}
\partial_{\omega} \Big[ M \partial_\omega \Big(  \frac{f}{\sqrt{M}}  \Big)  \Big] \frac{\bar{f}}{\sqrt{M}} \,d\omega d\nu  \\
&=- \frac{\sigma}{m^2} \int_{\bbr^2} \partial_\omega \Big(  \frac{f}{\sqrt{M}} \Big) \Big[ M \partial_\omega \Big( \frac{\bar{f}}{\sqrt{M}} \Big) \Big] \,d\omega d\nu \\
&= \frac{\sigma}{m^2} \int_{\bbr^2} \Big(  \frac{f}{\sqrt{M}} \Big) \partial_\omega \Big[ M \partial_\omega \Big( \frac{\bar{f}}{\sqrt{M}} \Big) \Big] \,d\omega d\nu = \langle  f, \mathcal{L}_0 \bar{f} \rangle.
\end{aligned}
\end{align}
Hence $\mathcal{L}_0$ is self-adjoint. In the course of estimate \eqref{B-26}, we have
\begin{equation} \label{C-2-0}
 \langle  -\mathcal{L}_0 f, f \rangle =  \frac{\sigma}{m^2} \int_{\bbr^2} M \Big| \partial_\omega \Big(  \frac{f}{\sqrt{M}} \Big) \Big|^2  d\omega d\nu.
\end{equation}

\vspace{0.2cm}

\noindent (ii)~It follows from the relation \eqref{C-2-0} that ${\mathcal L}_0 f = 0$ implies
\[ \partial_\omega \Big(  \frac{f}{\sqrt{M}} \Big) = 0, \quad \mbox{i.e.,} \quad  \frac{f}{\sqrt{M}} = h(\nu).  \]
Therefore, we have
\[  f \in \mbox{Ker} {\mathcal L}_0  \quad \Longleftrightarrow \quad f \in \big\{ h(\nu)\sqrt{M}  ~\big|~ h(\nu)\sqrt{M} \in L_{\omega, \nu}^{2}\big\}. \]
\end{proof}

\subsection{Projection operators and coercivity estimate} \label{sec:3.3}
In this part, we provide the coercivity estimate of the linear operator ${\mathcal L}_0$. For this, we introduce a projection operator $\bbp_0$:
\begin{equation} \label{C-3}
\chi_0 :=  \sqrt{2\pi} \sqrt{M}, \quad f_0 :=  \langle \chi_0, f \rangle, \quad \bbp_0 f := f_0 \chi_0,
\end{equation}
where $\langle \cdot, \cdot\rangle$ is the inner product in $L^2_{\omega,\nu}(\mathbb{R}^2)$. 

On the other hand, since the estimates for ${\mathcal L}_1$ and ${\mathcal N}$ give bad terms, we introduce another projection operator $\bbp_1$ as follows:
\begin{equation} \label{C-3-1}
\chi_1 :=\sqrt{\frac{2\pi m}{\sigma}}(\omega - \nu) \sqrt{M}, \quad  f_1 :=  \langle \chi_1, f \rangle, \quad \mbox{and} \quad \bbp_1 f := f_1 \chi_1.
\end{equation}
Note that $\chi_i$ is normalized to make $\langle \chi_i, \chi_i \rangle = 1,~i= 0, 1$. For notational simplicity, we set
\[ \bbp := \bbp_0 + \bbp_1, \quad \mbox{i.e.,} \quad \bbp f = f_0 \chi_0 + f_1 \chi_1. \]
In the following lemma, we provide several properties of the linear operators $\mathcal{L}_0$ and $\mathcal{L}_1$.

\begin{lemma}\label{L3.2}
The following assertions hold.
\begin{itemize}
\item[(i)]~The linear operators ${\mathcal L}_i,~\bbp_i,~i = 0, 1$ satisfy the following properties:
\begin{align*}
\begin{aligned}
& {\mathcal L}_0 \bbp_0 f =  \bbp_0 {\mathcal L}_0 f = 0, \quad
{\mathcal L}_0 \bbp_1 f = \bbp_1 {\mathcal L}_0 f = -\frac1m f_1 \chi_1, \\
& {\mathcal L}_0 \bbp f = \bbp {\mathcal L}_0 f = -\frac1m f_1 \chi_1, \quad  (\bbi - \bbp) {\mathcal L}_1 f  = 0,
\end{aligned}
\end{align*}
where $\langle \cdot, \cdot\rangle$ denotes the inner product in the Hilbert space $L^2_{\omega,\nu}(\mathbb{R}^2)$.

\vspace{0.2cm}

\item[(ii)] There exists a positive constant $\lambda_0 > 0$ such that the coercivity estimate holds:
\begin{align*}
\begin{aligned}
& \langle -{\mathcal L}_0 f, f \rangle \geq \frac{\lambda_0}{m} \| (\bbi - \bbp_0) f \|_{L_{\mu}^2}^2 , \\
& \langle -{\mathcal L}_0 f, f \rangle \geq \frac{\lambda_0}{m} \| (\bbi - \bbp) f \|_{L_{\mu}^2}^2 + \frac1m  |f_1|^2.
\end{aligned}
\end{align*}
\end{itemize}
\end{lemma}

\begin{proof}
\noindent (i)~We use the explicit form  ${\mathcal L}_0 f = (\sigma/m^2)(1/\sqrt M) \partial_{\omega} \big( M \partial_{\omega}(f/\sqrt M) \big)$ and the values of $\chi_0, \chi_1$ in \eqref{C-3} and \eqref{C-3-1} to derive
\begin{align}
\begin{aligned} \label{B-28}
{\mathcal L}_0 \chi_0   &= \frac{\sigma}{m^2} \frac{1}{\sqrt{M}} \partial_{\omega} \big( M \partial_{\omega}(\frac{1}{\sqrt{M}} \chi_0) \big) = \frac{\sigma}{m^2} \frac{1}{\sqrt{M}} \partial_{\omega} \big( M \partial_{\omega}\sqrt{2\pi} \big) = 0, \\
 {\mathcal L}_0  \chi_1 &= \frac{\sigma}{m^2} \frac{1}{\sqrt{M}} \partial_{\omega} \left( M \partial_{\omega}\left(\sqrt{\frac{2 \pi m}{\sigma}} (\omega - \nu) \right) \right) =
 \frac{\sigma}{m^2} \sqrt{\frac{2\pi m}{\sigma}} \frac{1}{\sqrt{M}} \partial_\omega M \\
  &= - \frac{1}{m} \sqrt{\frac{2\pi m}{\sigma}} (\omega - \nu) \sqrt{M}= - \frac{1}{m} \chi_1.
\end{aligned}
\end{align}
Now by definition of $\bbp_i, i = 0, 1$ and ${\mathcal L}_0$, we have
\begin{align}
\begin{aligned} \label{B-27}
& {\mathcal L}_0 \bbp_0 f = {\mathcal L}_0\big(f_0 \chi_0 \big) = f_0 {\mathcal L}_0 \chi_0= 0, \\
& {\mathcal L}_0 \bbp_1 f = {\mathcal L}_0\big(f_1 \chi_1 \big) = f_1 {\mathcal L}_0 \chi_1 =  - \frac{1}{m} f_1 \chi_1, \quad \mbox{and} \\
& {\mathcal L}_0 \bbp f  = {\mathcal L}_0 \bbp_0 f + {\mathcal L}_0 \bbp_1 f  = - \frac{1}{m} f_1 \chi_1.
\end{aligned}
\end{align}
On the other hand,  we use the self-adjoint property of ${\mathcal L}_0$ and \eqref{B-28} to obtain
\begin{align}
\begin{aligned} \label{B-30}
& \bbp_0 {\mathcal L}_0 f  = \langle {\mathcal L}_0 f, \chi_0 \rangle \chi_0 = \langle f, {\mathcal L}_0\chi_0 \rangle \chi_0 = 0, \\
& \bbp_1 {\mathcal L}_0 f  = \langle {\mathcal L}_0 f, \chi_1 \rangle \chi_1 = \langle f, {\mathcal L}_0\chi_1 \rangle \chi_1 = -\frac1m f_1 \chi_1, \\
&\bbp {\mathcal L}_0 f =  \bbp_0 {\mathcal L}_0 f  + \bbp_1 {\mathcal L}_0 f = \langle {\mathcal L}_0 f, \chi_0 \rangle \chi_0 + \langle {\mathcal L}_0 f, \chi_1 \rangle \chi_1 
= -\frac1m f_1 \chi_1.
\end{aligned}
\end{align}
The first three relations in $(i)$ follow from \eqref{B-27} and \eqref{B-30}. For the last relation in $(i)$, we use $(\omega - \nu) \sqrt{M} = \sqrt{\sigma/(2\pi m)} \chi_1$ to derive
\[   {\mathcal L}_1 f = \frac{\kappa}{\sigma} (\omega - \nu) \sqrt{M}\mathcal{S}[\sqrt{M} f](\theta,t) = \frac{\kappa}{\sigma} \sqrt{\frac{\sigma}{2\pi m}} \mathcal{S}[\sqrt{M} f] \chi_1.  \]
This yields
\[
(\bbi - \bbp)  {\mathcal L}_1 f  = \frac{\kappa}{\sigma} \sqrt{\frac{\sigma}{2\pi m}} \mathcal{S}[\sqrt{M} f] (\bbi -\bbp)  \chi_1 = 0,
\]
which yields the desired estimate.

\vspace{0.5cm}

\noindent $(ii)$ We use $f =  \bbp_0 f +  (\bbi - \bbp_0) f$, self-adjoint property of ${{\mathcal L}_0}$, and the estimate $(i)$ to get
\begin{align}
\begin{aligned} \label{B-31-1}
\langle -{\mathcal L}_0 f, f \rangle &= \langle -{\mathcal L}_0 (\bbi - \bbp_0) f, f \rangle + \langle -{\mathcal L}_0 \bbp_0 f, f \rangle  \\
& =  \langle -{\mathcal L}_0 (\bbi - \bbp_0) f, (\bbi - \bbp_0) f \rangle + \langle -{\mathcal L}_0 (\bbi - \bbp_0) f, \bbp_0 f \rangle \\
& = \langle -{\mathcal L}_0 (\bbi - \bbp_0) f, (\bbi - \bbp_0) f \rangle + \langle - (\bbi - \bbp_0) f, {\mathcal L}_0 \bbp_0 f \rangle  \\
& =\langle -{\mathcal L}_0 (\bbi - \bbp_0) f, (\bbi - \bbp_0) f \rangle.
\end{aligned}
\end{align}
We use the relation $\partial_\omega M = -(m/\sigma) (\omega - \nu) M$ and Lemma \ref{L3.1} to get
\begin{align}
\begin{aligned} \label{B-32-1}
 \langle -{\mathcal L}_0 (\bbi - \bbp_0) f, (\bbi - \bbp_0) f \rangle &= \frac{\sigma}{m^2} \int_{\mathbb{R}^2} \Big| \partial_{\omega}\Big( \frac{1}{\sqrt{M}} (\bbi - \bbp_0) f \Big) \Big|^2 M\, d\omega d\nu \\
&= \frac{\sigma}{m^2} \int_{\mathbb{R}^2} \Big| \frac{m}{2\sigma} \frac{1}{\sqrt{M}} (\omega - \nu) (\bbi - \bbp_0) f + \frac{1}{\sqrt{M}} \partial_{\omega} \big( (\bbi - \bbp_0) f \big) \Big|^2 M \, d\omega d\nu.
\end{aligned}
\end{align}
Next, we set
\[ {\tilde \omega} := \sqrt{m/\sigma} (\omega - \nu), \quad {\tilde M}:= \frac{1}{\sqrt{2\pi}} e^{-\frac{{\tilde \omega}^2}{2}}, \quad {\tilde L} f :=
\frac{1}{\sqrt{\tilde M}} \partial_{\tilde \omega} \Big[  \tilde M  \partial_{\tilde \omega} \Big(\frac{1}{\sqrt{\tilde M}} f \Big) \Big]. \]
Then, it follows from the estimate given in \cite[Section 2.1]{C-D-M} that there exits a positive constant $\lambda_0$ such that
\[ \big( -{\tilde L} g, g \big) \geq \lambda_0  \|g\|_{L_{\tilde \nu}}^2, \quad \text{for any} ~ g \in \left(\mbox{span}\{\sqrt{\tilde M}\}\right)^{\perp},   \]
where we used $(\bbi - \bbp_0)g = g$ for $g \in \left(\mbox{span}\{\sqrt{\tilde M}\}\right)^{\perp}$. Here $(\cdot, \cdot)$ is $L_{\tilde{\omega}}^2$-inner product and norm $\|\cdot\|_{L_{\tilde \nu}}$ is defined as follows.
\[  \|g\|^2_{L_{\tilde \nu}} := \int_{\bbr} \Big( (1 + |\tilde \omega|^2) |g|^2 + |\partial_{\tilde \omega} g|^2 \Big) \,d{\tilde \omega}.    \]
On the other hand, note that $\left(\mbox{span}\{\sqrt{\tilde M}\}\right)^{\perp} = \left(\mbox{span}\{\sqrt{M}\}\right)^{\perp}$. Thus, we have
\begin{align}
\begin{aligned} \label{B-33-1}
 \langle -{\mathcal L}_0 (\bbi - \bbp_0) f,  (\bbi - \bbp_0) f \rangle
& = \frac{1}{m} \sqrt{\frac{\sigma}{m}} \int_{\mathbb{R}^2} \Big| \frac12 \tilde{\omega} (\bbi - \bbp_0) f + \partial_{\tilde{\omega}} (\bbi - \bbp_0) f \Big|^2 d\tilde{\omega}  d\nu\\
& =  \frac{1}{m} \sqrt{\frac{\sigma}{m}} \int_{\mathbb{R}} \big( -{\tilde L} (\bbi - \bbp_0)f, (\bbi - \bbp_0)f \big) d\nu \\
&  \geq \frac{\lambda_0}{m} \sqrt{\frac{\sigma}{m}} \int_{\mathbb{R}} \|(\bbi - \bbp_0)f\|_{L_{\tilde{\nu}}^{2}}^2 d\nu \\
& = \frac{\lambda_0}{m} \sqrt{\frac{\sigma}{m}} \int_{\mathbb{R}^2}   \left( (1 + \tilde{\omega}^2) |  (\bbi - \bbp_0) f|^2 + |\partial_{\tilde{\omega}} (\bbi - \bbp_0) f|^2 \right) d\tilde{\omega}  d\nu \\
& =  \frac{\lambda_0}{m} \int_{\mathbb{R}^2} \left( (1 + \frac{m}{\sigma}(\omega - \nu)^2) | (\bbi - \bbp_0) f|^2 + \frac{\sigma}{m} |\partial_{\omega}  (\bbi - \bbp_0) f|^2 \right) d\omega d \nu \\
& =  \frac{\lambda_0}{m} \| (\bbi - \bbp_0) f\|_{L_{\mu}^2}^2.
\end{aligned}
\end{align}
Finally, we combine all estimates in \eqref{B-31-1}, \eqref{B-32-1} and \eqref{B-33-1} to obtain
\[
\langle -{\mathcal L}_0 f, f \rangle \geq  \frac{\lambda_0}{m}  \| (\bbi - \bbp_0) f\|_{L_{\mu}^2}^2.
\]
To prove the second assertion in $(ii)$, we use $f =  \bbp f +  (\bbi - \bbp) f$ to obtain
\begin{align}
\begin{aligned} \label{B-31}
\langle -{\mathcal L}_0 f, f \rangle &= \langle -{\mathcal L}_0 (\bbi - \bbp) f, f \rangle + \langle -{\mathcal L}_0 \bbp f, f \rangle  \\
& =  \langle -{\mathcal L}_0 (\bbi - \bbp) f, (\bbi - \bbp) f \rangle + \langle -{\mathcal L}_0 (\bbi - \bbp) f, \bbp f \rangle + \langle -{\mathcal L}_0 \bbp f, f \rangle \\
& =: {\mathcal I}_{11}  + {\mathcal I}_{12}  + {\mathcal I}_{13}.
\end{aligned}
\end{align}
Note that the terms ${\mathcal I}_{1i},~i =1, 2,3$ can be treated as follows.

\vspace{0.2cm}

\noindent  $\bullet$ (Estimate of $ {\mathcal I}_{11}$): Similar to estimate \eqref{B-33-1}, we have
\begin{align}
\begin{aligned} \label{B-32}
{\mathcal I}_{11}
&= \frac{\sigma}{m^2} \int_{\mathbb{R}^2} \Big| \partial_{\omega}\Big( \frac{1}{\sqrt{M}} (\bbi - \bbp) f \Big) \Big|^2 M\, d\omega d\nu \\
&= \frac{\sigma}{m^2} \int_{\mathbb{R}^2} \Big| \frac{m}{2\sigma}(\omega - \nu) \frac{1}{\sqrt{M}} (\bbi - \bbp) f + \frac{1}{\sqrt{M}} \partial_{\omega} \big( (\bbi - \bbp) f \big) \Big|^2 M \, d\omega d\nu \\
&= \frac{1}{m} \sqrt{\frac{\sigma}{m}} \int_{\mathbb{R}^2} \Big| \frac12 \tilde{\omega} (\bbi - \bbp) f + \partial_{\tilde{\omega}} (\bbi - \bbp) f \Big|^2 d\tilde{\omega} d\nu\\
&\geq \frac{\lambda_0}{m} \sqrt{\frac{\sigma}{m}} \int_{\mathbb{R}^2} \left( (1 + \tilde{\omega}^2) |  (\bbi - \bbp) f|^2 + |\partial_{\tilde{\omega}} (\bbi - \bbp) f|^2 \right) d\tilde{\omega}d\nu \\
&=  \frac{\lambda_0}{m} \int_{\mathbb{R}^2} \left( (1 + \frac{m}{\sigma}(\omega - \nu)^2) | (\bbi - \bbp) f|^2 + \frac{\sigma}{m} |\partial_{\omega}  (\bbi - \bbp) f|^2 \right) d\omega d \nu \\
&=  \frac{\lambda_0}{m} \| (\bbi - \bbp) f\|_{L_{\mu}^2}^2.
\end{aligned}
\end{align}

\noindent $\bullet$ (Estimate of $ {\mathcal I}_{12}$): We use the self-adjoint property of ${{\mathcal L}_0}$ and the first relation in $(i)$ to obtain
\begin{equation} \label{C-32-1}
{\mathcal I}_{12} = \langle -{\mathcal L}_0 (\bbi - \bbp) f, \bbp f \rangle = \langle -(\bbi - \bbp) f, {\mathcal L}_0 \bbp f \rangle  = \frac{1}{m} f_1 \langle (\bbi - \bbp) f, \chi_1 \rangle = 0.
\end{equation}

\vspace{0.2cm}

\noindent $\bullet$ (Estimate of $ {\mathcal I}_{13}$): We use the first relation in $(i)$ again to find
\begin{equation} \label{C-32-2}
{\mathcal I}_{13} = \langle -{\mathcal L}_0 \bbp f, f \rangle = \frac1m  f_1  \langle \chi_1, f \rangle = \frac1m |f_1|^2.
\end{equation}
Finally, we combine all estimates in \eqref{B-31}, \eqref{B-32}, \eqref{C-32-1}, and \eqref{C-32-2} to obtain
\[
\langle -{\mathcal L}_0 f, f \rangle \geq  \frac{\lambda_0}{m}  \| (\bbi - \bbp) f\|_{L_{\mu}^2}^2 + \frac1m |f_1|^2.
\]
\end{proof}

\section{Unique local-in-time solvability of strong solutions} \label{sec:4}
\setcounter{equation}{0}
In this section, we present the local-in-time existence and uniqueness of strong solutions to the Cauchy problem \eqref{A-3}. For this, we first linearize the equation \eqref{A-3}, and study the global existence and uniqueness of solutions to that system. Then we construct approximated solutions $\{f^m\}_{m\in \mathbb{N}}$ in $H^s$-space and show that they are Cauchy sequence in $L^2$-space. We finally show that the limiting function is the solution to the equation \eqref{A-3} satisfying the desired regularity. Before we present detailed discussion, we first state our main result on the local-in-time solution below.

\begin{theorem}\label{T4.1} Let $s \geq 1$. For any given constants $\eta_0 < \eta$, there exists a positive constant $T_0= T_0(\eta_0, \eta)$ such that if $\|f^{in}\|_{H^s(\bbt \times \bbr^2)} \leq \eta_0$ and  $M + \sqrt{M}f^{in} \geq 0$, then there exists a unique strong solution $f$ to the Cauchy problem \eqref{A-3} such that $M + \sqrt{M} f \geq 0$ and $f \in \mathcal{C}([0,T_0]; H^s(\mathbb{T} \times \mathbb{R}^2))$. In particular, we have
\[
\sup_{0 \leq t \leq T_0}\|f(t)\|_{H^s}  \leq \eta.
\]
\end{theorem}

\subsection{Solvability of linear equation} For $T \in (0, \infty),$ suppose that $g \in \mathcal{C}([0,T];H^s(\mathbb{T} \times \mathbb{R}^2))$ is given. Then, we consider the Cauchy problem to the following linear equation:
\begin{equation}
\begin{cases} \label{D-1}
\displaystyle \partial_t f + \omega \partial_{\theta} f    =
\displaystyle  {\mathcal L}_0 f + {\mathcal L}_1 f+ \mathcal{N}(g, f), \quad (\theta, \omega, \nu) \in \bbt \times \bbr^2,~t > 0, \\[2mm]
\displaystyle \mathcal{N}(g, f) = \frac{\kappa}{2\sigma}  \mathcal {S}[\sqrt{M}g] (\omega - \nu) f - \frac{\kappa}{m} \mathcal{S}[\sqrt{M}g] \partial_{\omega}f, \\
f(\theta,\omega,\nu,0) =: f^{in}(\theta,\omega,\nu).
\end{cases}
\end{equation}
For a constant $T>0$, the existence of the unique solution $f \in \mathcal{C}([0,T]; H^s(\mathbb{T} \times \mathbb{R}^2))$ to the above linearized equation \eqref{D-1} can be easily obtained by the standard linear solvability theory.
In the lemma below, we provide some useful estimates for $f$.
\begin{lemma}\label{L4.1}
For any given $f \in \mathcal{C}([0,T];H^s(\mathbb{T} \times \mathbb{R}^2))$ with $s \geq 1$, let $f_0$ and $f_1$ be defined in \eqref{C-3} and \eqref{C-3-1} respectively. Then, we have 
\begin{eqnarray*}
&(i)& \|\partial_\theta^k  f_0 \|_{L^2} +   \| \partial_\theta^k  f_1    \|_{L^2}  \lesssim \|\partial_\theta^k f\|_{L^2}, \quad \|\bbp\partial_\theta^k f\|_{\mu} \lesssim \|\partial_\theta^k f\|_{L^2}. \\
&(ii)& \|\partial_{\theta}^k \mathcal{S}[\sqrt{M}f]  \|_{L^{\infty}(\bbt \times \bbr^2)} \lesssim \|f\|_{L^2} \quad 0 \leq k \leq s.
\end{eqnarray*}
\end{lemma}
\begin{proof}  (i)~We first provide the estimate of $\|\partial_\theta^k  f_0 \|_{L^2}$.  For $0 \leq k \leq s$, we use \eqref{C-3}  and H\"older inequality to obtain
\begin{align}
\begin{aligned} \label{D-3}
\|\partial_\theta^k  f_0 \|_{L^2}^2 &= \int_\mathbb{T} |\partial_\theta^k  f_0|^2\,d\theta =  \int_\mathbb{T}  \Big| \int_{\mathbb{R}^2} \sqrt{2\pi}  \sqrt{M} \partial_\theta^kf d\omega d\nu \Big|^2 d\theta  \\
&\leq  \int_{\bbt}  \Big( \int_{\mathbb{R}^2}  2\pi M  d\omega d\nu \Big)  \Big( \int_{\mathbb{R}^2} |\partial_\theta^k f|^2 d\omega d\nu \Big) d\theta = \|\partial_\theta^k f\|^2_{L^2}.
\end{aligned}
\end{align}
Similarly, we have
\begin{equation} \label{D-4}
\|\partial_\theta^k  f_1 \|_{L^2} \lesssim \|\partial_\theta^k f\|^2_{L^2}.
\end{equation}
We combine \eqref{D-3} and \eqref{D-4} to derive the desired first estimate. \newline

Since $ \bbp$ and $\partial_\theta^k$ are commutative, we have
\[ \bbp \partial_\theta^k f = \partial_\theta^k  \bbp f  = \Big( \partial_\theta^k f_0 \Big) \chi_0 + \Big( \partial_\theta^k f_1 \Big) \chi_1,  \]
and this yields
\[ \|\bbp \partial_\theta^k f\|_\mu \leq \| \partial_\theta^k f_0 \chi_0  \|_\mu + \|  \partial_\theta^k f_1 \chi_1 \|_\mu.
\]
Next, we show that $\| \partial_\theta^k f_0 \chi_0  \|_\mu \lesssim \|\partial_\theta^k f\|_{L^2}$ for $0 \leq k \leq s$. For this, we use the integrability of the weighted Maxwellian to obtain
\begin{align}
\begin{aligned} \label{D-5}
\| \partial_\theta^k f_0 \chi_0 \|_\mu^2 &\leq 2\pi\int_{\mathbb{T} \times \mathbb{R}^2} \left(\alpha M |\partial_\theta^k f_0 |^2 + \beta |\partial_\omega \sqrt M|^2 |\partial_\theta^k f_0 |^2\right)d\theta d\omega d\nu \\
&\lesssim \|\partial_\theta^k f_0 \|_{L^2}^2 \lesssim \|\partial_\theta^k f\|_{L^2},
\end{aligned}
\end{align}
where $\alpha = 1 + (m/\sigma)(\omega - \nu)^2$ and $\beta = \sigma/m$. Similarly, we have
\begin{equation} \label{D-6}
  \|  \partial_\theta^k f_1 \chi_1 \|_\mu \lesssim \|\partial_\theta^k f\|_{L^2}.
\end{equation}
Finally, we combine \eqref{D-5} and \eqref{D-6} to derive the second estimate in $(i)$. \newline

\noindent $(ii)$ For any $0 \leq k \leq s$, we use H\"older inequality,
\[
\int_{\bbr^2} M \,d\omega d\nu = \frac{1}{2\pi} \quad \mbox{and} \quad \int_{\bbt} |\partial_{\theta}^k \sin(\theta_*- \theta)|^2 \,d\theta_* = \pi
\]
to get
\begin{align*}
\begin{aligned}
\Big| \partial_{\theta}^k \mathcal{S}[\sqrt{M}f] \Big| &= \Big| \int_{\bbt \times \bbr^2} \partial_{\theta}^k \sin(\theta_*- \theta) \sqrt{M} f \,d\theta_* d\omega_* d\nu_* \Big|  \\
& \leq \Big( \int_{\bbt \times \bbr^2} |\partial_{\theta}^k \sin(\theta_*- \theta)|^2 M \,d\theta_* d\omega_* d\nu_* \Big)^{\frac{1}{2}} \Big( \int_{\bbt \times \bbr^2} |f|^2 \,d\theta_* d\omega_* d\nu_* \Big)^{\frac{1}{2}}  \lesssim \|f\|_{L^2}.
\end{aligned}
\end{align*}
\end{proof}
With the aid of Lemma \ref{L4.1}, we next present the uniform bound estimate for the linearized equation \eqref{D-1}.
\begin{lemma}\label{L4.2}
For given initial data $f^{in} \in H^s(\mathbb{T} \times \mathbb{R}^2)$ with $s \geq 1$, there exist positive constants $T_0$ and $\eta_0$ such that if
\[
M + \sqrt M f^{in} \geq 0, \quad \|f^{in}\|_{H^s} \leq \eta_0 \quad  \mbox{and} \quad \sup_{0 \leq t \leq T_0}\|g(t)\|_{H^s} \leq \eta,
\]
then we have
\[   f \in \mathcal{C}([0,T_0]; H^s(\mathbb{T} \times \mathbb{R}^2)), \quad M + \sqrt M f \geq 0, \quad \sup_{0 \leq t \leq T_0}\|f(t)\|_{H^s} \leq \eta. \]
\end{lemma}
\begin{proof} For $0 \leq k \leq s$, we apply $\partial_\theta^k$ to $\eqref{D-1}$ to obtain
\[
\partial_t \partial_\theta^k f + \omega \partial_\theta^{k+1}f  -  \partial_\theta^k {\mathcal L}_0 f = \partial_\theta^k {\mathcal{L}_1}f +   \partial_\theta^k \mathcal{N}(g, f).
\]
Then we use the commutative property of $\mathcal{L}_0$ and $\partial_{\theta}^{k}$ to yield
\begin{equation} \label{D-7}
\frac12\frac{d}{dt}\|\partial_\theta^k f\|_{L^2}^2 + \int_\mathbb{T} \langle -{\mathcal L}_0 \partial_\theta^k f,\partial_\theta^k f \rangle \,d\theta  =   \int_\mathbb{T} \langle \partial_\theta^k \mathcal{L}_1 f,\partial_\theta^k f \rangle \,d\theta + \int_\mathbb{T} \langle \partial_\theta^k \mathcal{N}(g, f),\partial_\theta^k f \rangle \,d\theta.
\end{equation}

\noindent $\bullet$ (Estimate A): It follows from Lemma \ref{L3.1} $(ii)$ that
\begin{equation} \label{D-12}
\int_\mathbb{T} \langle -{\mathcal L}_0 \partial_\theta^k f,\partial_\theta^k f \rangle \,d\theta  \geq \frac{\lambda_0}{m} \| (\bbi - \bbp) \partial_\theta^k f\|_{\mu}^2 + \frac1m \|\partial_\theta^k f_1 \|_{L^2}^2.
\end{equation}

\noindent $\bullet$ (Estimate B):  It is clear that
\begin{align*}
\begin{aligned}
\int_\mathbb{T} \langle \partial_\theta^k \mathcal{L}_1 f,\partial_\theta^k f \rangle \,d\theta &=\frac \kappa\sigma \int_{\mathbb{T} \times
\bbr^2}  \partial_\theta^k \mathcal{S}[\sqrt{M} f] (\omega - \nu)\sqrt{M}\partial_\theta^k f \,d\theta  d\omega d\nu.
\end{aligned}
\end{align*}
Thus, we use H\"older inequality, Lemma \ref{L4.1} $(ii)$, and
\[
\int_{\bbt \times \bbr^2} (\omega - \nu)^2 M \,d\theta d\omega d\nu = \frac{\sigma}{m}
\]
to have
\begin{align}
\begin{aligned} \label{D-10}
&\left| \int_\mathbb{T} \langle \partial_\theta^k \mathcal{L}_1 f,\partial_\theta^k f \rangle \,d\theta \right| \cr
&\quad \leq \frac \kappa\sigma \| \partial_\theta^k \mathcal{S}[\sqrt{M} f]\|_{L^{\infty}} \Big( \int_{\mathbb{T} \times \bbr^2} (\omega - \nu)^2 M \,d\theta d\omega d\nu  \Big)^{\frac{1}{2}} \Big( \int_{\mathbb{T}\times \bbr^2} |\partial_\theta^k f|^2 \,d\theta d\omega d\nu \Big)^{\frac{1}{2}} \\
 &\quad = \frac{\kappa}{\sqrt{m\sigma}}\|\partial_\theta^k \mathcal{S}[\sqrt{M} f]\|_{L^\infty}\|\partial_\theta^k f\|_{L^2} \lesssim \frac{\kappa}{\sqrt{m\sigma}} \|f\|_{L^2}\|\partial_\theta^k f\|_{L^2}.
\end{aligned}
\end{align}
\noindent $\bullet$ (Estimate C):
We use $\eqref{D-1}_2$ to obtain
\begin{align*}
\begin{aligned}
\int_\mathbb{T} \langle \partial_\theta^k \mathcal{N}(g, f),\partial_\theta^k f \rangle \,d\theta&= \frac{\kappa}{2\sigma} \sum_{0 \leq \ell \leq k} \binom{k}{\ell} \int_{\bbt \times \bbr^2} \partial_{\theta}^{k - \ell}\mathcal {S}[\sqrt{M} g] (\omega - \nu) \partial_{\theta}^{\ell}f  \partial_{\theta}^{k}f \,d\theta d\omega d\nu \\
&- \frac{\kappa}{m} \sum_{0 \leq \ell \leq k} \binom{k}{\ell} \int_{\mathbb{T} \times \bbr^2}  \partial_\theta^{k-\ell}\mathcal{S}[\sqrt{M} g]\partial_\omega \partial_{\theta}^{\ell}f \partial_\theta^k f \,d\theta d\omega d\nu \\
&=: {\mathcal I}_{21} + {\mathcal I}_{22}.
\end{aligned}
\end{align*}
Next we will estimate the ${\mathcal I}_{2i}, i =1,2$ one by one. \newline

\noindent $\diamond$ (Estimate on ${\mathcal I}_{21}$): We use H\"older inequality and Lemma \ref{L4.1} $(ii)$ to obtain
 \begin{align}
\begin{aligned} \label{D-8}
|{\mathcal I}_{21}| &\leq \frac{\kappa}{2\sigma} \sum_{0 \leq \ell \leq k} \binom{k}{\ell} \|\partial_\theta^{k - \ell} \mathcal{S}[\sqrt{M} g]\|_{L^\infty} \\
& \quad \times \Big( \int_{\mathbb{T}\times \bbr^2} \frac{\sigma}{m} \cdot \frac{m}{\sigma}(\omega - \nu)^2 |\partial_\theta^{\ell} f|^2 d\theta d\omega d\nu \Big)^{\frac{1}{2}}  \Big( \int_{\mathbb{T}\times \bbr^2} | \partial_\theta^k f|^2 d\theta d\omega d\nu \Big)^{\frac{1}{2}}  \cr
&\lesssim \frac{\kappa}{\sqrt{m\sigma}} \sum_{0 \leq \ell \leq k} \|\partial_\theta^\ell \mathcal{S}[\sqrt M g]\|_{L^\infty} \|\partial_\theta^{\ell}f\|_{\mu} \|\partial_\theta^k f\|_{L^2} \\
&\lesssim  \frac{\kappa}{\sqrt{m\sigma}} \|g\|_{L^2} \|\partial_\theta^k f\|_{L^2}  \sum_{0 \leq \ell \leq k} \|\partial_\theta^{ \ell}f\|_{\mu}.
\end{aligned}
\end{align}
\noindent $\diamond$ (Estimate on ${\mathcal I}_{22}$): Similar to the estimate of ${\mathcal I}_{21}$, we have
\begin{align}
\begin{aligned} \label{D-11}
|{\mathcal I}_{22}| &\leq  \frac {\kappa}{m}  \sum_{0 \leq \ell \leq k} \binom{k}{\ell} \|\partial_\theta^{k - \ell} \mathcal{S}[\sqrt{M} g]\|_{L^\infty} \Big( \int_{\mathbb{T}\times \bbr^2} \frac{m}{\sigma} \cdot \frac{\sigma}{m}|\partial_\theta^{\ell} \partial_{\omega} f|^2 \,d\theta d\omega d\nu \Big)^{\frac{1}{2}} \\
&\times \Big( \int_{\mathbb{T}\times \bbr^2} | \partial_\theta^k f|^2 \,d\theta d\omega d\nu \Big)^{\frac{1}{2}}  \\
&\lesssim \frac{\kappa}{\sqrt{m\sigma}} \sum_{0 \leq \ell \leq k} \|\partial_\theta^\ell \mathcal{S}[\sqrt M g]\|_{L^\infty} \|\partial_\theta^{\ell}f\|_{\mu} \|\partial_\theta^k f\|_{L^2} \\
& \lesssim  \frac{\kappa}{\sqrt{m\sigma}} \|g\|_{L^2} \|\partial_\theta^k f\|_{L^2}  \sum_{0 \leq \ell \leq k} \|\partial_\theta^{ \ell}f\|_{\mu}.
\end{aligned}
\end{align}
We combine estimates \eqref{D-8} and \eqref{D-11} to have
\begin{align}
\begin{aligned} \label{D-11-1}
|{\mathcal I}_{21}| + |{\mathcal I}_{22}|  \lesssim  \frac{\kappa}{\sqrt{m\sigma}} \|g\|_{L^2} \|\partial_\theta^k f\|_{L^2}  \sum_{0 \leq \ell \leq k} \|\partial_\theta^{ \ell}f\|_{\mu}.
\end{aligned}
\end{align}
Finally,  we substitute estimates  \eqref{D-12}, \eqref{D-10}, and \eqref{D-11-1} into \eqref{D-7} obtain
\begin{align}\label{D-13}
\begin{aligned}
&\frac12\frac{d}{dt}\|\partial_\theta^k f\|_{L^2}^2 + \frac{\lambda_0}{m} \| (\bbi - \bbp) \partial_\theta^k f\|_{\mu}^2 \cr
&\hspace{1cm} \lesssim \frac{\kappa}{\sqrt{m\sigma}} \Big( \|f\|_{L^2} \|\partial_\theta^k f\|_{L^2}
+ \|g\|_{L^2} \|\partial_\theta^k f\|_{L^2}  \sum_{0 \leq \ell \leq k} \|\partial_\theta^{ \ell}f\|_{\mu} \Big).
\end{aligned}
\end{align}
On the other other hand, it follows from Lemma \ref{L4.1} that
\begin{equation} \label{D-14}
\|\partial_\theta^k f\|_\mu \leq \| (\bbi - \bbp)\partial_\theta^k f\|_\mu + \|\bbp \partial_\theta^k f\|_\mu \leq \| (\bbi - \bbp) \partial_\theta^k f\|_\mu + C\|\partial_\theta^k f\|_{L^2},
\end{equation}
for some $C>0$. Now, we combine \eqref{D-13} and \eqref{D-14} to obtain
\begin{align}\label{DD-1}
\begin{aligned}
& \frac12\frac{d}{dt}\|\partial_\theta^k f\|_{L^2}^2 + \frac{\lambda_0}{m} \|\partial_\theta^k f\|_{\mu}^2 \\
& \hspace{0.5cm} \lesssim \frac{\kappa}{\sqrt{m\sigma}} \Big( \|f\|_{L^2} \|\partial_\theta^k f\|_{L^2}
+ \|g\|_{L^2} \|\partial_\theta^k f\|_{L^2}  \sum_{0 \leq \ell \leq k} \|\partial_\theta^{ \ell}f\|_{\mu} \Big) + \|\partial_\theta^k f\|_{L^2}^2.
\end{aligned}
\end{align}
We then sum \eqref{DD-1} over $k$ to get
\begin{align*}
\begin{aligned}
& \frac12\frac{d}{dt} \| f\|_{H_{\theta}^{k}}^2 + \frac{\lambda_0}{m} \sum_{\ell =0}^{k} \|\partial_\theta^k f\|_{\mu}^2  \leq  C \|f\|_{H_{\theta}^{k}}^2.
\end{aligned}
\end{align*}
Using the similar strategy for the $\partial_\theta^\alpha \partial_\omega^\beta$ derivatives with $0 \leq \alpha + \beta \leq s$, we also find
\[
\frac{d}{dt}\|f\|_{H^s}^2 + \frac{\lambda_0}{m} \sum_{0 \leq \alpha + \beta \leq s}\|\partial_\theta^\alpha \partial_\omega^\beta f\|_\mu^2 \leq C \|f\|_{H^s}^2.
\]
We integrate the above relation over the time interval $[0,T_0]$ to obtain
\begin{align*}
\begin{aligned}
\|f(t)\|_{H^s}  \leq e^{CT_0} \|f^{in}\|_{H^s}.
\end{aligned}
\end{align*}
Finally, we choose the parameters $T_0$ such that
\[
\eta_0 e^{C T_0} < \eta
\]
to conclude
\[
\sup_{0 \leq t \leq T_0}\|f(t)\|_{H^s} \leq \eta_0 e^{C T_0} < \eta.
\]
\end{proof}

\subsection{Approximate solutions} In this subsection, we provide a sequence of approximate solutions $\{ f^{\ell} \}_{\ell=0}^{\infty}$ as follows. For $\ell = 0$,  we set 
\[ f^0(\theta,\omega,\nu,t) = 0 \quad \mbox{for} \quad (\theta,\omega,\nu,t) \in \mathbb{T} \times \mathbb{R}^2 \times \mathbb{R}_+.   \]
Suppose that  the $\ell$-th iterate $f^{\ell}$ is given. Then, we define approximated solution $f^{\ell + 1}$ as a solution to the linear equation:
\begin{equation}
\begin{cases} \label{D-16}
 \partial_t f^{\ell + 1} + \omega \partial_{\theta} f^{\ell + 1} =  {\mathcal L}_0 f^{\ell + 1} +  {\mathcal L}_1 f^{\ell + 1} + \mathcal{N}(f^{\ell}, f^{\ell + 1}),   \\
 f^{\ell}(\theta,\omega,\nu, 0) =: f^{in}(\theta,\omega,\nu), \quad (\theta,\omega,\nu) \in \mathbb{T} \times \mathbb{R}^2,
\end{cases}
\end{equation}
for $\ell \geq 0$. Here $\mathcal{N}(f^{\ell},f^{\ell + 1})$ is given by
\[
\mathcal{N}(f^{\ell},f^{\ell + 1}) = \frac{\kappa}{2\sigma} \mathcal {S}[\sqrt{M} f^{\ell}] (\omega - \nu) f^{\ell + 1} -  \frac{\kappa}{m} \mathcal{S}[\sqrt{M} f^{\ell}] \partial_{\omega}f^{\ell + 1} .
\]
Let $s \geq 1$ and ${\mathcal X}(s,T; \eta)$ be a solution space for $f$ defined by
\[
{\mathcal X}(s,T;\eta) := \left\{ f \in \mathcal{C}([0,T]; H^s(\mathbb{T}\times \mathbb{R}^2)) \,:\, M + \sqrt M f \geq 0 \quad \mbox{and} \quad \sup_{0 \leq t \leq T}\|f(t)\|_{H^s}  < \eta \right \}.
\]
As a direct consequence of Lemma \ref{L4.2}, we have the following proposition providing the uniform bound estimate of $f^{\ell}$ for each $\ell \geq 1$.

\begin{proposition}\label{P4.1}
For any given constants $\eta_0 < \eta$, there exists a positive constant $T_0= T_0(\eta_0, \eta)$ such that if the initial datum $f^{in}$ satisfies
\[ f^{in} \in H^s(\mathbb{T} \times \mathbb{R}^2), \quad \|f^{in}\|_{H^s} \leq \eta_0 \quad  \mbox{and} \quad M + \sqrt M f^{in} \geq 0, \]
then for each $\ell \geq 0$, $f^{\ell}$ is well-defined and $f^{\ell} \in {\mathcal X}(s,T_0; \eta)$.
\end{proposition}

In next lemma, we show that the sequence $\{ f^{\ell} \}$ in $\mathcal{C}([0,T_0]; L^2(\mathbb{T} \times \mathbb{R}^2))$ is Cauchy.

\begin{lemma}\label{L4.3}
Let $\{ f^{\ell} \}$ be a sequence of approximated solutions  with the initial data $f^{in}$ satisfying $\|f^{in}\|_{H^s} \leq \eta_0$. Then, the sequence $\{ f^{\ell} \}$ is a Cauchy sequence in $\mathcal{C}([0,T_0]; L^2(\mathbb{T} \times \mathbb{R}^2))$.
\end{lemma}

\begin{proof}
It follows from \eqref{D-16} that
$$\begin{aligned}
&\partial_t (f^{\ell + 1} - f^{\ell}) + \omega \partial_\theta(f^{\ell + 1} - f^{\ell})  \\
& \hspace{1cm} =   {\mathcal L}_0(f^{\ell + 1} - f^{\ell}) +  {\mathcal L}_1(f^{\ell + 1} - f^{\ell}) + \mathcal{N}(f^{\ell} - f^{\ell - 1}, f^{\ell + 1}) + \mathcal{N}(f^{\ell - 1}, f^{\ell + 1} - f^{\ell}).
\end{aligned}$$
We use the similar argument as in Lemma \ref{L4.2} to yield
\begin{align}
\begin{aligned} \label{D-17}
&\frac12\frac{d}{dt}\|f^{\ell + 1} - f^{\ell}\|_{L^2}^2 + \frac{\lambda_0}{m}\| (\bbi - \bbp) (f^{\ell + 1} - f^{\ell})\|_\mu^2\cr
&\hspace{1cm} \leq \int_{\mathbb{T}} \big| \langle  {\mathcal L}_1(f^{\ell + 1} - f^{\ell}), f^{\ell + 1} - f^{\ell} \rangle \big| \,d\theta \\
& \hspace{1cm} + \int_{\mathbb{T}} \big| \langle \mathcal{N}(f^{\ell} - f^{\ell - 1}, f^{\ell + 1}), f^{\ell + 1} - f^{\ell} \rangle \big| \,d\theta \\
& \hspace{1cm} + \int_{\mathbb{T}} \big| \langle \mathcal{N}(f^{\ell - 1}, f^{\ell + 1} - f^{\ell}), f^{\ell + 1} - f^{\ell} \rangle \big| \,d\theta \\
& \hspace{1cm} =: {\mathcal I}_{31} + {\mathcal I}_{32} + {\mathcal I}_{33},
\end{aligned}
\end{align}
where $\mathcal{I}_{31}$ and $\mathcal{I}_{32}$ can be estimated as follows.
\begin{align}
\begin{aligned} \label{D-18}
|{\mathcal I}_{31}|  &\lesssim \frac{\kappa}{\sqrt{m \sigma}} \|f^{\ell + 1} - f^{\ell}\|_{L^2}^2, \\
|{\mathcal I}_{32}|  &\lesssim \frac{\kappa}{\sqrt{m \sigma}} \|f^{\ell + 1}\|_{\mu} \|f^{\ell} - f^{\ell - 1}\|_{L^2} \|f^{\ell + 1} - f^{\ell}\|_{L^2} \\
& \leq \frac{C\kappa \eta_0}{\sqrt{m \sigma}} \|f^{\ell} - f^{\ell - 1}\|_{L^2} \|f^{\ell + 1} - f^{\ell}\|_{L^2} \\
& \leq C \|f^{\ell} - f^{\ell - 1}\|_{L^2}^2 + C \|f^{\ell + 1} - f^{\ell}\|_{L^2},
\end{aligned}
\end{align}
for some $C>0$. Next, we estimate $\mathcal{I}_{33}$ as follows. Note that
\begin{align*}
\begin{aligned}
{\mathcal I}_{33} &= \frac{\kappa}{2\sigma} \int_{\bbt \times \bbr^2} \mathcal {S}[\sqrt{M}f^{\ell - 1}] (\omega - \nu)( f^{\ell + 1} - f^{\ell})^2 \,d\theta d\omega d\nu \\
&\hspace{0.5cm} -  \frac{\kappa}{m} \int_{\bbt \times \bbr^2} \mathcal{S}[\sqrt{M} f^{\ell - 1}] \partial_{\omega}(f^{\ell + 1} -f^{\ell}) (f^{\ell + 1} -f^{\ell}) \,d\theta d\omega d\nu.
\end{aligned}
\end{align*}
We can directly check that the second term in ${\mathcal I}_{33}$ equals to zero. Thus, we only need to estimate the first term. We use H\"older inequality and Lemma \ref{L4.1} $(ii)$ to get
\begin{align} \label{D-18-1}
\begin{aligned}
|{\mathcal I}_{33}|  &\leq \frac{\kappa}{2\sigma} \|\mathcal {S}[\sqrt{M}f^{\ell - 1}]\|_{L^{\infty}} \\
& \quad \times \sqrt{ \frac{\sigma}{m}} \Big( \int_{\bbt \times \bbr^2}  \frac{m}{\sigma}(\omega - \nu)^2|f^{\ell + 1} - f^{\ell}|^2 \,d\theta d\omega d\nu \Big)^{\frac{1}{2}}  \Big( \int_{\bbt \times \bbr^2}  |f^{\ell + 1} - f^{\ell}|^2 \,d\theta d\omega d\nu \Big)^{\frac{1}{2}} \\
& \lesssim \frac{\kappa}{\sqrt{m \sigma}} \|f^{\ell - 1}\|_{L^2}  \|f^{\ell + 1} - f^{\ell}\|_{\mu}  \|f^{\ell + 1} - f^{\ell}\|_{L^2} \\
& \leq \frac{C\kappa \eta_0}{\sqrt{m \sigma}} \|f^{\ell + 1} - f^{\ell}\|_{\mu}  \|f^{\ell + 1} - f^{\ell}\|_{L^2} \\
&\leq \frac{\lambda_0}{2m}\|f^{\ell + 1} - f^{\ell}\|_{\mu}^2 + C \|f^{\ell + 1} - f^{\ell}\|_{L^2}^2,
\end{aligned}
\end{align}
for some $C>0$. We combine \eqref{D-17}, \eqref{D-18}, and \eqref{D-18-1} to get
\begin{align*}
\begin{aligned}
& \frac{d}{dt}\|f^{\ell + 1} - f^{\ell}\|_{L^2}^2 + \frac{\lambda_0}{m}\|f^{\ell + 1} - f^{\ell}\|_\mu^2 \lesssim   \|f^{\ell + 1} - f^{\ell}\|_{L^2}^2 + \|f^{\ell} -f^{\ell - 1}\|_{L^2}^2.
\end{aligned}
\end{align*}
Thus we have
\[
\|f^{\ell + 1} - f^{\ell}\|_{L^2}^2 \leq \frac{CT_0^{\ell+1}}{(\ell+1)!} \quad \mbox{for} \quad t \in [0,T_0].
\]
\end{proof}

\subsection{Proof of Theorem \ref{T4.1}} We are now ready to provide the proof of Theorem \ref{T4.1}.  We apply Gagliardo-Nirenberg interpolation inequality together with Proposition \ref{P4.1} and Lemma \ref{L4.3} to see
\[
f^{\ell} \to f \quad \mbox{in} \quad \mathcal{C}([0,T_0];H^{s-1}(\mathbb{T} \times \mathbb{R}^2)) \quad \mbox{as} \quad \ell \to \infty.
\]
Moreover, it follows from the lower semicontinuity of the norm that $f^{\ell} \in X(s,T_0; \eta)$ gives
\[
M + \sqrt M f \geq 0 \quad \mbox{and} \quad \sup_{0 \leq t \leq T_0}\|f(t)\|_{H^s} \leq L,
\]
(see \cite{Choi,D-F-T} for more details). This proves the local-in-time existence of solutions. For the proof of uniqueness, let $f$ and $g$ be the strong solutions obtained above with the same initial data. Then, we have
\[
\|f(t) - g(t)\|_{L^2}^2 \leq C\int_0^t \|f(\tau) - g(\tau)\|_{L^2}^2 \,d\tau,
\]
for some positive constant $C > 0$, thanks to Lemma \ref{L4.3}. Then the standard argument establishes the uniqueness of the strong solutions satisfying the desired regularity.


%
%
%
%
\section{A priori estimates: Proof of Theorem \ref{T-I}} \label{sec:5}
In this section, we provide a priori estimates for the global-in-time existence of the unique strong solution to the perturbed equation \eqref{A-3}. Employing the classical energy method together with our careful analysis on the macro-micro decomposition, we obtain uniform a priori estimates of energy inequalities without any smallness assumptions on the initial data. Those estimates enable us to extend the local-in-time strong solutions to the global-in-time ones. \newline

Recall our main equation:
\begin{align}
\left \{ \begin{aligned} \label{F-1}
&\partial_t f + \omega \partial_{\theta} f = {\mathcal L}_0 f + {\mathcal L}_1 f + {\mathcal N}(f, f), \\
& f(\theta, \omega, \nu, 0) = \frac{1}{\sqrt{M}} \big( F^{in} - M \big) := f^{in},
\end{aligned} \right.
\end{align}
where the linear operators ${\mathcal L}_i,~i=0,1$ and the nonlinear operator  ${\mathcal N}$ are given by
\begin{align*}
\begin{aligned} 
{\mathcal L}_0 f &:= \frac{\sigma}{m^2}  \frac{1}{\sqrt{M}} \partial_{\omega} \Big[  M \partial_{\omega} \Big( \frac{f}{\sqrt{M}}  \Big) \Big], \\
{\mathcal L}_1 f &:=\frac{\kappa}{\sigma} (\omega - \nu) \sqrt{M} \mathcal{S}[\sqrt{M} f], \\
{\mathcal N}(f,f) &:= \frac{\kappa}{2\sigma} \mathcal {S}[\sqrt{M} f] (\omega - \nu) f -  \frac{\kappa}{m} \mathcal{S}[\sqrt{M} f] \partial_{\omega}f.
\end{aligned}
\end{align*}

We first present a technical lemma which will be used later.

\begin{lemma}\label{L6.1}
Let $s \geq 1$ and $T >0$ be given. Suppose that $f \in \mathcal{C}([0,T]; H^s(\mathbb{T} \times \mathbb{R}^2))$ is a solution to the equation \eqref{F-1}. Then, for $0 \leq k \leq s$ and $0 \leq t \leq T$,  we have
\begin{align*}
\begin{aligned}
&(i)~\| \partial^k_\theta S[\sqrt{M} f](t) \|_{L^{\infty}} \leq \min \big\{ 1, \| f_0\|_{L^{2}} \big\}. \\
&(ii)~\|\partial_\theta^k \bbp f(t)\|_{L^2} \leq  \|\partial_\theta^k f_0 \|_{L^2} + \|\partial_\theta^k f_1 \|_{L^2}. \\
&(iii)~\|\partial_\theta^k \partial_{\omega} \bbp f\|_{L^2_{\omega,\nu}}
 \lesssim \sqrt{\frac{m}{\sigma}} \left(|\partial_{\theta}^k f_0| + |\partial_{\theta}^k f_1|\right).
\end{aligned}
\end{align*}
\end{lemma}

\begin{proof} (i) We use \eqref{C-3} to obtain
$$\begin{aligned}
\partial^k_\theta S[\sqrt{M} f](\theta,t) & = \int_{\mathbb{T} \times \mathbb{R}^2} \partial_{\theta}^{k} \sin(\theta_* - \theta) \sqrt{M}(\omega_*, \nu_*) f(\theta_*, \omega_*, \nu_*,  t) \,d\theta_* d\omega_* d\nu_* \\
& =  \int_{\mathbb{T} } \partial_{\theta}^{k} \sin(\theta_* - \theta) \left(\int_{\mathbb{R}^2} \sqrt{M}(\omega_*, \nu_*) f(\theta_*, \omega_*, \nu_*,  t) \, d\omega_*  d\nu_* \right)d\theta_*, \\
& = \frac{1}{\sqrt{2\pi}} \int_{\mathbb{T}} \partial_{\theta}^{k} \sin(\theta_* - \theta) f_0 (\theta_*,t) \,d\theta_*.
\end{aligned}$$
This yields
\begin{equation} \label{F-1-1}
\|\partial^k_\theta S[\sqrt{M} f](t)\|_{L^\infty} \leq \frac{1}{\sqrt{2\pi}} \int_{\mathbb{T} } | f_0(\theta_*,t)| \,d\theta_* \leq  \| f_0(t)\|_{L^{2}}.
\end{equation}
On the other hand, we use the periodicity of $\sin (\theta_* - \theta)$ in $\theta$-variable to find
\begin{equation} \label{F-2}
\int_{\mathbb{T} \times \mathbb{R}^2} \partial_{\theta}^{k} \sin(\theta_* - \theta) M(\omega_*,\nu_*) \,d\theta_* d\omega_* d\nu_* = 0.
\end{equation}
Then, the relation \eqref{F-2}, definition relation \eqref{B-20} of $f$ and the initial assumption of $F$ yield
\begin{align}
\begin{aligned} \label{F-3}
\partial^k_\theta S[\sqrt{M} f](\theta,t) & = \int_{\mathbb{T} \times \mathbb{R}^2 } \partial_{\theta}^{k} \sin(\theta_* - \theta) \sqrt{M(\omega_*, \nu_*)} f(\theta_*, \omega_*, \nu_*,  t)\, d\theta_* d\omega_* d\nu_* \\
& = \int_{\mathbb{T} \times \mathbb{R}^2 }\partial_{\theta}^{k} \sin(\theta_* - \theta) F(\theta_*, \omega_*, \nu_*, t)\, d\theta_* d\omega_* d\nu_* \\
& \quad - \int_{\mathbb{T} \times \mathbb{R}^2 }\partial_{\theta}^{k} \sin(\theta_* - \theta) M(\omega_*, \nu_*) \,d\theta_* d\omega_* d\nu_* \\
&  =\int_{\mathbb{T} \times \mathbb{R}^2 } \partial_{\theta}^{k} \sin(\theta_* - \theta) F(\theta_*, \omega_*, \nu_*,  t)\, d\theta_* d\omega_* d\nu_* \\
& \leq \int_{\mathbb{T} \times \mathbb{R}^2 } F(\theta_*, \omega_*, \nu_*, t)\, d\theta_* d\omega_* d\nu_* = 1,
\end{aligned}
\end{align}
where we used the conservation of mass. Then, the first desired estimate follows from \eqref{F-1-1} and \eqref{F-3}.  \newline

\noindent (ii) Recall that
\begin{equation} \label{F-4}
\bbp f =  f_0 \chi_0 + f_1 \chi_1, \quad \chi_0 := \sqrt{2\pi M}, \quad \chi_1 :=  \sqrt\frac{2\pi m}{\sigma} (\omega - \nu) \sqrt M.
\end{equation}
Note that for $0 \leq k \leq s$
\[
\partial_\theta^k \bbp f = ( \partial_\theta^k f_0 ) \chi_0 + ( \partial_\theta^k f_1 ) \chi_1,
\]
Then, we use the normalization of $\chi_0$ and $\chi_1$ to see
\[
\|\partial_\theta^k \bbp f\|_{L^2_{\omega,\nu}} \leq |\partial_\theta^k f_0| + |\partial_\theta^k f_1|.
\]
This yields
\[ \|\partial_\theta^k \bbp f\|_{L^2} \leq \|\partial_\theta^k f_0 \|_{L^2} + \|\partial_\theta^k f_1 \|_{L^2}. \]

\vspace{0.5cm}

\noindent (iii)~ It follows from \eqref{B-21} and \eqref{F-4} that
\begin{align}
\begin{aligned} \label{F-5}
& \partial_{\theta}^k \partial_{\omega} \bbp f = \partial_\theta^k f_0 \partial_\omega \chi_0 +  \partial_\theta^k f_1 \partial_\omega \chi_1, \\
& \partial_\omega \chi_0 = -\sqrt{2\pi} \frac{m}{2\sigma} (\omega - \nu) \sqrt{M}, \quad \partial_\omega \chi_1 = \sqrt{\frac{2\pi m}{\sigma}} \sqrt{M} - \frac{m}{2\sigma} \sqrt{\frac{2m \pi}{\sigma}} (\omega - \nu)^2 \sqrt{M}.
\end{aligned}
\end{align}
Now, we use \eqref{F-5} to obtain
\begin{align*}
\begin{aligned}
\|\partial_\theta^k \partial_{\omega}\bbp f\|_{L^2_{\omega,\nu}}^2
&= \int_{\mathbb{R}^2} 2\pi (\frac{m}{2\sigma})^2 (\omega - \nu)^2 M |\partial_{\theta}^k f_0|^2 \,d\omega d\nu  + \int_{\mathbb{R}^2}  \frac{2\pi m}{\sigma} M |\partial_{\theta}^k f_1|^2 \,d\omega d\nu \\
 &+  \int_{\mathbb{R}^2} \frac{2\pi m}{\sigma} (\frac{m}{2\sigma})^2 (\omega - \nu)^4 M |\partial_{\theta}^k f_1|^2 \,d\omega d\nu - \int_{\mathbb{R}^2} 2\frac{2\pi m}{\sigma} \frac{m}{2\sigma} (\omega - \nu)^2 M |\partial_{\theta}^k f_1|^2 \,d\omega d\nu \\
 & = \frac{m}{4\sigma} |\partial_{\theta}^k f_0 |^2 + \frac{3m}{4\sigma} |\partial_{\theta}^k f_1|^2.
\end{aligned}
\end{align*}
Here $\int_{\mathbb{R}^2} \frac{2\pi m}{\sigma} (\frac{m}{2\sigma})^2 (\omega - \nu)^4 M |\partial_{\theta}^k f_1|^2 \,d\omega d\nu=\frac{3m}{4\sigma} |\partial_{\theta}^k f_1|^2$.
This yields the desired result:
\[
\|\partial_\theta^k \bbp f\|_{L^2_{\omega,\nu}}
 \lesssim \sqrt{\frac{m}{\sigma}} \left(|\partial_{\theta}^k f_0| + |\partial_{\theta}^k f_1|\right).
\]
\end{proof}

\subsection{Zeroth-order estimate}
We present the zeroth-order estimate for solution $f$ as follows.

\begin{lemma} \label{L6.2}
 Let $s \geq 1$ and $T >0$ be given. Let $f \in \mathcal{C}([0,T]; H^s(\mathbb{T} \times \mathbb{R}^2))$ be a solution to the equation \eqref{F-1}. Suppose that $\sigma$ satisfies $\sigma \geq Cm\kappa^2$ for sufficiently large constant $C>0$.
 Then, we have
\[
\frac{d}{dt} \|f\|^2_{L^2} + \frac{1}{m} \| (\bbi - \bbp) f\|_{\mu}^2 + \frac{1}{m} \| f_1 \|_{L^{2}}^{2} \lesssim\frac{\kappa}{\sqrt{m\sigma}} \| f_0 \|_{L^2}^2. \]
\end{lemma}

\begin{proof}
We multiply equation $\eqref{F-1}_1$ by $f$, and integrate the resulting relation over $\mathbb{T} \times \mathbb{R}^2$  to get
\begin{align}
\begin{aligned} \label{F-6}
\frac12\frac{d}{dt} \|f\|^2_{L^2} + \int_{\mathbb{T} } \langle -{\mathcal L}_0 f, f \rangle \,d\theta  &=  \int_{\mathbb{T} } \langle {\mathcal L}_1 f, f \rangle \,d\theta  + \int_{\mathbb{T} } \langle f, \mathcal{N}(f, f) \rangle \,d\theta.
\end{aligned}
\end{align}

\vspace{0.5cm}

\noindent $\bullet$ (Estimate A):  It follows from Lemma \ref{L3.2} $(ii)$ that
\begin{equation} \label{F-6-1}
\int_{\mathbb{T} } \langle -{\mathcal L}_0 f, f \rangle \,d\theta \geq \frac{\lambda_0}{m} \|(\bbi - \bbp) f\|_{\mu}^2 + \frac{1}{m} \|f_1 \|_{L^{2}}^{2}.
\end{equation}

\noindent $\bullet$ (Estimate B):  We use \eqref{C-3-1} to have
\begin{align} \label{F-7-1}
\begin{aligned}
\langle {\mathcal L}_1 f, f \rangle = \frac{\kappa}{\sigma}  \mathcal{S}[\sqrt{M}f] \langle (\omega - \nu) \sqrt{M}, f \rangle = \frac{\kappa}{\sqrt{2 \pi m\sigma}}  \mathcal{S}[\sqrt{M}f] \langle \chi_1, f \rangle = \frac{\kappa}{\sqrt{2 \pi m\sigma}}  \mathcal{S}[\sqrt{M}f] f_1.
\end{aligned}
\end{align}
Thus, we use \eqref{F-7-1}, H\"older inequality, and Lemma \ref{L6.1} $(i)$ to obtain
\begin{align}\label{F-7}
\begin{aligned}
\Big| \int_{\mathbb{T} } \langle {\mathcal L}_1 f, f \rangle d\theta \Big|&= \frac{\kappa}{\sqrt{2 \pi m\sigma}}\int_\mathbb{T} \mathcal{S}[\sqrt{M} f] f_1 \,d\theta  \leq \frac{\kappa}{\sqrt{m\sigma}}\|\mathcal{S}[\sqrt{M} f]\|_{L^2}\|f_1 \|_{L^2} \\
&  \leq \frac{\kappa}{\sqrt{m\sigma}}\| f_0 \|_{L^2}\| f_1 \|_{L^2} \leq \frac{\kappa}{2\sqrt{m\sigma}}\|(f_0, f_1)\|_{L^2}^2.
\end{aligned}
\end{align}

\noindent $\bullet$ (Estimate C): Direct calculation yields
\begin{align}\label{F-8}
\begin{aligned}
\int_{\mathbb{T} } \langle f, \mathcal{N}(f, f) \rangle \,d\theta &= \frac{\kappa}{2\sigma} \int_{\mathbb{T} \times \mathbb{R}^2} \mathcal{S}[\sqrt{M} f] (\omega - \nu) f^2 \,d\theta d\omega d\nu  - \frac{\kappa}{m}  \int_{\mathbb{T} \times \mathbb{R}^2} \mathcal{S}[\sqrt{M} f]  \partial_{\omega}f f \,d\theta d\omega d\nu  \\
& = \frac{\kappa}{2\sigma} \int_{\mathbb{T} \times \mathbb{R}^2} \mathcal{S}[\sqrt{M} f] (\omega - \nu) f^2 \,d\theta d\omega d\nu\\
& = \frac{\kappa}{2\sigma} \int_{\mathbb{T} \times \mathbb{R}^2} \mathcal{S}[\sqrt{M} f] (\omega - \nu) |\bbp f|^2 \,d\theta d\omega d\nu  \\
& + \frac{\kappa}{\sigma} \int_{\mathbb{T} \times \mathbb{R}^2} \mathcal{S}[\sqrt{M} f] (\omega - \nu) \bbp f\cdot (\bbi - \bbp) f\, d\theta d\omega d\nu \\
&+ \frac{\kappa}{2\sigma} \int_{\mathbb{T} \times \mathbb{R}^2} \mathcal{S}[\sqrt{M} f] (\omega - \nu) |(\bbi - \bbp) f|^2 \,d\theta d\omega d\nu\\
&  =: {\mathcal I}_{41} + {\mathcal I}_{42} + {\mathcal I}_{42}.
\end{aligned}
\end{align}
Next, we estimate ${\mathcal I}_{4i}$, $i = 1,2,3$ as follows. \newline

\noindent $\diamond$ (Estimate on ${\mathcal I}_{41}$): We first note that
\begin{align}
\begin{aligned} \label{F-9}
& \int_{\mathbb{R}^2}(\omega - \nu) |\bbp f|^2 \, d\omega  d\nu  \\
& \hspace{0.5cm} = \int_{\mathbb{R}^2}(\omega - \nu) | f_0 \chi_0 + f_1 \chi_1 |^2 \, d\omega  d\nu \\
& \hspace{0.5cm} = \int_{\mathbb{R}^2}(\omega - \nu) \left( \sqrt{2\pi}  f_0  \sqrt{M} + \sqrt{\frac{2\pi m}{\sigma}}  f_1 (\omega - \nu) \sqrt{M} \right)^2 d\omega  d\nu \\
& \hspace{0.5cm} = \int_{\mathbb{R}^2}(\omega - \nu) \left( 2\pi  f_0^2 M + \frac{2\pi m}{\sigma}  f_1^2 (\omega - \nu)^2 M + 4\pi \sqrt{\frac{m}{\sigma}}  f_0  f_1 (\omega - \nu)M \right)  d\omega  d\nu \\
& \hspace{0.5cm} = 4\pi \sqrt{\frac{m}{\sigma}}  f_0  f_1 \int_{\mathbb{R}^2}(\omega - \nu)^2 M\, d\omega  d\nu \\
& \hspace{0.5cm} = 2\sqrt{\frac{\sigma}{m}}  f_0  f_1.
\end{aligned}
\end{align}
Here we used
\[ \int_{\mathbb{R}^2}(\omega - \nu)^2 M \,d\omega d\nu = \frac{1}{2\pi}\frac{\sigma}{m}. \]
Then, the above estimate \eqref{F-9} implies
\begin{align}
\begin{aligned} \label{F-10}
|{\mathcal I}_{41}| &= \Big| \frac{\kappa}{\sqrt{m \sigma}} \int_{\mathbb{T}} \mathcal{S}[\sqrt{M} f]  f_0  f_1 \,d\theta \Big| \leq \frac{\kappa}{\sqrt{m \sigma}} \|\mathcal{S}[\sqrt{M} f]\|_{L^{\infty}} \| f_0\|_{L^2} \| f_1\|_{L^2} \\
& \leq \frac{\kappa}{\sqrt{m \sigma}}  \| f_0\|_{L^2} \| f_1\|_{L^2} \leq \frac{\kappa}{2\sqrt{m\sigma}}\|( f_0,  f_1)\|_{L^2}^2.
\end{aligned}
\end{align}

\vspace{0.5cm}

\noindent $\diamond$ (Estimate on ${\mathcal I}_{42}$): Note that
\begin{align*}
\begin{aligned}
\langle (\omega - \nu) \bbp f, (\bbi - \bbp) f \rangle & = \sqrt{\frac{\sigma}{m}} \int_{\mathbb{R}^2} \bbp f \cdot \sqrt{\frac{m}{\sigma}} (\omega - \nu) (\bbi - \bbp) f \,d\omega d\nu \\
& \leq \sqrt{\frac{\sigma}{m}} \left( \int_{\mathbb{R}^2} |\bbp f|^2 \,d\omega d\nu \right)^{\frac12} \left( \int_{\mathbb{R}^2} \frac{m}{\sigma} (\omega - \nu) |(\bbi - \bbp) f|^2 \,d\omega d\nu \right)^{\frac12} \\
& \leq \sqrt{\frac{\sigma}{m}} (\sqrt{ |f_0|^2 + | f_1|^2} \| (\bbi - \bbp) f\|_{L_\mu^2}.
\end{aligned}
\end{align*}
This yields
\begin{align}
\begin{aligned} \label{F-11}
|{\mathcal I}_{42}| &\leq \frac{\kappa}{\sqrt{m\sigma}} \| \mathcal{S}[\sqrt{M} f] \|_{L^{\infty}} \int_{\mathbb{T}}  \sqrt{ |f_0|^2 + | f_1|^2}
\| (\bbi - \bbp) f\|_{L_\mu^2}\, d\theta \\
 & \leq \frac{\kappa}{2\sqrt{m\sigma}} \left(  \|(\bbi - \bbp) f\|^2_{\mu} + \| (f_0, f_1) \|_{L^2}^2 \right).
\end{aligned}
\end{align}

\vspace{0.5cm}

\noindent $\diamond$ (Estimate on ${\mathcal I}_{43}$):  By a straightforward computation and Lemma \ref{L6.1} $(i)$, we have
\begin{align}
\begin{aligned} \label{F-12}
|{\mathcal I}_{43}| & =  \left|\frac{\kappa}{2\sigma}\sqrt{\frac{\sigma}{m}} \int_{\mathbb{T}} \mathcal{S}[\sqrt{M} f]\left( \int_{\mathbb{R}^2}\sqrt{\frac{m}{\sigma}}(\omega - \nu) | (\bbi - \bbp) f|^2 \,d\omega  d\nu \right)d\theta \right| \\
& \leq  \frac{\kappa}{2\sqrt{m \sigma}} \int_{\mathbb{T}} |\mathcal{S}[\sqrt{M} f] |\| (\bbi - \bbp) f \|_{\mu}^2 \,d\theta   \leq  \frac{\kappa}{2\sqrt{m \sigma}} \| \mathcal{S}[\sqrt{M} f] \|_{L^{\infty}} \| (\bbi - \bbp) f \|_{\mu}^2 \\
&  \leq  \frac{\kappa}{2\sqrt{m \sigma}} \| (\bbi - \bbp) f \|_{\mu}^2.
\end{aligned}
\end{align}
We  substitute all the above estimates \eqref{F-10}, \eqref{F-11}, \eqref{F-12} for ${\mathcal I}_{4i},~~i=1,2,3$ into \eqref{F-8} to get
\begin{equation}\label{F-13}
\Big| \int_{\mathbb{T} } \langle f, \mathcal{N}(f, f) \rangle \,d\theta \Big| \leq \frac{\kappa}{\sqrt{m \sigma}} \Big( \| (f_0, f_1) \|_{L^2}^2  +
\| (\bbi - \bbp) f \|_{\mu}^2 \Big).
\end{equation}
In \eqref{F-6}, we collect the estimates \eqref{F-6-1},  \eqref{F-7} and \eqref{F-13} to have
\begin{align*}
\begin{aligned}
& \frac12\frac{d}{dt} \|f\|^2_{L^2} + \frac{\lambda_0}{m} \| (\bbi - \bbp) f\|_{\mu}^2 + \frac{1}{m} \|f_1\|_{L^{2}}^{2} \leq \frac{\kappa}{2\sqrt{m \sigma}}
\Big( 3\| (f_0, f_1) \|_{L^2}^2  + 2\| (\bbi - \bbp) f \|_{\mu}^2 \Big),
\end{aligned}
\end{align*}
i.e.,
\[
\frac{d}{dt} \|f\|^2_{L^2} + \left(\frac{2\lambda_0}{m}-\frac{2\kappa}{\sqrt{m \sigma}} \right)\| (\bbi - \bbp) f\|_{\mu}^2 + \left(\frac{2}{m}-\frac{3\kappa}{\sqrt{m \sigma}} \right) \|f_1\|_{L^{2}}^{2} \leq \frac{3\kappa}{\sqrt{m \sigma}} \|f_0\|_{L^2}^2.
\]
By the assumption, we have
\[ \frac{2\kappa}{\sqrt{m \sigma}} \leq  \frac{\lambda_0}{4m} \quad \mbox{and} \quad \frac{3\kappa}{\sqrt{m \sigma}} \leq  \frac{1}{4m}. \]
This yields
\begin{align*}
\begin{aligned}
& \frac{d}{dt} \|f\|^2_{L^2} + \frac{\lambda_0}{m} \| (\bbi - \bbp) f\|_{\mu}^2 + \frac{1}{m} \|f_1\|_{L^{2}}^{2} \lesssim \frac{\kappa}{\sqrt{m \sigma}} \|f_0\|_{L^2}^2.
\end{aligned}
\end{align*}
\end{proof}

\subsection{Higher-order estimates with respect to $\theta$}
In this subsection, we will give the $H_{\theta}^k$-estimates on solution $f$.
\begin{lemma} \label{L6.3}
For $s \geq 1$ and $T >0$, let $f \in \mathcal{C}([0,T]; H^s(\mathbb{T} \times \mathbb{R}^2))$ be a solution to the equation \eqref{F-1}. Suppose that $\sigma$ satisfies $\sigma \geq Cm\kappa^2$ for sufficiently large constant $C>0$. Then, for $1 \leq k \leq s$, we have
$$\begin{aligned}
&\frac{d}{dt}\|\partial_\theta^k f\|_{L^2}^2 + \frac{1}{m} \|\partial_\theta^k (\bbi - \bbp) f\|_{\mu}^2 + \frac1m \|\partial_\theta^k f_1\|_{L^2}^2 \lesssim \frac{\kappa}{\sqrt{m\sigma}}\left(\|(\bbi - \bbp) f\|_{\mu}^2 + \|f_1\|_{L^2}^2 + \|\partial_\theta^k f_0\|_{L^2}^2 \right).
\end{aligned}$$
\end{lemma}
\begin{proof}
Let $1 \leq k \leq s$, Then, we take the differential operator $\partial_\theta^k$ to \eqref{F-1} to get
\[
\partial_t \partial_\theta^k f + \omega \partial_\theta^{k+1}f  = {\mathcal L}_0 \partial_\theta^k f + \partial_\theta^k {\mathcal L}_1  f + \partial_\theta^k \Gamma(f,f),
\]
where we have used the communicative property of $\mathcal{L}_0$ and $\partial_{\theta}^{k}$. Then we find
\[ \frac12\frac{d}{dt} \|\partial_\theta^k f\|_{L^2}^2 + \int_\mathbb{T} \langle -{\mathcal L}_0 \partial_\theta^k f,\partial_\theta^k f \rangle \,d\theta \\
=  \int_\mathbb{T} \langle \partial_\theta^k {\mathcal L}_1  f,\partial_\theta^k f \rangle \,d\theta + \int_\mathbb{T} \langle \partial_\theta^k \mathcal{N}(f,f),\partial_\theta^k f \rangle \,d\theta.
\]

\vspace{0.2cm}

\noindent $\bullet$ (Estimate A):  By the coercivity estimate of ${\mathcal L}_0$, we have
\begin{equation*}
\int_\mathbb{T} \langle -{\mathcal L}_0 \partial_\theta^k f,\partial_\theta^k f \rangle \,d\theta  \geq \frac{\lambda_0}{m} \| (\bbi - \bbp)\partial_\theta^k f\|_{\mu}^2 + \frac1m \|\partial_\theta^k f_1 \|_{L^2}^2.
\end{equation*}

\vspace{0.2cm}

\noindent $\bullet$ (Estimates B) By direct estimate, we have
\begin{align*}
\begin{aligned} 
\left| \int_\mathbb{T} \langle \partial_\theta^k {\mathcal L}_1  f,\partial_\theta^k f \rangle \,d\theta \right| &= \left| \frac{\kappa}{\sqrt{2\pi m\sigma}}\int_\mathbb{T} \partial_\theta^k \mathcal{S}[\sqrt{M} f] \partial_\theta^k f_1 \,d\theta  \right| \\
& \leq \frac{\kappa}{\sqrt{m\sigma}}\|\partial_\theta^k \mathcal{S}[\sqrt{M} f]\|_{L^\infty}\|\partial_\theta^k f_1 \|_{L^2} \\
& \leq \frac{\kappa}{\sqrt{m\sigma}}\|f_0\|_{L^2} \|\partial_\theta^k f_1\|_{L^2}  \leq \frac{\kappa}{2\sqrt{m\sigma}}\|(f_0, \partial_\theta^k f_1)\|_{L^2}^2.
\end{aligned}
\end{align*}

\vspace{0.2cm}

\noindent $\bullet$ (Estimate C):  Direct calculation yields
\begin{align*}
\begin{aligned}
\int_\mathbb{T} \langle \partial_\theta^k \mathcal{N}(f,f),\partial_\theta^k f \rangle \,d\theta & = \frac{\kappa}{2\sigma} \int_{\mathbb{T}} \langle \partial_{\theta}^k \big( \mathcal{S}[\sqrt{M} f] (\omega - \nu) f\big), \partial_{\theta}^k f \rangle \,d\theta   \\
&  - \frac{\kappa}{m}  \int_{\mathbb{T}} \langle\partial_{\theta}^k \big( \mathcal{S}[\sqrt{M} f]  \partial_{\omega}f \big), \partial_{\theta}^k f \rangle \,d\theta =: {\mathcal I}_{51} + {\mathcal I}_{52}.
\end{aligned}
\end{align*}
Next, we will estimate ${\mathcal I}_{5i}, i =1, 2$ as follows. \newline

\noindent $\diamond$ (Estimate on ${\mathcal I}_{51}$): We split it into four terms as in Lemma \ref{L6.2}:
\begin{align*}
\begin{aligned}
{\mathcal I}_{51} &= \frac{\kappa}{2\sigma}\sum_{0 \leq \ell \leq k} \binom{k}{\ell} \int_\mathbb{T} \partial_\theta^{k -\ell} (\mathcal{S}[\sqrt{M} f]) \langle (\omega - \nu)\partial_\theta^{\ell} f , \partial_\theta^k f\rangle\,d\theta\cr
&=  \frac{\kappa}{2\sigma}\sum_{0 \leq \ell \leq k} \binom{k}{\ell}\int_\mathbb{T} \partial_\theta^{k -\ell} (\mathcal{S}[\sqrt{M} f] )\left\langle(\omega - \nu) \partial_\theta^{ \ell} \bbp f, \partial_\theta^k \bbp f\right\rangle d\theta \cr
&+  \frac{\kappa}{2\sigma}\sum_{0 \leq \ell \leq k} \binom{k}{\ell}\int_\mathbb{T} \partial_\theta^{k- \ell} (\mathcal{S}[\sqrt{M} f] )\left\langle(\omega - \nu) \partial_\theta^{\ell} \bbp f, \partial_\theta^k (\bbi - \bbp) f\right\rangle d\theta  \cr
&+  \frac{\kappa}{2\sigma}\sum_{0 \leq \ell \leq k} \binom{k}{\ell}\int_\mathbb{T} \partial_\theta^{k -\ell} (\mathcal{S}[\sqrt{M} f]) \left\langle(\omega - \nu) \partial_\theta^{\ell} (\bbi - \bbp) f, \partial_\theta^k  \bbp f\right\rangle d\theta  \cr
&+  \frac{\kappa}{2\sigma}\sum_{0 \leq \ell \leq k} \binom{k}{\ell}\int_\mathbb{T} \partial_\theta^{k -\ell} (\mathcal{S}[\sqrt{M} f] )\left\langle(\omega - \nu) \partial_\theta^{\ell}
 (\bbi - \bbp) f, \partial_\theta^k (\bbi - \bbp) f\right\rangle d\theta \cr
&=: \sum_{j=1}^4 {\mathcal I}_{51}^j.
\end{aligned}
\end{align*}
 Note that
\[
\left\langle(\omega - \nu) \partial_\theta^{\ell} \bbp f, \partial_\theta^k \bbp f\right\rangle = \sqrt{\frac\sigma m}\Big(\partial_\theta^{\ell} f_0
\partial_\theta^k f_1  + \partial_\theta^{\ell} f_1 \partial_\theta^k f_0  \Big), \quad \mbox{for $0 \leq \ell \leq k$.}
\]
This gives
\begin{align}
\begin{aligned} \label{F-17}
|{\mathcal I}_{51}^1| &= \frac{\kappa}{2\sqrt{m\sigma}}\sum_{0 \leq \ell \leq k} \binom{k}{\ell}\int_\mathbb{T} \partial_\theta^{k -\ell} \mathcal{S}[\sqrt{M} f] \left(\partial_\theta^{\ell} f_0  \partial_\theta^k f_1 + \partial_\theta^{\ell}f_1 \partial_\theta^k f_0 \right)\,d\theta \\
&\lesssim \frac{\kappa}{\sqrt{m\sigma}}\sum_{0 \leq \ell \leq k}\|\partial_\theta^{k-\ell} \mathcal{S}[\sqrt{M} f]\|_{L^\infty}\left(\|\partial_\theta^{\ell} f_0 \|_{L^2} \|\partial_\theta^k f_1\|_{L^2} + \|\partial_\theta^{\ell}f_1\|_{L^2} \|\partial_\theta^k f_0 \|_{L^2}\right) \\
&\lesssim \frac{\kappa}{\sqrt{m\sigma}}\sum_{0 \leq \ell \leq k} \|\partial_\theta^\ell(f_0,f_1)\|_{L^2}^2.
\end{aligned}
\end{align}
It follows from Lemma \ref{L6.1} that
\begin{align}
\begin{aligned} \label{F-18}
|{\mathcal I}_{51}^2| &\lesssim \frac{\kappa}{\sqrt{m\sigma}}\sum_{0 \leq \ell \leq k}\|\partial_\theta^{k - \ell} \mathcal{S}[M^\frac12 f] \|_{L^\infty}\int_\mathbb{T} \|\partial_\theta^{\ell} \bbp f\|_{L^2}\|\partial_\theta^k (\bbi - \bbp) f\|_{L^2_\mu} \,d\theta \\
&\lesssim \frac{\kappa}{\sqrt{m\sigma}} \left( \sum_{0 \leq \ell \leq k}  \|\partial_\theta^\ell(f_0,f_1)\|_{L^2}^2 + \|\partial_\theta^k (\bbi - \bbp) f\|_{\mu}^2 \right).
\end{aligned}
\end{align}
Similarly, we obtain
\begin{align}
\begin{aligned} \label{F-18-1}
|{\mathcal I}_{51}^3| &\lesssim \frac{\kappa}{\sqrt{m\sigma}} \left( \sum_{0 \leq \ell \leq k}  \|\partial_\theta^k (f_0,f_1)\|_{L^2}^2 + \|\partial_\theta^{\ell} (\bbi - \bbp) f\|_{\mu}^2 \right),\cr
|{\mathcal I}_{51}^4| &\lesssim \frac{\kappa}{\sqrt{m\sigma}}\sum_{0 \leq \ell \leq k}\| \partial_\theta^\ell(\bbi - \bbp) f\|_{\mu}^2. 
\end{aligned}
\end{align}
We now combine \eqref{F-17}, \eqref{F-18}, and \eqref{F-18-1}  to obtain
\begin{equation*}
|{\mathcal I}_{51}| \leq \frac{C\kappa}{\sqrt{m\sigma}} \left(\sum_{0 \leq \ell \leq k} \|\partial_\theta^\ell(f_0,f_1)\|_{L^2}^2 +  \|\partial_\theta^{\ell} (\bbi - \bbp) f\|_{\mu}^2 \right).
\end{equation*}
\noindent $\diamond$ (Estimate on ${\mathcal I}_{52}$): Similarly as before, we decompose $\mathcal{I}_{52}$ into four terms:
\begin{align}
\begin{aligned} \label{F-21}
{\mathcal I}_{52} &=-\frac{\kappa}{m} \sum_{0 \leq \ell < k} \binom{k}{\ell}\int_\mathbb{T} \left\langle \partial_\theta^{k- \ell} (\mathcal{S}[\sqrt M f])\partial_\theta^{ \ell} \partial_\omega f, \partial_\theta^{k} f \right\rangle d\theta\cr
&=-\frac{\kappa}{m} \sum_{0 \leq \ell < k} \binom{k}{\ell}\int_\mathbb{T} \left\langle \partial_\theta^{k- \ell} (\mathcal{S}[\sqrt M f]) \partial_\theta^{ \ell} \partial_\omega (\bbi - \bbp) f, \partial_\theta^{k} (\bbi - \bbp) f \right\rangle d\theta\cr
&-\frac{\kappa}{m} \sum_{0 \leq \ell < k} \binom{k}{\ell}\int_\mathbb{T} \left\langle  \partial_\theta^{k- \ell} (\mathcal{S}[\sqrt M f]) \partial_\theta^{ \ell} \partial_\omega \bbp f, \partial_\theta^{k}  (\bbi - \bbp) f \right\rangle d\theta\cr
&-\frac{\kappa}{m} \sum_{0 \leq \ell < k} \binom{k}{\ell}\int_\mathbb{T} \left\langle  \partial_\theta^{k- \ell} (\mathcal{S}[\sqrt M f]) \partial_\theta^{ \ell} \partial_\omega
(\bbi - \bbp) f, \partial_\theta^{k} \bbp f \right\rangle d\theta\cr
&-\frac{\kappa}{m} \sum_{0 \leq \ell < k} \binom{k}{\ell}\int_\mathbb{T} \left\langle \partial_\theta^{k- \ell} (\mathcal{S}[\sqrt M f]) \partial_\theta^{ \ell} \partial_\omega \bbp f, \partial_\theta^{k} \bbp f \right\rangle d\theta \\
& =: \sum_{j=1}^4 {\mathcal I}_{52}^j,
\end{aligned}
\end{align}
where we used the fact that for any function $h$, it holds:
\[
\int_\mathbb{T} \left\langle  \mathcal{S}[\sqrt M f] \partial_{\omega} \partial_\theta^k h, \partial_\theta^k h \right\rangle d\theta = 0.
\]

\noindent We use Lemma \ref{L3.2} to have
\begin{align}
\begin{aligned} \label{F-22}
| {\mathcal I}_{52}^1| & = \left| \frac{\kappa}{m} \sqrt{\frac{m}{\sigma}} \sum_{0 \leq \ell < k} \binom{k}{\ell}\int_\mathbb{T}  \partial_\theta^{k -\ell} \mathcal{S}[\sqrt M f] \left\langle \sqrt{\frac{\sigma}{m}} \partial_\theta^{\ell}\partial_\omega (\bbi - \bbp) f, \partial_\theta^{k} (\bbi - \bbp) f \right\rangle d\theta \right| \\
& \leq \frac {\kappa}{\sqrt{m \sigma}} \sum_{0 \leq \ell < k} \binom{k}{\ell}\int_\mathbb{T}  |\partial_\theta^{k- \ell} \mathcal{S}[\sqrt M f]| \|\partial_\theta^{\ell} (\bbi - \bbp) f\|_{L_\mu^2} \|\partial_\theta^{k} (\bbi - \bbp) f\|_{L_\mu^2} d\theta \\
&\leq \frac {\kappa}{\sqrt{m \sigma}} \sum_{0 \leq \ell < k} \binom{k}{\ell}\|\partial_\theta^{k -\ell} \mathcal{S}[\sqrt{M} f]\|_{L^\infty} \|\partial_\theta^{\ell} (\bbi - \bbp) f\|_{\mu} \|\partial_\theta^{k} (\bbi - \bbp) f\|_{\mu} \\
&\lesssim \frac {\kappa}{\sqrt{m \sigma}} \sum_{0 \leq \ell \leq k} \|\partial_\theta^{\ell} (\bbi - \bbp) f\|_{\mu}^2.
\end{aligned}
\end{align}

\noindent Note that
\begin{equation} \label{F-23}
 {\mathcal I}_{52}^2 = \frac{\kappa}{m} \sum_{0 \leq \ell < k} \binom{k}{\ell}\int_\mathbb{T} \partial_\theta^{k-\ell} \mathcal{S}[\sqrt M f]  \left\langle \partial_\theta^{\ell} \bbp f, \partial_\theta^k \partial_{\omega} (\bbi - \bbp) f \right\rangle d\theta,
\end{equation}
and
\begin{align}
\begin{aligned} \label{F-24}
& \left|\left\langle \partial_\theta^{\ell} \bbp f, \partial_\theta^k \partial_{\omega} (\bbi - \bbp) f \right\rangle \right| = \left|\sqrt{\frac{m}{\sigma}} \int_{\mathbb{R}^2} \partial_\theta^{\ell} \bbp f \cdot \sqrt{\frac{\sigma}{m}} \partial_\theta^k \partial_{\omega} (\bbi - \bbp) f \,d\omega d\nu\right| \\
& \hspace{1cm}  \leq \sqrt{\frac{m}{\sigma}} \|\partial_\theta^{\ell} \bbp f\|_{L_{\omega, \nu}^2} \| \partial_\theta^k (\bbi - \bbp) f\|_{L_{\mu}^{2}}  \leq  \sqrt{\frac{m}{\sigma}} \Big(|\partial_\theta^{\ell} f_0 | + |\partial_\theta^{\ell} f_1| \Big) \| \partial_\theta^k (\bbi - \bbp) f\|_{L_{\mu}^{2}}.
\end{aligned}
\end{align}
We now combine \eqref{F-23} and \eqref{F-24} to obtain
\begin{align}
\begin{aligned} \label{F-25}
| {\mathcal I}_{52}^2|  &\lesssim
\frac {\kappa}{\sqrt{m \sigma}} \sum_{0 \leq \ell < k} \|\partial_\theta^{k-\ell} \mathcal{S}[\sqrt{M} f]\|_{L^\infty} \Big( \|\partial_\theta^{\ell}  f_0  \|_{L^2} + \|\partial_\theta^{\ell} f_1 \|_{L^2} \Big)  \|  \partial_\theta^k (\bbi - \bbp) f\|_{\mu}  \\
& \lesssim \frac {\kappa}{\sqrt{m \sigma}} \left(\sum_{0 \leq \ell < k} \|\partial_\theta^{\ell}( f_0 , f_1 )\|_{L^2}^2 +  \|  \partial_\theta^k (\bbi - \bbp) f\|_{\mu}^2 \right).
\end{aligned}
\end{align}

\noindent  Similar to the estimate of ${\mathcal I}_{52}^2$, we have
\begin{equation} \label{F-26}
 |{\mathcal I}_{52}^3|  \lesssim \frac {\kappa}{\sqrt{m \sigma}} \left(\sum_{0 \leq \ell < k} \|\partial_\theta^{\ell} (\bbi - \bbp) f\|_{\mu}^2 +  \|  \partial_\theta^k (f_0, f_1 )\|_{L^2}^2 \right).
\end{equation}

\noindent For the estimate of ${\mathcal I}_{52}^4$,
we use Lemma \ref{L6.1} to see that
\begin{align*}
\begin{aligned} 
 & \left|\left\langle\partial_\theta^{ \ell}\partial_\omega \bbp f, \partial_\theta^{k} \bbp f \right\rangle\right| \leq \|\partial_\theta^{\ell}\partial_\omega \bbp f\|_{L^2} \| \partial_\theta^{k} \bbp f\|_{L^2} \quad \mbox{and} \\
 & \| \partial_\theta^{\ell} \partial_{\omega} \bbp f\|_{L^2}   \lesssim \sqrt{\frac{m}{\sigma}} \left(|\partial_\theta^{\ell} f_0 | + |\partial_\theta^{\ell} f_1| \right).
\end{aligned}
\end{align*}
This yields
\begin{equation} \label{F-29}
 |{\mathcal I}_{52}^4|  \lesssim \frac {\kappa}{\sqrt{m \sigma}} \sum_{0 \leq \ell < k} \|\partial_\theta^{\ell}(f_0, f_1 )\|_{L^2}^2.
\end{equation}
In \eqref{F-21}, we combine all estimates \eqref{F-22}, \eqref{F-25}, \eqref{F-26}, and \eqref{F-29} to get
\begin{equation*}
|{\mathcal I}_{52}| \lesssim \frac{\kappa}{\sqrt{m\sigma}} \sum_{0 \leq \ell \leq k} \left(\|\partial_\theta^\ell(f_0,f_1)\|_{L^2}^2 + \|\partial_\theta^\ell (\bbi - \bbp) f\|_{\mu}^2 \right).
\end{equation*}
We then collect all the above estimates to find
\begin{align*}
\begin{aligned}
&\frac{d}{dt}\|\partial_\theta^k f\|_{L^2}^2 + \frac{2\lambda_0}{m} \|\partial_\theta^k (\bbi - \bbp) f\|_{\mu}^2 + \frac2m \|\partial_\theta^k f_1\|_{L^2}^2\cr
&\hspace{3cm} \lesssim \frac{\kappa}{\sqrt{m\sigma}} \sum_{0 \leq \ell \leq k}  \left( \|\partial_\theta^\ell(f_0,f_1)\|_{L^2}^2 + \|\partial_\theta^\ell (\bbi - \bbp) f\|_{\mu}^2 \right) \quad \mbox{for $1 \leq k \leq s$.}
\end{aligned}
\end{align*}
Since $\int_\bbt f_0\,d\theta = 0$, we can use the Poincar\'e inequality, $\|f_0\|_{L^2} \lesssim \|\partial_\theta f_0\|_{L^2}$. This together with the assumption $\sigma \gg m\kappa^2$ concludes
\[ \frac{d}{dt}\|\partial_\theta^k f\|_{L^2}^2 + \frac{\lambda_0}{m} \|\partial_\theta^k (\bbi - \bbp) f\|_{\mu}^2 + \frac1m \|\partial_\theta^k f_1\|_{L^2}^2 
 \lesssim \frac{\kappa}{\sqrt{m\sigma}} \left(\|(\bbi - \bbp) f\|_{\mu}^2 + \|f_1\|_{L^2}^2 + \|\partial_\theta^k f_0\|_{L^2}^2 \right), \]
for $1 \leq k \leq s$.
\end{proof}
To estimate the R.H.S. in Lemmas \ref{L6.1} and \ref{L6.2}, we use the macro-micro decomposition. We derive the hyperbolic-parabolic system which the macro components $f_0$ and $f_1$ should satisfy. We multiply the equation \eqref{F-1}$_1$ by $\chi_0$ and $\chi_1$  respectively, and then integrate the resulting one with respect to $(\omega,\nu)$ to find
\begin{align}
\begin{aligned} \label{F-31}
& \partial_t f_0 + \sqrt{\frac{\sigma}{m}} \partial_{\theta} f_1 + \partial_{\theta} f_0 \int_{\mathbb{R}}\nu g(\nu)\,d\nu + \int_{\mathbb{R}^2} \nu \chi_0 \partial_{\theta} (\bbi - \bbp) f \,d\omega d\nu = 0, \\
& \partial_t f_1 +\sqrt{\frac{\sigma}{m}}  \partial_{\theta}f_0 + \partial_{\theta}f_1 \int_{\mathbb{R}}\nu g(\nu)\,d\nu  +\frac{1}{m} f_1 - \frac{1}{\sqrt{2\pi }} \frac{\kappa}{\sqrt{m \sigma}} \mathcal{S}[\sqrt{M} f] \\
& \hspace{1cm} -\frac{\kappa}{\sqrt{m\sigma}} \mathcal{S}[\sqrt{M} f]f_0    +  \int_{\mathbb{R}^2}\nu \chi_1 \partial_{\theta} (\bbi - \bbp) f \,d\omega d\nu\\
& \hspace{1cm}  + \int_{\mathbb{R}^2}  (\omega - \nu)\chi_1 \partial_{\theta} (\bbi - \bbp) f \,d\omega d\nu = 0.
\end{aligned}
\end{align}
Then, for $0\leq k\leq s-1$, we apply $\partial_{\theta}^{k}$ to equations \eqref{F-31} to obtain
\begin{align}
\begin{aligned} \label{F-32}
& \partial_t  \partial^{k}_{\theta}f_0 + \sqrt{\frac{\sigma}{m}} \partial^{k+1}_{\theta} f_1 + \partial^{k+1}_{\theta} f_0 \int_{\mathbb{R}}\nu g(\nu)\,d\nu   + \int_{\mathbb{R}^2} \nu \chi_0 \partial_{\theta}^{k+1} (\bbi - \bbp) f \,d\omega d\nu = 0, \\
& \partial_t  \partial^{k}_{\theta}f_1 +\sqrt{\frac{\sigma}{m}}  \partial^{k+1}_{\theta}f_0 +   \partial^{k+1}_{\theta}f_1 \int_{\mathbb{R}}\nu g(\nu)\,d\nu  + \frac{1}{m} \partial^{k}_{\theta}f_1  - \frac{1}{\sqrt{2\pi }} \frac{\kappa}{\sqrt{m \sigma}}  \partial^{k}_{\theta}\mathcal{S}[\sqrt{M} f]   \\
&\hspace{1cm} - \frac{\kappa}{\sqrt{m\sigma}} \partial^{k}_{\theta}(\mathcal{S}[\sqrt{M} f]f_0) +  \int_{\mathbb{R}^2} \nu\chi_1   \partial^{k+1}_{\theta}(\bbi - \bbp) f \,d\omega d\nu\\
&\hspace{1cm} +\int_{\mathbb{R}^2} (\omega - \nu)\chi_1  \partial^{k+1}_{\theta} (\bbi - \bbp) f \,d\omega d\nu = 0.
\end{aligned}
\end{align}
Note that the similar idea is used for the study of the collisional kinetic equations \cite{Guo04}, see also \cite{Kawa83} for the hyperbolic-parabolic dissipation argument.

\begin{lemma} \label{L6.4}
For $s \geq 1$ and $T >0$, let $f \in \mathcal{C}([0,T]; H^s(\mathbb{T} \times \mathbb{R}^2))$ be a solution to the equation \eqref{F-1}. Suppose that $\|g\|_{\nu} < \infty$ and $\sigma$ satisfies $\sigma > C \max\{\kappa, m \|g\|_{\nu}^{2} \}$ for sufficiently large constant $C>0$. Then, for $0 \leq k \leq s-1$, we have
\[ \frac{1}{\sqrt{m\sigma}} \frac{d}{dt}  \int_{\mathbb{T} }\partial^{k}_{\theta}f_1 \partial^{k+1}_{\theta}f_0 \,d\theta + \frac1m \|\partial^{k+1}_{\theta}f_0\|_{L^2}^2 \lesssim \frac1m \|\partial^{k+1}_{\theta}f_1\|_{L^2}^2 + \frac{1}{m^2 \sigma} \|\partial_{\theta}^{k}f_1\|_{L^2}^2 + \frac{1}{m}  \|\partial_{\theta}^{k+1} (\bbi - \bbp) f\|_{\mu}^2.
\]
\end{lemma}
\begin{proof} We multiply $\eqref{F-32}_2$ by $\partial^{k+1}_{\theta}f_0$ and integrate the resulting equation with respect to $\theta$  to obtain
\begin{align*}
\begin{aligned}
&\frac{d}{dt} \int_{\mathbb{T}} \partial^{k}_{\theta}f_1 \partial^{k+1}_{\theta}f_0 \,d\theta + \int_{\mathbb{T}} \partial^{k+1}_{\theta}f_1 \partial_{t}\partial^{k}_{\theta}f_0 \,d\theta + \sqrt{\frac{\sigma}{m}} \|\partial^{k+1}_{\theta}f_0\|_{L^2}^2  \\
& \hspace{1cm} + \int_{\mathbb{T} \times \bbr} \nu g(\nu)  \partial^{k+1}_{\theta}f_1\partial^{k+1}_{\theta}f_0 \,d\theta d\nu +\frac{1}{m}  \int_{\mathbb{T} }\partial^{k}_{\theta}f_1\partial^{k+1}_{\theta}f_0\,d\theta \\
& \hspace{1cm} - \frac{1}{\sqrt{2\pi }} \frac{\kappa}{\sqrt{m \sigma}} \int_{\mathbb{T}} \partial^{k}_{\theta}\mathcal{S}[\sqrt{M} f]\partial^{k+1}_{\theta}f_0 \,d\theta - \frac{\kappa}{\sqrt{m\sigma}} \int_{\mathbb{T}} \partial^{k}_{\theta}(\mathcal{S}[\sqrt{M} f]f_0) \partial^{k+1}_{\theta}f_0 \,d\theta  \\
& \hspace{1cm} + \int_{\mathbb{T} \times \bbr^2}  \nu \chi_1  \partial^{k+1}_{\theta} (\bbi - \bbp) f \partial^{k+1}_{\theta}f_0 \,d\theta d\omega d\nu   \\
& \hspace{1cm} + \int_{\mathbb{T} \times \bbr^2} (\omega - \nu) \chi_1  \partial^{k+1}_{\theta} (\bbi - \bbp) f \partial^{k+1}_{\theta}f_0 \,d\theta d\omega d\nu  = 0.
\end{aligned}
\end{align*}
Then, we further use $\eqref{F-32}_1$ to obtain
\begin{align}
\begin{aligned} \label{F-33}
&\frac{d}{dt}  \int_{\mathbb{T} }\partial^{k}_{\theta}f_1 \partial^{k+1}_{\theta}f_0 \,d\theta+\sqrt{\frac{\sigma}{m}}   (\|\partial^{k+1}_{\theta}f_0\|_{L^2}^2-\|\partial^{k+1}_{\theta}f_1\|_{L^2}^2)\\
& \hspace{1cm} = - \frac{1}{m}  \int_{\mathbb{T} }\partial^{k}_{\theta}f_1\partial^{k+1}_{\theta}f_0\,d\theta + \frac{1}{\sqrt{2\pi }} \frac{\kappa}{\sqrt{m \sigma}} \int_{\mathbb{T} } \partial^{k}_{\theta}\mathcal{S}[\sqrt{M} f]\partial^{k+1}_{\theta}f_0\,d\theta\\
&  \hspace{1cm}  + \frac{\kappa}{\sqrt{m\sigma}} \int_{\mathbb{T} }\partial^{k}_{\theta}(\mathcal{S}[\sqrt{M} f]f_0)\partial^{k+1}_{\theta}f_0\,d\theta  + \int_{\mathbb{T} \times \bbr^2} \nu \chi_0 \partial^{k+1}_{\theta} (\bbi - \bbp) f \partial^{k+1}_{\theta}f_1 \,d\omega d\nu d\theta \\
&  \hspace{1cm}  -  \int_{\mathbb{T} \times \bbr^2}  \nu \chi_1  \partial^{k+1}_{\theta} (\bbi - \bbp) f \partial^{k+1}_{\theta}f_0 \,d\theta d\omega d\nu \\
 &  \hspace{1cm}  - \int_{\mathbb{T} \times \bbr^2} (\omega - \nu) \chi_1  \partial^{k+1}_{\theta} (\bbi - \bbp) f \partial^{k+1}_{\theta}f_0 \,d\theta d\omega d\nu =: \sum_{i=1}^{6} {\mathcal I}_{6i}.
\end{aligned}
\end{align}
Below, we estimate the terms ${\mathcal I}_{6i}, i = 1, \cdots, 6$ in \eqref{F-33}. \newline

\noindent $\bullet$ (Estimate on ${\mathcal I}_{6i},~i=1,2,3$):  It follows from Lemma \ref{L6.1} $(i)$ that
\begin{align*}
\begin{aligned}
|{\mathcal I}_{61}| &\leq \frac1m \|\partial_{\theta}^{k} f_1\|_{L^2} \|\partial_{\theta}^{k+1} f_0\|_{L^2} \leq \frac{1}{8} \sqrt{\frac{\sigma}{m}} \|\partial^{k+1}_{\theta}f_0\|_{L^2}^2 + \frac{2}{m\sqrt{m\sigma}} \|\partial_{\theta}^{k}f_1\|_{L^2}^2, \\
|{\mathcal I}_{62}| &\leq \frac{1}{\sqrt{2\pi }} \frac{\kappa}{\sqrt{m \sigma}}\left| \int_{\mathbb{T} } \partial^{k}_{\theta}\mathcal{S}[\sqrt{M} f]\partial^{k+1}_{\theta}f_0\,d\theta\right| \lesssim  \frac{\kappa}{\sqrt{m \sigma}}\|f_0\|_{L^2}\|\partial_{\theta}^{k+1}f_0\|_{L^2} \\
& \lesssim \frac{\kappa}{\sqrt{m\sigma}}\sum_{\ell=0}^{k+1}\|\partial_{\theta}^{\ell}f_0\|_{L^2}^2, \\
|{\mathcal I}_{63}| &\leq \frac{\kappa}{\sqrt{m\sigma}} \left|\int_{\mathbb{T} }\partial^{k}_{\theta}(\mathcal{S}[\sqrt{M} f]f_0)\partial^{k+1}_{\theta}f_0\,d\theta\right|
\lesssim \frac{\kappa}{\sqrt{m\sigma}}\sum_{\ell=0}^{k+1}\|\partial_{\theta}^{\ell}f_0\|_{L^2}^2.
\end{aligned}
\end{align*}

\noindent $\bullet$ (Estimate of ${\mathcal I}_{6i},~i=4,5,6$):
We use H\"older and Young's inequalities to get
\begin{align*}
\begin{aligned}
|{\mathcal I}_{64}| &=  \left| \int_{\mathbb{T} }  \partial^{k+1}_{\theta}f_1 \Big( \int_{\mathbb{R}^2} \nu \chi_0  \partial^{k+1}_{\theta} (\bbi - \bbp)f\,d\omega d\nu \Big) d\theta \right| \\
& \leq \int_{\bbt} |\partial^{k+1}_{\theta}f_1| \Big(\int_{\bbr^2} \nu^2 \chi_0^2 \,d\omega d\nu\Big)^{\frac{1}{2}} \Big(\int_{\bbr^2}  |\partial^{k+1}_{\theta} (\bbi - \bbp)f|^2\,d\omega d\nu \Big)^{\frac{1}{2}} d\theta \\
& \leq  \|g(\nu)\|_{\nu} \int_{\mathbb{T}} |\partial^{k+1}_{\theta}f_1|  \|\partial_{\theta}^{k+1} (\bbi - \bbp) f\|_{L_{\mu}^2} \,d\theta \\
& \leq \|g(\nu)\|_{\nu} \|\partial_{\theta}^{k+1}f_1\|_{L^2} \|\partial_{\theta}^{k+1} (\bbi - \bbp) f\|_{\mu} \\
& \leq \frac{1}{8} \sqrt{\frac{\sigma}{m}} \|\partial^{k+1}_{\theta}f_1\|_{L^2}^2 + 2\sqrt{\frac{m}{\sigma}}  \|g\|_{\nu}^2  \|\partial_{\theta}^{k+1} (\bbi - \bbp) f\|_{\mu}^2.
\end{aligned}
\end{align*}
Similarly, we also find
\begin{align*}
\begin{aligned}
|{\mathcal I}_{65}|  \leq \frac{1}{8} \sqrt{\frac{\sigma}{m}} \|\partial^{k+1}_{\theta}f_0\|_{L^2}^2 + 2\sqrt{\frac{m}{\sigma}} \|g\|_{\nu}^2 \|\partial_{\theta}^{k+1}  (\bbi - \bbp) f\|_{\mu}^2.
\end{aligned}
\end{align*}
At last, we have
\begin{align*}
\begin{aligned}
|{\mathcal I}_{66}| &= \left| \int_{\mathbb{T} } \partial^{k+1}_{\theta}f_0 \Big(\int_{\mathbb{R}\times \mathbb{R}}  (\omega - \nu) \chi_1  \partial^{k+1}_{\theta} (\bbi - \bbp) f\,d\omega d\nu \Big) d\theta \right|\\
& \leq \sqrt{\frac{\sigma}{m}} \int_{\bbt} |\partial^{k+1}_{\theta}f_0| \Big(\int_{\bbr^2} \chi_1^2 \,d\omega d\nu\Big)^{\frac{1}{2}} \Big(\int_{\bbr^2}  | \frac{m}{\sigma}(\omega - \nu)^2 \partial^{k+1}_{\theta} (\bbi - \bbp)f|^2 \,d\omega d\nu\Big)^{\frac{1}{2}} d\theta \\
& \leq \sqrt{\frac{\sigma}{m}} \int_{\mathbb{T}}  |\partial^{k+1}_{\theta}f_0|   \| \partial_{\theta}^{k+1} (\bbi - \bbp) f\|_{L_\mu^2}\,d\theta  \leq  \sqrt{\frac{\sigma}{m}}\|\partial_{\theta}^{k+1}f_0\|_{L^2} \|\partial_{\theta}^{k+1} (\bbi - \bbp) f\|_{\mu} \\
&\leq \frac{1}{8} \sqrt{\frac{\sigma}{m}} \|\partial^{k+1}_{\theta}f_0\|_{L^2}^2 + 2\sqrt{\frac{\sigma}{m}} \|\partial_{\theta}^{k+1} (\bbi - \bbp) f\|_{\mu}^2.
\end{aligned}
\end{align*}
It follows from the above estimates that
\begin{align}
\begin{aligned} \label{F-34}
&\frac{d}{dt}  \int_{\mathbb{T} }\partial^{k}_{\theta}f_1 \partial^{k+1}_{\theta}f_0 \,d\theta + \sqrt{\frac{\sigma}{m}}   (\|\partial^{k+1}_{\theta}f_0\|_{L^2}^2 - 2\|\partial^{k+1}_{\theta}f_1\|_{L^2}^2) \\
& \hspace{1cm} \lesssim \frac{\kappa}{\sqrt{m\sigma}}\sum_{\ell=0}^{k+1}\|\partial_{\theta}^{\ell}f_0\|_{L^2}^2 + \frac{1}{m\sqrt{m\sigma}} \|\partial_{\theta}^{k}f_1\|_{L^2}^2 + \left( \sqrt{\frac{\sigma}{m}} + \sqrt{\frac{m}{\sigma}} \|g\|_{\nu}^2 \right) \|\partial_{\theta}^{k+1} (\bbi - \bbp) f\|_{\mu}^2.
\end{aligned}
\end{align}
Now, we multiply \eqref{F-34} by $1/\sqrt{m\sigma}$ to obtain
\begin{align*}
\begin{aligned}
&\frac{1}{\sqrt{m\sigma}} \frac{d}{dt}  \int_{\mathbb{T} }\partial^{k}_{\theta}f_1 \partial^{k+1}_{\theta}f_0 \,d\theta + \frac1m  (\|\partial^{k+1}_{\theta}f_0\|_{L^2}^2 - 2\|\partial^{k+1}_{\theta}f_1\|_{L^2}^2) \\
& \hspace{1cm} \lesssim \frac{\kappa}{m\sigma}\sum_{\ell=0}^{k+1}\|\partial_{\theta}^{\ell}f_0\|_{L^2}^2 + \frac{1}{m^2\sigma} \|\partial_{\theta}^{k}f_1\|_{L^2}^2 + \left( \frac{1}{m} + \frac{1}{\sigma} \|g\|_{\nu}^2 \right) \|\partial_{\theta}^{k+1} (\bbi - \bbp) f\|_{\mu}^2.
\end{aligned}
\end{align*}
We use Poincar\'e inequality and the assumption on $\sigma$:
\[ \sigma > C \max\{\kappa, m \|g\|_{\nu}^{2} \}, \]
to obtain that for any $0 \leq k \leq s$,
\[ \frac{1}{\sqrt{m\sigma}} \frac{d}{dt}  \int_{\mathbb{T} }\partial^{k}_{\theta}f_1 \partial^{k+1}_{\theta}f_0 \,d\theta + \frac1m \|\partial^{k+1}_{\theta}f_0\|_{L^2}^2 
 \lesssim \frac1m \|\partial^{k+1}_{\theta}f_1\|_{L^2}^2 + \frac{1}{m^2 \sigma} \|\partial_{\theta}^{k}f_1\|_{L^2}^2 + \frac{1}{m}  \|\partial_{\theta}^{k+1} (\bbi - \bbp) f\|_{\mu}^2.
\]
This yields the desired estimate.
\end{proof}

\begin{proposition} \label{P6.1}
For $s \geq 1$ and $T >0$, let $f \in \mathcal{C}([0,T]; H^s(\mathbb{T} \times \mathbb{R}^2))$ be a solution to the equation \eqref{F-1}. Suppose that $\|g\|_{\nu} < \infty$ and $\sigma$ satisfies $\sigma > C \max\{m\kappa^2, \kappa, m^{-1}, m \|g\|_{\nu}^{2} \}$ for sufficiently large constant $C>0$.  Then, for $0 \leq \ell \leq s-1$, there exist positive constants $C_{0}$ and $C_{0,1}$ such that
\begin{align} \label{e-theta}
\begin{aligned}
& \frac{d}{dt} \Big\{ C_{0}\|(\partial_{\theta}^{\ell}f, \partial_{\theta}^{\ell + 1}f)\|^2_{L^2} + \frac{1}{\sqrt{m\sigma}} \int_{\mathbb{T}} \partial_{\theta}^{\ell }f_1 \partial_{\theta}^{\ell  +1}f_0 \,d\theta \Big\} \\
& \hspace{1cm} + \frac{C_{0,1}}{m} \Big\{ \| \big( \partial_{\theta}^{\ell }(\bbi - \bbp)f, \partial_{\theta}^{\ell  +1} (\bbi - \bbp)f \big) \|_{\mu}^2 +  \|(\partial_{\theta}^{\ell }, \partial_{\theta}^{\ell +1}) (f_0, f_1)\|_{L^2}^{2} \Big\} \leq 0.
\end{aligned}
\end{align}
where
\[ C_{0}\|(\partial_{\theta}^{\ell}f, \partial_{\theta}^{\ell + 1}f)\|^2_{L^2} + \frac{1}{\sqrt{m\sigma}} \int_{\mathbb{T}} \partial_{\theta}^{\ell }f_1 \partial_{\theta}^{\ell  +1}f_0 \,d\theta\simeq \|(\partial_{\theta}^{\ell}f, \partial_{\theta}^{\ell + 1}f)\|^2_{L^2}.\]

\end{proposition}
\begin{proof}
It follows from Lemma \ref{L6.2}, Lemma \ref{L6.3} with $k=1$, and Lemma \ref{L6.4} with $k=0$ that
\begin{align}
\begin{aligned} \label{F-36}
 & \frac{d}{dt} \|f\|^2_{L^2} + \frac{1}{m} \| (\bbi - \bbp) f\|_{\mu}^2 + \frac{1}{m} \|f_1\|_{L^{2}}^{2} \lesssim \frac{\kappa}{\sqrt{m\sigma}} \|f_0\|_{L^2}^2, \\
 & \frac{d}{dt} \|\partial_{\theta}f\|^2_{L^2} + \frac{1}{m} \|\partial_{\theta} (\bbi - \bbp) f\|_{\mu}^2 + \frac{1}{m} \|\partial_{\theta}f_1\|_{L^{2}}^{2} \cr
 &\qquad \lesssim \frac{\kappa}{\sqrt{m\sigma}}\left(\|(\bbi - \bbp) f\|_{\mu}^2 + \|f_1\|_{L^2}^2 + \|\partial_{\theta}f_0\|_{L^2}^2 \right), \\
 &\frac{1}{\sqrt{m\sigma}} \frac{d}{dt} \int_{\mathbb{T}} f_1 \partial_{\theta}f_0 \,d\theta + \frac1m  \|\partial_{\theta}f_0\|_{L^2}^2\cr
 &\qquad  \lesssim \frac{1}{m}  \|\partial_{\theta} (\bbi - \bbp) f\|_{\mu}^2 + \frac1m \|\partial_{\theta}f_1\|_{L^2}^2 + \frac{1}{m^2\sigma}\| f_1\|_{L^2}^2.
\end{aligned}
\end{align}
Now, we multiply \eqref{F-36}$_1$ and \eqref{F-36}$_2$ with a constant $C_{0} >0$ large enough such that
\[ C_{0}\|(f, \partial_{\theta}f)\|^2_{L^2} + \frac{1}{\sqrt{m\sigma}} \int_{\mathbb{T}} f_1 \partial_{\theta}f_0\, d\theta \simeq \|(f, \partial_{\theta}f)\|^2_{L^2}, \]
then we add all the resulting relations to \eqref{F-36}$_3$ to get
\begin{align*}
\begin{aligned}
 & \frac{d}{dt} \left\{ C_{0}\|(f, \partial_{\theta}f)\|^2_{L^2} + \frac{1}{\sqrt{m\sigma}} \int_{\mathbb{T}} f_1 \partial_{\theta}f_0\, d\theta \right\} \\
 & \hspace{1cm} + \frac{C_{0} - C}{m}  \big\| \big( (\bbi - \bbp)f, \partial_{\theta} (\bbi - \bbp)f \big) \big\|_{\mu}^2  + \frac1m  \|\partial_{\theta}f_0\|_{L^2}^2 + \frac{C_{0}  -C}{m} \|(f_1, \partial_{\theta}f_1)\|_{L^{2}}^{2}  \\
 & \hspace{1cm} \leq \frac{C_0 C\kappa}{\sqrt{m\sigma}} \|(f_0, \partial_{\theta}f_0)\|_{L^2}^2,
\end{aligned}
\end{align*}
due to our assumption on $\sigma$. We next use Poincar\'e's inequality and $\sigma \gg m\kappa^2$ to obtain that there exists a constant $C_{0,1}>0$ such that

\begin{align}
\begin{aligned} \label{F-37}
& \frac{d}{dt} \left\{ C_{0}\|(f, \partial_{\theta}f)\|^2_{L^2} + \frac{1}{\sqrt{m\sigma}} \int_{\mathbb{T}} f_1 \partial_{\theta}f_0 \,d\theta \right\} \\
& \hspace{1cm} + \frac{C_{0,1}}{m} \Big\{  \big\| \big( (\bbi - \bbp)f, \partial_{\theta} (\bbi - \bbp)f \big) \big\|_{\mu}^2 +  \|(f_0, \partial_{\theta}f_0)\|_{L^2} + \|(f_1, \partial_{\theta}f_1)\|_{L^{2}}^{2}\Big\}   \leq 0.
\end{aligned}
\end{align}
Now we take Lemma \ref{L6.3} with $k =1 $ and $k=2$, Lemma \ref{L6.4} with $k=1$, and choose $C_{0}>0$ large enough to get
\begin{align*}
\begin{aligned}
 & \frac{d}{dt} \Big\{ C_{0}\|(\partial_{\theta}f, \partial_{\theta}^{2}f)\|^2_{L^2} + \frac{1}{\sqrt{m\sigma}} \int_{\mathbb{T}} \partial_{\theta}f_1 \partial_{\theta}^{2}f_0 d\theta \Big\} \\
 & \hspace{1cm} + \frac{C_{0} -C}{m}  \big\| \big( \partial_{\theta} (\bbi - \bbp)f, \partial_{\theta}^{2} (\bbi - \bbp) f\big) \big\|_{\mu}^2  + \frac1m  \|\partial_{\theta}^{2} f_0\|_{L^2}^2 + \frac{C_{0}  - C}{m} \|( \partial_{\theta}f_1, \partial_{\theta}^{2}f_1)\|_{L^{2}}^{2}  \\
 & \hspace{1cm} \leq \frac{C_{0}C\kappa}{\sqrt{m \sigma}} \|(\partial_{\theta}f_0, \partial_{\theta}^{2} f_0)\|_{L^2}^2.
\end{aligned}
\end{align*}
We use analogous estimates to \eqref{F-37} to get that there exists a positive constant $C_{0,1}$ such that
\begin{align*}
\begin{aligned}
& \frac{d}{dt} \Big\{ C_{0}\|(\partial_{\theta}f, \partial_{\theta}^{2}f)\|^2_{L^2} + \frac{1}{\sqrt{m\sigma}} \int_{\mathbb{T}} \partial_{\theta}f_1 \partial_{\theta}^{2}f_0 d\theta \Big\} \\
& \hspace{1cm} + \frac{C_{0,1}}{m} \Big\{ \big\| \big( \partial_{\theta}(\bbi - \bbp)f, \partial_{\theta}^{2} (\bbi - \bbp)f \big) \big\|_{\mu}^2 +  \|(\partial_{\theta}, \partial_{\theta}^{2}) (f_0, f_1)\|_{L^2}\Big\}   \leq 0.
\end{aligned}
\end{align*}
Repeating the above analysis, we have
\begin{align*}
\begin{aligned}
& \frac{d}{dt} \Big\{ C_{0}\|(\partial_{\theta}^{\ell}f, \partial_{\theta}^{\ell + 1}f)\|^2_{L^2} + \frac{1}{\sqrt{m\sigma}} \int_{\mathbb{T}} \partial_{\theta}^{\ell }f_1 \partial_{\theta}^{\ell  +1}f_0 \,d\theta \Big\} \\
& \hspace{1cm} + \frac{C_{0,1}}{m} \Big\{ \big\| \big( \partial_{\theta}^{\ell }(\bbi - \bbp)f, \partial_{\theta}^{\ell  +1} (\bbi - \bbp)f \big) \big\|_{\mu}^2 +  \|(\partial_{\theta}^{\ell }, \partial_{\theta}^{\ell +1}) (f_0, f_1)\|_{L^2}^{2} \Big\} \leq 0,\\
&C_{0}\|(\partial_{\theta}^{\ell}f, \partial_{\theta}^{\ell + 1}f)\|^2_{L^2} + \frac{1}{\sqrt{m\sigma}} \int_{\mathbb{T}} \partial_{\theta}^{\ell }f_1 \partial_{\theta}^{\ell  +1}f_0 \,d\theta\simeq \|(\partial_{\theta}^{\ell}f, \partial_{\theta}^{\ell + 1}f)\|^2_{L^2}, \quad \mbox{for $0 \leq \ell \leq s-1$}.
\end{aligned}
\end{align*}
This concludes our desired result.
\end{proof}

\subsection{Mixed-type estimates}
Next, we derive mixed type estimates. We apply $(\bbi - \bbp)$ to equation \eqref{F-1} and use Lemma \ref{L3.2} $(i)$ to get
\begin{align}
\begin{aligned} \label{F-39}
& \partial_t (\bbi - \bbp) f + \omega \partial_{\theta} (\bbi - \bbp) f  = {\mathcal L}_0 (\bbi - \bbp) f + \mathcal{N}(f, (\bbi - \bbp) f) \\
& \hspace{1cm} + \frac{\kappa}{2\sigma} \mathcal{S}[\sqrt{M}f][\omega- \nu, \bbp]f  - \frac{\kappa}{m} \mathcal{S}[\sqrt{M}f] [\partial_{\omega}, \bbp]f  + \partial_{\theta} [\bbp, \omega]f,
\end{aligned}
\end{align}
where $[\cdot, \cdot]$ denotes the commutator operator:
\[ [A, B]:= AB - BA. \]
Now we apply $\partial_{\theta}^i \partial_{\omega}^j$ with $1 \leq i+j \leq s$ to equation \eqref{F-39} to obtain
\begin{align}
\begin{aligned} \label{F-40}
& \partial_t \partial_{\theta}^i \partial_{\omega}^j (\bbi - \bbp) f  + \omega \partial_{\theta}^{i+1} \partial_{\omega}^j (\bbi - \bbp) f    \\
& \hspace{0.5cm} = -\partial_{\theta}^i [\partial_{\omega}^j, \omega \partial_{\theta}] (\bbi - \bbp) f + {\mathcal L}_0 \partial_{\theta}^i \partial_{\omega}^j (\bbi - \bbp) f  +  \partial_{\theta}^i [\partial_{\omega}^j, \mathcal{L}_0] (\bbi - \bbp) f  + \partial_{\theta}^i \partial_{\omega}^j \mathcal{N}(f, (\bbi - \bbp) f)  \\
& \hspace{0.5cm} + \frac{\kappa}{2\sigma} \partial_{\theta}^i \partial_{\omega}^j  \big( \mathcal{S}[\sqrt{M}f][\omega- \nu, \bbp]f\big)  - \frac{\kappa}{m} \partial_{\theta}^i \partial_{\omega}^j  \big( \mathcal{S}[\sqrt{M}f] [\partial_{\omega}, \bbp]f \big) + \partial_{\theta}^{i+1} \partial_{\omega}^j [\bbp, \omega]f.
\end{aligned}
\end{align}
Before proceeding the mixed type estimates, we first give some technical lemmas with respect to the communication operator $[\cdot , \cdot]$ and $\chi_0$, $\chi_1$ as follows.
\begin{lemma}\label{L5.6}
 Let $h = h(\theta, \omega)$ be any smooth function. Then, for $j \geq 1$, we have
\begin{align*}
\begin{aligned}
& (i)\quad  [\partial_{\omega}^j, \omega \partial_{\theta}] h = j \partial_{\theta} \partial_{\omega}^{j-1} h. \\
& (ii)\quad  [(\omega-\nu)^2, \partial_{\omega}^j] h
=  - j(j-1) \partial_{\omega}^{j-2}h - 2j (\omega - \nu) \partial_{\omega}^{j-1}h. \\
& (iii)\quad  \partial_{\theta}^{\ell}\partial_{\omega}^j  [\omega - \nu, \bbp] h = \partial_{\theta}^{\ell} h_1 \partial_{\omega}^{j}\big( (\omega - \nu)\chi_1 \big) - \sqrt{\frac{\sigma}{m}} \partial_{\theta}^{\ell} h_1 \partial_{\omega}^{j}  \chi_0 \\
& \hspace{3.4cm}-\big\langle \chi_1, (\omega - \nu)\partial_{\theta}^{\ell} (\bbi - \bbp)h \big\rangle \partial_{\omega}^{j}  \chi_1. \\
& (iv)\quad  \partial_{\theta}^{\ell}\partial_{\omega}^j  [\nu, \bbp] h = \partial_{\theta}^{\ell}h_0  \partial_{\omega}^{j} (\nu \chi_0) + \partial_{\theta}^{\ell}h_1  \partial_{\omega}^{j}(\nu \chi_1) \\
& \hspace{2.8cm} - \partial_{\theta}^{\ell} h_0  \partial_{\omega}^{j}\chi_0 \int_{\bbr} \nu g(\nu) d\nu - \partial_{\theta}^{\ell}h_1  \partial_{\omega}^{j}\chi_1 \int_{\bbr} \nu g(\nu) d\nu \\
&\hspace{2.8cm} - \langle\chi_0, \nu \partial_{\theta}^{\ell}(\bbi - \bbp) h \rangle  \partial_{\omega}^{j}\chi_0 - \langle\chi_1, \nu \partial_{\theta}^{\ell}(\bbi - \bbp) h \rangle  \partial_{\omega}^{j}\chi_1. \\
& (v) \quad \partial_{\theta}^{\ell}\partial_{\omega}^j  [\partial_{\omega}, \bbp] h = -\frac{m}{2\sigma}  \partial_{\theta}^{\ell} h_1 \partial_{\omega}^{j}\big( (\omega - \nu)\chi_1 \big) + \frac{1}{2} \sqrt{\frac{m}{\sigma}}  \partial_{\theta}^{\ell} h_1 \partial_{\omega}^{j} \chi_0 \\
& \hspace{2.8cm} - \big\langle \chi_0, \partial_{\omega} \partial_{\theta}^{\ell} (\bbi - \bbp)h \big\rangle \partial_{\omega}^{j}  \chi_0 -  \big\langle \chi_1, \partial_{\omega}\partial_{\theta}^{\ell} (\bbi - \bbp)h \big\rangle \partial_{\omega}^{j}  \chi_1.
\end{aligned}
\end{align*}
\end{lemma}
\begin{proof}  (i)~ We use the inductive method. For $j =1$, we have
\[ [\partial_{\omega}, \omega \partial_{\theta}]h = (\partial_\omega (\omega \partial_\theta) - \omega \partial_\theta \partial_\omega) h = \partial_\theta h, \quad \mbox{i.e.,} \quad
 [\partial_{\omega}, \omega \partial_{\theta}] = \partial_\theta. \]
Suppose that the relation $(i)$ holds for $j = \ell$, i.e., $[\partial_{\omega}^{\ell}, \omega \partial_{\theta}]h= \ell \partial_{\theta} \partial_{\omega}^{\ell-1}h$. Now we use the above relation to obtain
\begin{eqnarray*}
[\partial_{\omega}^{\ell+1}, \omega \partial_{\theta}]h &=& \partial_{\omega}^{\ell+1}(\omega \partial_{\theta}h) - \omega \partial_{\theta} \partial_{\omega}^{\ell+1}h = \partial_{\omega}([\partial_{\omega}^{\ell}, \omega \partial_{\theta}]h) + \partial_{\omega}(\omega \partial_{\theta} \partial_{\omega}^{\ell}h) - \omega \partial_{\theta}\partial_{\omega}^{\ell+1}h \\
&=& \partial_{\omega}(\ell \partial_{\theta} \partial_{\omega}^{\ell-1}h) + \partial_{\theta} \partial_{\omega}^{\ell}h = (\ell + 1) \partial_{\theta} \partial_{\omega}^{\ell}h
\end{eqnarray*}
which yields the desired estimate.  \newline

\noindent (ii)~By a direct calculation, we find
\begin{align}
\begin{aligned} \label{F-43}
\partial_{\omega}^j \big( (\omega - \nu)^2h \big) = j(j-1) \partial_{\omega}^{j-2} h + 2j (\omega - \nu) \partial_{\omega}^{j-1} h
+ (\omega - \nu)^2 \partial_{\omega}^{j} h, \quad j \geq 1.
\end{aligned}
\end{align}
 Thus, we use relation \eqref{F-43} to get
\begin{align*}
\begin{aligned}
[(\omega-\nu)^2, \partial_{\omega}^j]h=  -j(j-1) \partial_{\omega}^{j-2}h - 2j (\omega - \nu) \partial_{\omega}^{j-1}h, \quad \text{with}~ \partial_{\omega}^{-1}h = 0, ~ \partial_{\omega}^{0}h = h,
\end{aligned}
\end{align*}
for integer $j \geq 1$.
\newline

\noindent (iii)~ A straightforward computation gives
\[
(\omega - \nu) \bbp h =  \sqrt{\frac{\sigma}{m}} h_0 \chi_1 + h_1  (\omega- \nu) \chi_1.\,
\]
and
\begin{align*}
\begin{aligned}
\bbp \big((\omega - \nu)h \big)
&= \bbp_0 \big((\omega - \nu)h \big) + \bbp_1 \big((\omega - \nu) \bbp h \big) + \bbp_1 \big((\omega - \nu)(\bbi -\bbp)h \big) \\
& =  \sqrt{\frac{\sigma}{m}}h_1 \chi_0 +  \sqrt{\frac{\sigma}{m}} h_0 \chi_1 + \langle \chi_1, (\omega - \nu) (\bbi - \bbp)h \rangle \chi_1.
\end{aligned}
\end{align*}
Thus, we have
\begin{align*}
\begin{aligned}
[\omega - \nu, \bbp ]h = h_1 (\omega- \nu) \chi_1 - \sqrt{\frac{\sigma}{m}}h_1 \chi_0 - \langle \chi_1, (\omega - \nu) (\bbi - \bbp)h \rangle \chi_1.
\end{aligned}
\end{align*}
This yields
\begin{align*}
\begin{aligned}
\partial_{\theta}^{\ell} \partial_{\omega}^j [\omega - \nu, \bbp]h &= \partial_{\theta}^{\ell} h_1 \partial_{\omega}^{j}\big( (\omega - \nu)\chi_1 \big) - \sqrt{\frac{\sigma}{m}} \partial_{\theta}^{\ell} h_1 \partial_{\omega}^{j}  \chi_0 -\big\langle \chi_1, (\omega - \nu)\partial_{\theta}^{\ell} (\bbi - \bbp)h \big\rangle \partial_{\omega}^{j}  \chi_1.
\end{aligned}
\end{align*}
\noindent (iv)~ We use the following identities
\begin{align*}
\begin{aligned}
& \nu \bbp h = h_0 \nu \chi_0 + h_1 \nu \chi_1 \quad \mbox{and}  \\
& \bbp (\nu h) =  \bbp (\nu \bbp h) + \bbp (\nu (\bbi - \bbp) h)\\
& \hspace{0.9cm} = h_0 \chi_0 \int_{\bbr} \nu g(\nu) d\nu + h_1 \chi_1 \int_{\bbr} \nu g(\nu) d\nu + \langle\chi_0, \nu (\bbi - \bbp) h \rangle \chi_0 + \langle\chi_1, \nu (\bbi - \bbp) h \rangle \chi_1.
\end{aligned}
\end{align*}
to get
\begin{align*}
\begin{aligned}
[\nu, \bbp]h &= h_0 \nu \chi_0 + h_1 \nu \chi_1 - h_0 \chi_0 \int_{\bbr} \nu g(\nu)\,d\nu - h_1 \chi_1 \int_{\bbr} \nu g(\nu) \,d\nu \\
&- \langle\chi_0, \nu (\bbi - \bbp) h \rangle \chi_0 - \langle\chi_1, \nu (\bbi - \bbp) h \rangle \chi_1.
\end{aligned}
\end{align*}
Hence, for any integer $0 \leq \ell \leq s$ and $0 \leq j \leq s$, we can obtain
\begin{align*}
\begin{aligned}
\partial_{\theta}^{\ell} \partial_{\omega}^{j} [\nu, \bbp]h &= \partial_{\theta}^{\ell}h_0  \partial_{\omega}^{j} (\nu \chi_0) + \partial_{\theta}^{\ell}h_1  \partial_{\omega}^{j}(\nu \chi_1) - \partial_{\theta}^{\ell} h_0  \partial_{\omega}^{j}\chi_0 \int_{\bbr} \nu g(\nu)\, d\nu - \partial_{\theta}^{\ell}h_1  \partial_{\omega}^{j}\chi_1 \int_{\bbr} \nu g(\nu) \,d\nu \\
&- \langle\chi_0, \nu \partial_{\theta}^{\ell}(\bbi - \bbp) h \rangle  \partial_{\omega}^{j}\chi_0 - \langle\chi_1, \nu \partial_{\theta}^{\ell}(\bbi - \bbp) h \rangle  \partial_{\omega}^{j}\chi_1.
\end{aligned}
\end{align*}

\noindent (v)~We use the relation $\partial_{\omega}\chi_0 = - \sqrt{m/\sigma}\chi_1/2$ and $\partial_{\omega}\chi_1 = \sqrt{m/\sigma} \chi_0 - (\omega-\nu)(m/2\sigma) \chi_1$ to find
\begin{align*}
\begin{aligned}
[\partial_{\omega}, \bbp]h &= \partial_{\omega}(\bbp h) - \bbp (\partial_{\omega} \bbp h) - \bbp (\partial_{\omega} (\bbi - \bbp) h) \\
& = -\frac{m}{2\sigma} h_1   (\omega - \nu)\chi_1   + \frac{1}{2} \sqrt{\frac{m}{\sigma}}   h_1  \chi_0  - \big\langle \chi_0, \partial_{\omega} (\bbi - \bbp)h \big\rangle   \chi_0 -  \big\langle \chi_1, \partial_{\omega} (\bbi - \bbp)h \big\rangle \chi_1.
\end{aligned}
\end{align*}
Thus, we have
\begin{align*}
\begin{aligned}
 \partial_{\theta}^{\ell}\partial_{\omega}^j  [\partial_{\omega}, \bbp] h &= -\frac{m}{2\sigma}  \partial_{\theta}^{\ell} h_1 \partial_{\omega}^{j}\big( (\omega - \nu)\chi_1 \big) + \frac{1}{2} \sqrt{\frac{m}{\sigma}}  \partial_{\theta}^{\ell} h_1 \partial_{\omega}^{j} \chi_0 \\
&  - \big\langle \chi_0, \partial_{\omega} \partial_{\theta}^{\ell} (\bbi - \bbp)h \big\rangle \partial_{\omega}^{j}  \chi_0 -  \big\langle \chi_1, \partial_{\omega}\partial_{\theta}^{\ell} (\bbi - \bbp)h \big\rangle \partial_{\omega}^{j}  \chi_1.
\end{aligned}
\end{align*}
\end{proof}

\begin{lemma}\label{L5.5}
The following estimates hold.
\begin{align*}
\begin{aligned}
& \|\partial_\omega^j \chi_0\|_{L_{\omega, \nu}^2}^2 \lesssim  (\frac{m}{\sigma})^j, \quad  \|\partial_\omega^j \chi_1\|_{L_{\omega, \nu}^2}^2 \lesssim (\frac{m}{\sigma})^{j}, \quad \|\partial_\omega^j \big((\omega - \nu) \chi_1 \big)\|_{L_{\omega, \nu}^2}^2 \lesssim (\frac{m}{\sigma})^{j-1}, \\
& \|\partial_\omega^j ( \nu \chi_0 )\|_{L_{\omega, \nu}^2}^2 \lesssim (\frac{m}{\sigma})^{j} \|g\|_{\nu}^{2}, \quad \mbox{and} \quad \|\partial_\omega^j ( \nu \chi_1 )\|_{L_{\omega, \nu}^2}^2 \lesssim (\frac{m}{\sigma})^{j} \|g\|_{\nu}^{2}.
\end{aligned}
\end{align*}

\end{lemma}

\begin{proof}
We only provide  the first estimate, and the other two estimates can be treated similarly.  Direct calculation yield 
\begin{align*}
\begin{aligned}
\partial_{\omega}^j \sqrt{M} &= \mfa_1 \big(\frac{m}{\sigma}\big)^{k} (\omega - \nu) \sqrt{M} + \mfa_3 \big(\frac{m}{\sigma}\big)^{k+1} (\omega - \nu)^3 \sqrt{M} + \cdots  \\
& + \mfa_{2k-1} \big(\frac{m}{\sigma}\big)^{2k-1} (\omega - \nu)^{2k-1} \sqrt{M}, \quad \text{if} ~j = 2k - 1 ~\text{with}~ k \geq 1, \\
\partial_{\omega}^j \sqrt{M} &= \mfb_0 \big(\frac{m}{\sigma}\big)^{k} \sqrt{M} + \mfb_2 \big(\frac{m}{\sigma}\big)^{k+1} (\omega - \nu)^2 \sqrt{M} + \cdots  \\
& + \mfb_{2k} \big(\frac{m}{\sigma}\big)^{2k} (\omega - \nu)^{2k} \sqrt{M}, \quad \text{if}~ j = 2k ~\text{with}~ k \geq 1.
\end{aligned}
\end{align*}
Thus, if $j = 2k -1$, we have
\begin{align}
\begin{aligned} \label{F-41}
\|\partial_\omega^j \sqrt{M}\|_{L_{\omega, \nu}^2}^2 &= \int_{\bbr^2} \sqrt{2\pi} |\partial_{\omega}^j \sqrt{M}|^2 \,d\omega d\nu \\
& \lesssim \big(\frac{m}{\sigma}\big)^{2k} \int_{\bbr^2} (\omega - \nu)^2 M \,d\omega d\nu + \big(\frac{m}{\sigma}\big)^{2k+2} \int_{\bbr^2} (\omega - \nu)^6 M \,d\omega d\nu  + \cdots \\
&+ \big(\frac{m}{\sigma}\big)^{4k - 1} \int_{\bbr^2} (\omega - \nu)^{4k-2} M \,d\omega d\nu  \\
& \lesssim \big(\frac{m}{\sigma}\big)^{2k} \frac{\sigma}{m} +  \big(\frac{m}{\sigma}\big)^{2k+2}  \big(\frac{\sigma}{m}\big)^{3} + \cdots + \big(\frac{m}{\sigma}\big)^{4k-2}  \big(\frac{\sigma}{m}\big)^{2k-1} \lesssim \big(\frac{m}{\sigma}\big)^{2k-1},
\end{aligned}
\end{align}
where we used
\[  \int_{\bbr^2} (\omega - \nu)^{2\ell} M \,d\omega d\nu = \frac{(2\ell-1)!! }{2\pi} \big(\frac{\sigma}{m}\big)^{\ell}.  \]
Similarly, if $j = 2k$, we have
\begin{align} \label{F-41-1}
\begin{aligned}
\|\partial_\omega^j \sqrt{M}\|_{L_{\omega, \nu}^2}^2 & \lesssim \big(\frac{m}{\sigma}\big)^{2k}.
\end{aligned}
\end{align}
Hence, we obtain our desired estimate by \eqref{F-41} and \eqref{F-41-1}.
\end{proof}

In next lemma, we derive a Gronwall's inequality for $ \|\partial_{\theta}^i \partial_{\omega}^j (\bbi - \bbp) f\|_{L^2}^2$ as follows.

\begin{lemma} \label{L6.7}
For $s \geq 1$ and $T >0$, let $f \in \mathcal{C}([0,T]; H^s(\mathbb{T} \times \mathbb{R}^2))$ be a solution to the equation \eqref{F-1}. Suppose that $\|g\|_{\nu} < \infty$ and $\sigma$ satisfies $\sigma > C \max\{m\kappa^2, \kappa, m^{-1}, m \|g\|_{\nu}^{2} \}$ for sufficiently large constant $C>0$.Then, for $1 \leq j \leq s-i$, we have
$$\begin{aligned}
& \frac{d}{dt} \|\partial_{\theta}^i \partial_{\omega}^j (\bbi - \bbp) f\|_{L^2}^2 + \frac{1}{m} \|\partial_{\theta}^i \partial_{\omega}^j (\bbi - \bbp)f\|_{\mu}^2 \\
& \hspace{1cm} \lesssim  m\|\partial_{\theta}^{i+1} \partial_{\omega}^{j-1} (\bbi - \bbp) f\|^2_{\mu} + m^{2j-1} \Big\{ \|\partial_{\theta}^{i+1}(f_0, f_1)\|_{L^2}^2  + \|\partial_{\theta}^{i+1} (\bbi - \bbp)f\|_{\mu}^2 \Big\}\cr
& \hspace{1cm} + m\|\partial_\omega^{j-1} (\bbi - \bbp)f\|_{L^2}^2 + \frac{\kappa}{\sqrt{m\sigma}}\|\partial_\omega^j  (\bbi - \bbp)f\|_{\mu}^2 + m^{2j-1}  \left(\|f_1\|_{L^2}^2 + \| (\bbi - \bbp)f\|_{\mu}^2 \right).
\end{aligned}$$
\end{lemma}
\begin{proof}
Since this proof is rather lengthy and technical, we postpone it to Appendix \ref{app}.
\end{proof}
\begin{remark}\label{R5.1} Under our assumption on $\sigma$, we find from Lemma \ref{L6.7} that
$$\begin{aligned} 
& \frac{d}{dt} \|\partial_{\omega}^j (\bbi - \bbp) f\|_{L^2}^2 + \frac{1}{m} \|\partial_{\omega}^j (\bbi - \bbp)f\|_{\mu}^2 \\
& \hspace{1cm} \lesssim  m\left( \|\partial_{\theta} \partial_{\omega}^{j-1} (\bbi - \bbp) f\|^2_{\mu} + \|\partial_\omega^{j-1} (\bbi - \bbp)f\|_{L^2}^2\right) \cr
& \hspace{1cm} + m^{2j-1} \Big\{ \|f_1\|_{L^2}^2 + \|\partial_{\theta} (f_0, f_1)\|_{L^2}^2  +  \| (\bbi - \bbp)f\|_{\mu}^2 + \|\partial_{\theta} (\bbi - \bbp)f\|_{\mu}^2  \Big\},~~\mbox{for $1 \leq j \leq s$.}
\end{aligned}$$
\end{remark}
\begin{proposition} \label{P6.2}
For $s \geq 1$ and $T >0$, let $f \in \mathcal{C}([0,T]; H^s(\mathbb{T} \times \mathbb{R}^2))$ is a solution to the equation \eqref{F-1}. Suppose that $\|g\|_{\nu} < \infty$ and $\sigma$ satisfies $\sigma > C \max\{m\kappa^2, \kappa, m^{-1}, m \|g\|_{\nu}^{2} \}$ for sufficiently large constant $C>0$. Then, for $s \geq 1$, we have two energy functionals $\mathcal{E}_s(t)$ and $\mathcal{D}_s(t)$ such  that
$$\begin{aligned}
&\frac{d}{dt}\mathcal{E}_s(t)+\mathcal{D}_s(t)\leq 0, \quad\|f(t)\|_{H^s}^2\simeq\mathcal{E}_s(t)\lesssim \mathcal{D}_s(t). 
\end{aligned}$$
\end{proposition}

\begin{proof}
It follows from Lemma \ref{L6.7} with $i = 0$, $j =1$ and Proposition \ref{P6.1} with $\ell = 0$ that
\begin{align} \label{e6-71}
\begin{aligned}
& \frac{d}{dt} \|\partial_{\omega} (\bbi - \bbp) f\|_{L^2}^2 + \frac{1}{m} \|\partial_{\omega} (\bbi - \bbp)f\|_{\mu}^2\cr
& \hspace{0.5cm} \lesssim m \Big\{ \|\partial_{\theta}(f_0, f_1)\|_{L^2}^2  + \|\partial_{\theta} (\bbi - \bbp)f\|_{\mu}^2 \Big\} + m(\|f_1\|_{L^2}^2 + \|(\bbi - \bbp)f\|_{\mu}^2), \\
& \frac{d}{dt} \Big\{ C_{0}\|(f, \partial_{\theta}f)\|^2_{L^2} + \frac{1}{\sqrt{m\sigma}} \int_{\mathbb{T}}f_1 \partial_{\theta}f_0 \,d\theta \Big\} \\
& \hspace{0.5cm} + \frac{C_{0,1}}{m} \Big\{ \| \big( (\bbi - \bbp)f, \partial_{\theta} (\bbi - \bbp)f \big) \|_{\mu}^2 +  \|(f_0, \partial_{\theta}f_0)\|_{L^2}^{2} +  \|(f_1, \partial_{\theta}f_1)\|_{L^2}^{2}  \Big\} \leq 0,
\end{aligned}
\end{align}
for some positive constants $C_0$ and $C_{0,1}$. \newline

We multiply \eqref{e6-71}$_2$ by $C_{1,1}m^2$ with $C_{1,1} > 0$, then add \eqref{e6-71}$_1$ to obtain
\begin{align*}
\begin{aligned}
& \frac{d}{dt} \Big\{C_{1,1}m^2 \Big( C_{0} \|(f, \partial_{\theta}f)\|^2_{L^2} + \frac{1}{\sqrt{m\sigma}} \int_{\mathbb{T}}f_1 \partial_{\theta}f_0 \,d\theta \Big) +  \|\partial_{\omega} (\bbi - \bbp) f\|_{L^2}^2  \Big\} \\
& \hspace{0.5cm} + (C_{0,1} C_{1, 1} - C)m \Big\{ \big\| \big( (\bbi - \bbp)f, \partial_{\theta} (\bbi - \bbp)f \big) \big\|_{\mu}^2 +  \|(f_0, \partial_{\theta}f_0)\|_{L^2}^{2} +  \|(f_1, \partial_{\theta}f_1)\|_{L^2}^{2}  \Big\} \\
& \hspace{0.5cm}+ \frac{1}{m} \|\partial_{\omega} (\bbi - \bbp)f\|_{\mu}^2  \leq 0.
\end{aligned}
\end{align*}
We now choose $C_{1,1} > 0$ large enough so that
\[ C_{0,1}C_{1,1} - C > 0. \]
This yields that the above inequality can be rewritten as 
\begin{align}\label{e6-72}
\begin{aligned}
&\frac{d}{dt}\mathcal{E}_1(t)+\mathcal{D}_1(t)\leq 0, \quad\|f(t)\|_{H^1}^2\simeq\mathcal{E}_1(t)\lesssim \mathcal{D}_1(t),
\end{aligned}
\end{align}
where
\begin{align}
\begin{aligned}
& \mathcal{E}_1(t)=C_{1,1}m^2 \Big( C_{0}\sum_{k=0}^1\|\partial_{\theta}^k f\|^2_{L^2} + \frac{1}{\sqrt{m\sigma}} \int_{\mathbb{T}}f_1 \partial_{\theta}f_0 \,d\theta \Big) +  \|\partial_{\omega} (\bbi - \bbp) f\|_{L^2}^2  \nonumber \\
& \mathcal{D}_1(t)=\bar{C}_{1,1}m \Big\{ \sum_{k=0}^1\big\| \partial_{\theta}^k (\bbi - \bbp)f  \big\|_{\mu}^2 +  \sum_{k=0}^1 \|\partial_\theta^k(f_0, f_1)\|_{L^2}^{2}  \Big\}  + \frac1m\|\partial_{\omega} (\bbi - \bbp)f\|_{\mu}^2, \nonumber
\end{aligned}
\end{align}
for some $\bar{C}_{1,1} > 0$.

For the higher-order estimates, we claim that the following holds for $s \geq 2$:
\begin{align}\label{eq_ind}
\begin{aligned}
\frac{d}{dt}\mathcal{E}_s(t)+\mathcal{D}_s(t)\leq 0, \quad\|f(t)\|_{H^s}^2\simeq\mathcal{E}_s(t)\lesssim \mathcal{D}_s(t).
\end{aligned}
\end{align}
 For the proof, we use the strong inductive method. It follows from Remark \ref{R5.1} with $j=2$ and Lemma \ref{L6.7} with $i=1, j =1$ that
\begin{align} \label{e6-73}
\begin{aligned}
& \frac{d}{dt} \|\partial_{\omega}^{2} (\bbi - \bbp) f\|_{L^2}^2 + \frac{1}{m} \|\partial_{\omega}^{2} (\bbi - \bbp)f\|_{\mu}^2  \lesssim m\left( \|\partial_{\omega} (\bbi - \bbp)f\|_{\mu}^2 + \|\partial_{\theta} \partial_{\omega} (\bbi - \bbp)f\|_{\mu}^2\right)\cr
& \hspace{2.5cm} + m^3 \Big\{ \|(f_0,f_1),\partial_{\theta}(f_0, f_1)\|_{L^2}^2  + \|((\bbi - \bbp)f, \partial_{\theta} (\bbi - \bbp)f)\|_{\mu}^2 \Big\}
\end{aligned}
\end{align}
and
\begin{align} \label{e6-74}
\begin{aligned}
& \frac{d}{dt} \|\partial_{\theta} \partial_{\omega} (\bbi - \bbp) f\|_{L^2}^2 + \frac{1}{m} \|\partial_{\theta} \partial_{\omega} (\bbi - \bbp)f\|_{\mu}^2 \cr
& \hspace{0.5cm} \lesssim m \Big\{ \|\partial_{\theta}^{2}(f_0, f_1)\|_{L^2}^2  + \|\partial_{\theta}^{2} (\bbi - \bbp)f\|_{\mu}^2 + \|f_1\|_{L^2}^2 + \|(\bbi - \bbp)f\|_{\mu}^2\Big\} + \frac1m\|\partial_\omega(\bbi - \bbp)f\|_{\mu}^2.
\end{aligned}
\end{align}
We multiply \eqref{e-theta} for $\ell=1$, \eqref{e6-74} and \eqref{e6-72} by $C_{2,0} m^4$, $C_{2,1} m^2$ and $C_{2,2} m^2$ with $C_{2,0}, C_{2,1}, C_{2,2} > 0$, respectively, then add the two resulting estimates together with \eqref{e6-73}, we have
$$\begin{aligned}
& \frac{d}{dt} \Big\{C_{2,0} m^4\Big(C_{0}\|(\partial_{\theta} f, \partial_{\theta}^{2}f)\|^2_{L^2} + \frac{1}{\sqrt{m\sigma}} \int_{\mathbb{T}} \partial_{\theta}f_1 \partial_{\theta}^2f_0 \,d\theta \Big)\\
& \hspace{0.5cm}+ C_{2,1}m^2 \|\partial_{\theta} \partial_{\omega} (\bbi - \bbp) f\|_{L^2}^2 + \|\partial_\omega^2 (\bbi - \bbp) f\|_{L^2}^2 + C_{2,2}m^2 \mathcal{E}_1(t) \Big\} \\
& \hspace{0.5cm} + \frac{1}{m} \|\partial_{\omega}^{2} (\bbi - \bbp)f\|_{\mu}^2 + (C_{2,1} - C)m \|\partial_{\theta} \partial_{\omega} (\bbi - \bbp)f\|_{\mu}^2 + (C_{2,2}- CC_{2,1} - C) m \|\partial_{\omega} (\bbi - \bbp)f\|_{\mu}^2\cr
& \hspace{0.5cm} + (C_{2,0} C_{0,1}-C C_{2,1}) m^3 \sum_{k=1}^2\Big\{ \big\| \partial_{\theta}^k (\bbi - \bbp)f  \big\|_{\mu}^2 +  \|\partial_\theta^k(f_0, f_1)\|_{L^2}^{2} \Big\}   \cr
& \hspace{0.5cm} + (C_{2,2}\bar{C}_{1,1}-C C_{2,1} - C) m^3 \sum_{k=0}^1\Big\{ \big\| \partial_{\theta}^k (\bbi - \bbp)f  \big\|_{\mu}^2 +   \|\partial_\theta^k(f_0, f_1)\|_{L^2}^{2} \Big\}    \leq 0.
\end{aligned}$$
We next choose $C_{2,1} > 0$ large enough so that 
 \[ C_{2,1} - C > 0, \]
and then select $C_{2,0}>0$ and $C_{2,2} > 0$ sufficiently large so that 
 \[ C_{2,0} C_{0,1}-C C_{2,1} >0, \quad C_{2,2}- CC_{2,1} - C >0, \quad C_{2,2}\bar{C}_{1,1}-C C_{2,1} - C >0,  \]
and
$$\frac{d}{dt}\mathcal{E}_2(t)+\mathcal{D}_2(t)\leq 0, \quad\|f(t)\|_{H^2}^2\simeq\mathcal{E}_2(t)\lesssim \mathcal{D}_2(t),$$
where
$$\begin{aligned}
&\mathcal{E}_2(t)= C_{2,0} m^4\Big(C_{0}\|(\partial_{\theta} f, \partial_{\theta}^{2}f)\|^2_{L^2} + \frac{1}{\sqrt{m\sigma}} \int_{\mathbb{T}} \partial_{\theta}f_1 \partial_{\theta}^2f_0 \,d\theta \Big)\\
& \hspace{1cm}+ C_{2,1}m^2 \|\partial_{\theta} \partial_{\omega} (\bbi - \bbp) f\|_{L^2}^2 + \|\partial_\omega^2 (\bbi - \bbp) f\|_{L^2}^2 + C_{2,2}m^2 \mathcal{E}_1(t)  \cr
& \mathcal{D}_2(t)=\bar C_{2,0} m^3 \sum_{k=1}^2\Big\{ \big\| \partial_{\theta}^k (\bbi - \bbp)f  \big\|_{\mu}^2 +  \|\partial_\theta^k(f_0, f_1)\|_{L^2}^{2} \Big\} + \bar C_{2,1} m \sum_{k=0}^1\|\partial_{\theta}^k \partial_{\omega} f\|_{\mu}^2 \\
& \hspace{1cm}+ \bar C_{2,2} m^3  \sum_{k=0}^1 \Big\{\left(\|\partial_{\theta}^k (\bbi - \bbp)f\|_{\mu}^2 +\|\partial_\theta^k (f_0, f_1)\|_{L^2}^2\right)\Big\}+  \frac{1}{m} \|\partial_{\omega}^{2} (\bbi - \bbp) f\|_{\mu}^2,
\end{aligned}$$
for some positive constants $\bar C_{2,0}, \bar C_{2,1}$ and $\bar C_{2,2}$, which are independent of $m$.
  Thus we have the inequality \eqref{eq_ind} with $s=2$. Suppose now that \eqref{eq_ind} holds for all $1 \leq \ell \leq s-1$:
\begin{align}\label{eq_ind2}
\begin{aligned}
\frac{d}{dt}\mathcal{E}_{\ell}(t)+\mathcal{D}_{\ell}(t)\leq 0, \quad\|f(t)\|_{H^{\ell}}^2\simeq\mathcal{E}_{\ell}(t)\lesssim \mathcal{D}_{\ell}(t).
\end{aligned}
\end{align}
It follows from Lemma \ref{L6.7} that 
$$\begin{aligned}
&\frac{d}{dt}\left( \sum_{k=0}^\ell c_{\ell+1, k} m^{2k}\|\partial_\theta^k \partial_\omega^{\ell+1-k} (\bbi - \bbp)f\|_{L^2}^2 \right) + \sum_{k=0}^{\ell} c_{\ell+1, k} m^{2k-1} \|\partial_\theta^k \partial_\omega^{\ell+1 - k} (\bbi - \bbp)f\|_\mu^2 \cr
&\quad \leq C\sum_{k=0}^\ell c_{\ell+1, k} m^{2k+1} \|\partial_\theta^{k+1} \partial_\omega^{\ell - k} (\bbi - \bbp)f\|_\mu^2 + C\sum_{k=0}^\ell c_{\ell+1, k} m^{2\ell +1}\left(\|\partial_\theta^{k+1}(f_0, f_1)\|_{L^2}^2 + \|\partial_\theta^{k+1}(\bbi - \bbp)f\|_\mu^2 \right) \cr
&\quad + C\sum_{k=0}^\ell c_{\ell+1, k} m^{2k+1}\|\partial_\omega^{\ell - k}(\bbi - \bbp)f\|_{L^2}^2+ C \frac{\kappa}{\sqrt{m\sigma}} \sum_{k=0}^\ell   c_{\ell+1, k} m^{2k}\|\partial_\omega^{\ell +1- k}(\bbi - \bbp)f\|_\mu^2\cr
&\quad + Cm^{2\ell+1}\left( \|f_1\|_{L^2}^2 + \|(\bbi - \bbp)f\|_\mu^2 \right)\cr
&\quad =: \sum_{i=1}^5 \mathcal{I}_{7i},
\end{aligned}$$
where $c_{\ell+1, k},k=0,1,\cdots,\ell$ are positive constants which will be determined later. Note that $\mathcal{I}_{71}$ can be estimated as
$$\begin{aligned}
\mathcal{I}_{71} &= C\sum_{k=1}^{\ell+1} c_{\ell+1, k-1} m^{2k-1} \|\partial_\theta^k \partial_\omega^{\ell+1 - k} (\bbi - \bbp)f\|_\mu^2\cr
&=C\sum_{k=1}^{\ell} c_{\ell +1,k-1}m^{2k-1} \|\partial_\theta^k \partial_\omega^{\ell+1 - k} (\bbi - \bbp)f\|_\mu^2 + Cc_{\ell+ 1,\ell} m^{2\ell + 1}\|\partial_\theta^{\ell+1} (\bbi - \bbp)f\|_\mu^2\cr
&\leq C\sum_{k=1}^{\ell} c_{\ell + 1,k-1}m^{2k-1} \|\partial_\theta^k \partial_\omega^{\ell+1 - k} (\bbi - \bbp) f\|_\mu^2 + C\,\mathcal{I}_{72},
\end{aligned}$$
Similarly, we also find
 $$\begin{aligned}
\mathcal{I}_{73} + \mathcal{I}_{74} &\leq C\sum_{k=0}^{\ell-1}(c_{\ell+1, k}  m^{2k+1}\|\partial_\omega^{\ell-k} (\bbi - \bbp)f\|_\mu^2  + Cc_{\ell +1, \ell} m^{2\ell+1} \|(\bbi - \bbp)f\|_\mu^2 \\
&+ C  \sum_{k=1}^{\ell} (c_{\ell+1, k}  m^{2k -1}\|\partial_\omega^{\ell-k} (\bbi - \bbp)f\|_\mu^2  + C c_{\ell + 1, 0} \frac{\kappa}{\sqrt{m\sigma}} \|\partial_\omega^{\ell+1} f\|_\mu^2  \\
&\leq C\sum_{k=0}^{\ell-1}(c_{\ell+1, k} + c_{\ell + 1,k+1}) m^{2k+1} \|\partial_\omega^{\ell-k} (\bbi - \bbp) f\|_\mu^2 + C c_{\ell +1,0} \frac{\kappa}{\sqrt{m\sigma}} \|\partial_\omega^{\ell+1} f\|_\mu^2  + C\mathcal{I}_{75} .
\end{aligned}$$
We combine the above observations to yield
\begin{align}\label{eq_ind3}
\begin{aligned}
&\frac{d}{dt}\left( \sum_{k=0}^\ell c_{\ell + 1, k} m^{2k}\|\partial_\theta^k \partial_\omega^{\ell+1-k} (\bbi - \bbp)f\|_{L^2}^2 \right) + m^{-1}(c_{\ell + 1,0}  - C c_{\ell +1,0} \frac{\kappa m}{\sqrt{m\sigma}})\|\partial_\omega^{\ell+1} f\|_\mu^2 \cr
&\quad  + \sum_{k=0}^{\ell} (c_{\ell +1, k} - Cc_{\ell + 1,k-1}) m^{2k-1} \|\partial_\theta^k \partial_\omega^{\ell+1 - k} (\bbi - \bbp)f\|_\mu^2  \cr
&\quad \leq   C\sum_{k=0}^\ell m^{2\ell+1} \left(\|\partial_\theta^{k}(f_0, f_1)\|_{L^2}^2 + \|\partial_\theta^{k}(\bbi - \bbp)f\|_\mu^2 \right) + C\sum_{k=1}^{\ell} m^{2(\ell-k)+1} \|\partial_\omega^k (\bbi - \bbp)f\|_\mu^2.
\end{aligned}
\end{align}
We can choose 
\[ c_{\ell +1. k} - Cc_{\ell + 1, k-1}> 0 \quad \mbox{for all $k = 1,\cdots,\ell$}, \]
and by assumption $\sigma \geq C m\kappa^2$ for large $C$
$$c_{\ell + 1,0}  - C c_{\ell +1,0} \frac{\kappa m}{\sqrt{m\sigma}} > 0.$$ 
Similar to the derivation of \eqref{eq_ind2}, we use proper linear combination of \eqref{e-theta}, \eqref{eq_ind2} and \eqref{eq_ind3} to obtain the desired estimate:
\begin{align}\label{eq_ind4}
\begin{aligned}
\frac{d}{dt}\mathcal{E}_{\ell+1}(t)+\mathcal{D}_{\ell+1}(t)\leq 0, \quad\|f(t)\|_{H^{\ell+1}}^2\simeq\mathcal{E}_{\ell+1}(t)\lesssim \mathcal{D}_{\ell+1}(t), \quad \mbox{ for $1 \leq \ell \leq s-1$.}
\end{aligned}
\end{align}
\end{proof}
Now we are ready to provide the details of proof of Theorem \ref{T-I}.
\begin{proof}[Proof of Theorem \ref{T-I}]
From Proposition \ref{P6.2}, we have
$$\frac{d}{dt}\mathcal{E}_s(t)+\mathcal{D}_s(t)\leq 0, \quad\|f(t)\|_{H^s}^2\simeq\mathcal{E}_s(t)\lesssim \mathcal{D}_s(t).$$
Thus, we have
\[
\frac{d\mathcal{E}_s(t)}{dt} + C\mathcal{E}_s(t) \leq 0, \quad \mbox{i.e.,} \quad \mathcal{E}_s(f(t)) \leq \mathcal{E}_s(f^{in}) e^{-Ct},
\]
for $t \in [0,T]$, where $C>0$ is independent of $t$. Hence we have
\begin{equation}\label{res_final}
\|f(t)\|_{H^s} \leq C\|f^{in}\|_{H^s}e^{-Ct} \quad \mbox{for} \quad t \in [0,T],
\end{equation}
where $C>0$ is independent of $t$. Finally, we use the continuity argument together with the above uniform-in-time a priori estimates and local-in-time existence and uniqueness results obtained in Theorem \ref{T4.1} to have the global-in-time existence and uniqueness of strong solutions to the equation \eqref{F-1}. This subsequently implies that the inequality \eqref{res_final} holds for all $t \geq 0$. This yields the desired result.
\end{proof}

%
%
%
%

\section{Conclusion} \label{sec:6}

In this paper, we presented the global-in-time existence of the unique strong solution to the Kuramoto-Sakaguchi equation with inertia. For the equation, we considered the perturbation equation around the equilibrium and show that the regularity of the initial data is preserved in time if the strength of noise is strong enough. The main methods used in the current paper are the classical energy estimates combined with macro-micro decomposition based the hyperbolic-parabolic dissipation arguments. We also provided the large-time behavior of solutions showing the convergence to the equilibrium exponentially fast as time goes to infinity. It is worth mentioning that we do not need to have the initial data close to the equilibrium.
There are several interesting open questions for the Kuramoto-Sakaguchi equation, e.g., asymptotic stability and instability of the incoherent state depending on the coupling strength. This will be pursued in a future work.

%
%
%
%

\appendix
\section{Proof of Lemma \ref{L6.7}} \label{app}

We multiply equation \eqref{F-40} by $\partial_{\theta}^i \partial_{\omega}^j (\bbi - \bbp)f$, and then integrate the resulting relation over $\mathbb{T} \times \mathbb{R}^2$ to obtain
\begin{align}\label{F-45}
\begin{aligned}
& \frac{1}{2}\frac{d}{dt} \|\partial_{\theta}^i \partial_{\omega}^j (\bbi - \bbp) f\|_{L^2}^2 + \int_{\mathbb{T}} \langle -{\mathcal L}_0 \partial_{\theta}^i \partial_{\omega}^j (\bbi - \bbp) f, \partial_{\theta}^i \partial_{\omega}^j (\bbi - \bbp) f \rangle \,d\theta \\
& \hspace{0.5cm} = \int_{\mathbb{T}} \big\langle -\partial_{\theta}^i [\partial_{\omega}^j, \omega \partial_{\theta}] (\bbi - \bbp) f, \partial_{\theta}^i \partial_{\omega}^j (\bbi - \bbp)f \big\rangle \,d\theta  \\
& \hspace{0.5cm}+ \int_{\mathbb{T}} \big\langle \partial_{\theta}^i[\partial_{\omega}^j, \mathcal{L}_0]  (\bbi - \bbp) f, \partial_{\theta}^i \partial_{\omega}^j (\bbi - \bbp)f \big\rangle \,d\theta \\
&  \hspace{0.5cm} + \int_{\mathbb{T}} \big\langle \partial_{\theta}^i \partial_{\omega}^j \mathcal{N}(f, (\bbi - \bbp) f), \partial_{\theta}^i \partial_{\omega}^j (\bbi - \bbp)f \big\rangle \,d\theta \\
& \hspace{0.5cm} + \frac{\kappa}{2\sigma} \int_{\mathbb{T}} \big \langle \partial_{\theta}^i \partial_{\omega}^j \big(\mathcal{S}[\sqrt{M}f][\omega- \nu, \bbp]f \big), \partial_{\theta}^i \partial_{\omega}^j (\bbi - \bbp)f \big \rangle \,d\theta \\
& \hspace{0.5cm} - \frac{\kappa}{m}\int_{\mathbb{T}} \big\langle \partial_{\theta}^i \partial_{\omega}^j \big( \mathcal{S}[\sqrt{M}f] [\partial_{\omega}, \bbp]f \big), \partial_{\theta}^i \partial_{\omega}^j (\bbi - \bbp)f \big\rangle \,d\theta \\
& \hspace{0.5cm}+ \int_{\mathbb{T}} \big\langle \partial_{\theta}^{i+1} \partial_{\omega}^j [\bbp, \omega]f, \partial_{\theta}^i \partial_{\omega}^j (\bbi - \bbp)f \big\rangle \,d\theta \\
& \hspace{0.5cm} =: \sum_{i=1}^6 \mci_{8i},
\end{aligned}
\end{align}
To estimate the  second term in L.H.S. of  \eqref{F-45}, we use Lemma \ref{L3.2} $(ii)$ to obtain
\begin{align}
\begin{aligned} \label{F-45-1}
& \int_{\mathbb{T}} \big\langle - {\mathcal L}_0 \partial_{\theta}^i \partial_{\omega}^j (\bbi - \bbp) f, \partial_{\theta}^i \partial_{\omega}^j (\bbi - \bbp)f \big\rangle \,d\theta \\
& \hspace{1cm} \geq \frac{\lambda_0}{m} \|\{\bbi- \bbp_0\} \partial_{\theta}^i \partial_{\omega}^j (\bbi - \bbp)f\|_{\mu}^2 \geq  \frac{\lambda_0}{m} \|\partial_{\theta}^i \partial_{\omega}^j  (\bbi - \bbp)f\|_{\mu}^2 - \frac{\lambda_0}{m} \|\bbp_0\partial_{\theta}^i \partial_{\omega}^j (\bbi - \bbp)f\|_{\mu}^2.
\end{aligned}
\end{align}
For the second R.H.S. term in  relation \eqref{F-45-1}, we note that
\begin{align*} 
\begin{aligned}
& \frac{\lambda_0}{m} \|\bbp_0 \partial_{\theta}^i \partial_{\omega}^j (\bbi - \bbp)f\|_{\mu}^2 \\
& \hspace{0.5cm}= \frac{\lambda_0}{m} \int_{\bbt} \int_{\bbr^2} \Big( (1 + \frac{m}{\sigma}(\omega - \nu)^2)|\bbp_0 \partial_{\theta}^i \partial_{\omega}^j (\bbi - \bbp)f|^2 + \frac{\sigma}{m} |\partial_\omega \big( \bbp_0 \partial_{\theta}^i \partial_{\omega}^j (\bbi - \bbp)f\big)|^2 \Big) d\omega d\nu d\theta.
\end{aligned}
\end{align*}
We use Lemma \ref{L5.5} to get
\begin{align}
\begin{aligned} \label{F-46}
|\bbp_0 \partial_\theta^i  \partial_\omega^j (\bbi - \bbp)f| &= \big|\langle \chi_0, \partial_\omega^j \partial_\theta^i (\bbi - \bbp)f \rangle  \chi_0 \big| = \big|\langle  \partial_\omega^j \chi_0, \partial_\theta^i (\bbi - \bbp)f \rangle \chi_0 \big| \\
&\leq \|\partial_\omega^{j} \chi_0\|_{L_{\omega, \nu}^2} \|\partial_\theta^i (\bbi - \bbp)f\|_{L_{\omega, \nu}^2} |\chi_0|   \lesssim \big(\frac{m}{\sigma}\big)^{\frac{j}{2}} \|\partial_\theta^i (\bbi - \bbp)f\|_{L_{\omega, \nu}^2} |\chi_0|.
\end{aligned}
\end{align}
Thus we obtain
\begin{align}\label{F-46-1}
\begin{aligned}
& \int_{\bbr^2} \Big( (1 + \frac{m}{\sigma}(\omega - \nu)^2)|\bbp_0 \partial_{\theta}^i \partial_{\omega}^j (\bbi - \bbp)f|^2 \,d\omega d\nu \\
& \hspace{0.5cm} \lesssim \big(\frac{m}{\sigma}\big)^{j} \|\partial_{\theta}^{i} (\bbi -\bbp)f\|_{L_{\omega, \nu}^2}^2 \int_{\bbr^2} \big(1 + \frac{m}{\sigma}(\omega - \nu)^2\big) \chi_0^2 \,d\omega d\nu = 2 \big(\frac{m}{\sigma}\big)^{j}\|\partial_{\theta}^{i} (\bbi -\bbp)f\|_{L_{\omega, \nu}^2}^2.
\end{aligned}
\end{align}
Note that
\begin{eqnarray*}
\partial_{\omega }\big(\bbp_0 \partial_{\omega}^j \partial_\theta^i (\bbi - \bbp)f \big) = \langle\chi_0, \partial_\omega^j \partial_\theta^i (\bbi - \bbp)f \rangle \partial_{\omega} \chi_0 = (-1)^{j+1} \frac{1}{2} \sqrt{\frac{m}{\sigma}}  \langle \partial_{\omega}^j \chi_0, \partial_\theta^i (\bbi - \bbp)f \rangle \chi_1.
\end{eqnarray*}
Similar to  estimates in \eqref{F-46}, we also find
\[ |\partial_{\omega}(\bbp_0 \partial_{\omega}^j \partial_\theta^i (\bbi - \bbp)f)| \lesssim \big(\frac{m}{\sigma}\big)^{\frac{j+1}{2}} \|\partial_\theta^i (\bbi - \bbp)f\|_{L_{\omega, \nu}^2} |\chi_1|,
\]
and this yields
\begin{align}
\begin{aligned} \label{F-46-2}
& \int_{\bbr^2}  \frac{\sigma}{m} |\partial_\omega \big( \bbp\partial_{\theta}^i \partial_{\omega}^j (\bbi - \bbp)f\big)|^2 \,d\omega d\nu \\
& \hspace{0.5cm}\lesssim  \big(\frac{m}{\sigma}\big)^{j} \| \partial_\theta^i (\bbi - \bbp)f\|_{L_{\omega, \nu}^2}^2 \int_{\bbr^2}  |\chi_1|^2 \,d\omega d\nu = \big(\frac{m}{\sigma}\big)^{j} \|\partial_\theta^i(\bbi - \bbp)f\|_{L_{\omega, \nu}^2}^2.
\end{aligned}
\end{align}
We then combine the estimates \eqref{F-46-1} and \eqref{F-46-2} to obtain
\begin{equation} \label{F-47}
\frac{\lambda_0}{m} \|\bbp\partial_{\theta}^i \partial_{\omega}^j (\bbi - \bbp)f\|_{\mu}^2 \lesssim \frac1m \big(\frac{m}{\sigma}\big)^{j} \|\partial_\theta^i (\bbi - \bbp)f\|_{L^2}^2.
\end{equation}
Now we substitute the estimate \eqref{F-47} into the relation \eqref{F-45} to get
\begin{align}\label{F-48}
\begin{aligned}
\int_{\mathbb{T}} \big\langle - {\mathcal L}_0 \partial_{\theta}^i \partial_{\omega}^j (\bbi - \bbp) f, \partial_{\theta}^i \partial_{\omega}^j (\bbi - \bbp)f \big\rangle d\theta
  \gtrsim \frac{\lambda_0}{m} \|\partial_{\theta}^i \partial_{\omega}^j (\bbi - \bbp)f\|_{\mu}^2 - \frac1m \big(\frac{m}{\sigma}\big)^{j} \|  \partial_{\theta}^i (\bbi - \bbp)f\|_{L^2}^2.
\end{aligned}
\end{align}
Thus, we combine the estimates \eqref{F-48} and \eqref{F-45} to get
\begin{equation} \label{F-49}
\frac{d}{dt} \|\partial_{\theta}^i \partial_{\omega}^j (\bbi - \bbp) f\|_{L^2}^2 + \frac{\lambda_0}{m} \|\partial_{\theta}^i \partial_{\omega}^j (\bbi - \bbp)f\|_{\mu}^2 \lesssim \frac1m \big(\frac{m}{\sigma}\big)^{j} \|  \partial_{\theta}^i (\bbi - \bbp)f\|_{L^2}^2 + \sum_{i=1}^6 \mci_{8i}.
\end{equation}
In the sequel, we come to estimate terms $\mci_{8i}, i = 1, \cdots 6$. \newline

\noindent $\bullet$ (Estimate on $\mci_{81}$): We use Cauchy's inequality with $\epsilon > 0$ and Lemma \ref{L5.6} $(i)$ to get
\begin{align*}
\begin{aligned}
& \Big| \big\langle -\partial_{\theta}^i [\partial_{\omega}^j, \omega \partial_{\theta}] (\bbi - \bbp) f, \partial_{\theta}^i \partial_{\omega}^j (\bbi - \bbp)f \big\rangle \Big| \\
& \hspace{1cm} \leq \frac{\epsilon}{m} \|\partial_{\theta}^i \partial_{\omega}^j (\bbi - \bbp)f\|^2_{L_{\omega, \nu}^2} + \frac{m}{4\epsilon} \|\partial_{\theta}^i [\partial_{\omega}^j, \omega \partial_{\theta}] (\bbi - \bbp) f\|^2_{L_{\omega, \nu}^2} \\
& \hspace{1cm}  \lesssim \frac{\epsilon}{m} \|\partial_{\theta}^i \partial_{\omega}^j (\bbi - \bbp)f\|^2_{L_{\omega, \nu}^2} + \frac{m}{4\epsilon} \|\partial_{\theta}^{i+1} \partial_{\omega}^{j-1} (\bbi - \bbp) f\|^2_{L_{\omega, \nu}^2},
\end{aligned}
\end{align*}
where $\epsilon$ will be determined later. Thus we get
\[
|\mci_{81}| \lesssim \frac{\epsilon}{m} \|\partial_{\theta}^i \partial_{\omega}^j (\bbi - \bbp)f\|^2_{L^2} + \frac{m}{4\epsilon} \|\partial_{\theta}^{i+1} \partial_{\omega}^{j-1} (\bbi - \bbp) f\|^2_{L^2}.
\]

\vspace{0.5cm}

\noindent $\bullet$ (Estimate on $\mci_{82}$): Note that for any function $h(\theta, \omega, \nu ,t)$, we find
\begin{align*}
\begin{aligned}
\mathcal{L}_0 h &=  \frac{\sigma}{m^2}  \frac{1}{\sqrt{M}} \partial_{\omega} \bigg[  M \partial_{\omega} \Big( \frac{h}{\sqrt{M}}  \Big) \bigg] = \frac{\sigma}{m^{2}} \partial_{\omega}^{2}h + \frac{1}{2m} \big( 1 - \frac{m}{2\sigma}(\omega - \nu)^2 \big)h.
\end{aligned}
\end{align*}
This together with a straightforward computation yields
\[
[\partial_{\omega}^j, \mathcal{L}_0] h  = \frac{\sigma}{m^{2}}[\partial_{\omega}^j,  \partial_{\omega}^{2}]h + \frac{1}{2m} [\partial_{\omega}^j,  1 - \frac{m}{2\sigma}(\omega - \nu)^2] h = \frac{1}{4\sigma} [(\omega - \nu)^2, \partial_{\omega}^j] h.
\]
We use the above identity to have
\begin{eqnarray*}
\mathcal{I}_{82} = \frac{1}{4\sigma} \int_{\mathbb{T}} \big\langle \partial_{\theta}^i  [(\omega - \nu)^2, \partial_{\omega}^j] (\bbi - \bbp) f, \partial_{\theta}^i \partial_{\omega}^j (\bbi - \bbp)f \big\rangle \,d\theta.
\end{eqnarray*}
By Lemma \ref{L5.6} $(ii)$,  we get
\begin{eqnarray*}
&& \Big| \big\langle \partial_{\theta}^i[(\omega - \nu)^2, \partial_{\omega}^j] (\bbi - \bbp) f, \partial_{\theta}^i \partial_{\omega}^j (\bbi - \bbp)f \big\rangle \Big| \\
&& \hspace{1cm} \lesssim \Big| \big\langle \partial_{\theta}^i \partial_{\omega}^{j-2} (\bbi - \bbp) f, \partial_{\theta}^i \partial_{\omega}^j (\bbi - \bbp) f \big\rangle \Big| + \Big| \big\langle (\omega - \nu) \partial_{\theta}^i \partial_{\omega}^{j-1} (\bbi - \bbp) f, \partial_{\theta}^i \partial_{\omega}^j (\bbi - \bbp)f \big\rangle \Big| \\
&& \hspace{1cm} \lesssim \|\partial_{\theta}^i \partial_{\omega}^{j-1} (\bbi - \bbp)f\|_{L_{\mu}^2}^2 + \frac{\sigma \epsilon}{m} \|\partial_{\theta}^i \partial_{\omega}^j (\bbi - \bbp)f\|_{L_{\omega, \nu}^2}^2 +  \frac{1}{\epsilon} \|\partial_{\theta}^i \partial_{\omega}^{j-1} (\bbi - \bbp)f\|_{L_{\mu}^2}^2,
\end{eqnarray*}
where we used $\big\langle \partial_{\theta}^i \partial_{\omega}^{j-2} (\bbi - \bbp) f, \partial_{\theta}^i \partial_{\omega}^j (\bbi - \bbp) f \big\rangle = \|\partial_{\theta}^i \partial_{\omega}^{j-1} (\bbi - \bbp) f\|_{L_{\omega, \nu}^{2}}$.
Thus we have
\[
|\mathcal{I}_{82}|\lesssim \frac{\varepsilon}{m} \|\partial_{\theta}^i \partial_{\omega}^j (\bbi - \bbp)f\|_{L^2}^2 + ( \frac{1}{\sigma}+ \frac{1}{\sigma\varepsilon}) \|\partial_{\theta}^i \partial_{\omega}^{j-1}(\bbi - \bbp)f\|_{\mu}^2.
\]

\vspace{0.5cm}

\noindent $\bullet$ (Estimate on $\mci_{83}$): Note that
\begin{align}
\begin{aligned} \label{F-50}
\mci_{83} &= \frac{\kappa}{2\sigma} \int_{\bbt} \langle \partial_{\theta}^i \partial_{\omega}^j \big( \mathcal{S}[\sqrt{M} f](\omega - \nu) (\bbi - \bbp)f \big), \partial_{\theta}^i \partial_{\omega}^j (\bbi - \bbp)  f \rangle \,d\theta \\
& - \frac{\kappa}{m} \int_{\bbt} \langle \partial_{\theta}^i \partial_{\omega}^j \big( \mathcal{S}[\sqrt{M} f]\partial_{\omega}(\bbi - \bbp)f \big), \partial_{\theta}^i \partial_{\omega}^j (\bbi - \bbp)  f \rangle\, d\theta =: \mathcal{I}_{83}^{1} + \mathcal{I}_{83}^{2}.
\end{aligned}
\end{align}

\noindent $\diamond$ (Estimate on $\mathcal{I}_{83}^{1}$): We use the following identity
\[ \partial_{\omega}^j \big( (\omega - \nu)h \big) = j \partial_{\omega}^{j-1} h + (\omega - \nu) \partial_{\omega}^j h \]
to get
\begin{align*}
\begin{aligned}
\mathcal{I}_{83}^{1}
&  = \frac{\kappa}{2\sigma} \sum_{0 \leq \ell \leq i} \binom{i}{\ell}  \int_{\bbt} \langle \partial_{\theta}^{i-\ell} \mathcal{S}[\sqrt{M} f]  \partial_{\theta}^{\ell} \partial_{\omega}^j \big( (\omega - \nu) (\bbi - \bbp)f \big), \partial_{\theta}^i \partial_{\omega}^j (\bbi - \bbp)f \rangle \,d\theta \\
&  = \frac{\kappa}{2\sigma} \sum_{0 \leq \ell \leq i} \binom{i}{\ell}  \int_{\bbt} \langle j\partial_{\theta}^{i-\ell} \mathcal{S}[\sqrt{M} f]  \partial_{\theta}^{\ell} \partial_{\omega}^{j-1} (\bbi - \bbp)f, \partial_{\theta}^i \partial_{\omega}^j (\bbi - \bbp)f \rangle \,d\theta \\
& + \frac{\kappa}{2\sigma} \sum_{0 \leq \ell \leq i} \binom{i}{\ell}  \int_{\bbt} \langle \partial_{\theta}^{i-\ell} \mathcal{S}[\sqrt{M} f] (\omega - \nu) \partial_{\theta}^{\ell} \partial_{\omega}^{j} (\bbi - \bbp)f, \partial_{\theta}^i \partial_{\omega}^j (\bbi - \bbp)f \rangle \,d\theta \\
& =: \mci_{83}^{11} + \mci_{83}^{12}.
\end{aligned}
\end{align*}
We use Lemma \ref{L6.1} $(i)$ and Cauchy's inequality with $\epsilon$ to obtain
\begin{align}
\begin{aligned} \label{F-50-1}
 |\mathcal{I}_{83}^{11}|
& \lesssim \frac{\kappa}{\sigma} \sum_{0 \leq \ell \leq i} \|\partial_{\theta}^{i-\ell} \mathcal{S}[\sqrt{M} f]\|_{L^{\infty}} \left| \int_{\bbt} \langle \partial_{\theta}^{\ell} \partial_{\omega}^{j-1} (\bbi - \bbp)f, \partial_{\theta}^i \partial_{\omega}^j (\bbi - \bbp)f \rangle \,d\theta \right| \\
& \lesssim \frac{\kappa}{\sigma} \sum_{0 \leq \ell \leq i}  \left| \int_{\bbt} \langle \partial_{\theta}^{\ell} \partial_{\omega}^{j-1} (\bbi - \bbp)f, \partial_{\theta}^i \partial_{\omega}^j (\bbi - \bbp)f \rangle \,d\theta \right| \\
& \lesssim  \frac{\kappa}{\sigma} \sum_{0 \leq \ell \leq i}  \Big( \frac{\epsilon \sigma}{m \kappa} \|\partial_{\theta}^i \partial_{\omega}^j (\bbi - \bbp)f\|_{L^2}^{2} + \frac{m \kappa} {\epsilon \sigma} \|\partial_{\theta}^{\ell} \partial_{\omega}^{j-1} (\bbi - \bbp)f\|_{L^2}^{2}  \Big) \\
& \lesssim \frac{\epsilon}{m} \|\partial_{\theta}^i \partial_{\omega}^j (\bbi - \bbp)f\|_{L^2}^{2} + \frac{m \kappa^2} {\epsilon \sigma^2} \sum_{0 \leq \ell \leq i} \|\partial_{\theta}^{\ell} \partial_{\omega}^{j-1} (\bbi - \bbp)f\|_{L^2}^{2}.
\end{aligned}
\end{align}
For the estimate $\mathcal{I}_{83}^{12}$, we use Lemma \ref{L6.1} $(i)$ to get
\begin{align}
\begin{aligned} \label{F-50-2}
 |\mathcal{I}_{83}^{12}|
& \lesssim \frac{\kappa}{\sigma} \sum_{0 \leq \ell \leq i} \|\partial_{\theta}^{i-\ell} \mathcal{S}[\sqrt{M} f]\|_{L^{\infty}} \left| \sqrt{\frac{\sigma}{m}} \int_{\bbt} \langle  \sqrt{\frac{m}{\sigma}} (\omega - \nu) \partial_{\theta}^{\ell} \partial_{\omega}^{j} (\bbi - \bbp)f, \partial_{\theta}^i \partial_{\omega}^j (\bbi - \bbp)f \rangle \,d\theta \right| \\
& \lesssim \frac{\kappa}{\sqrt{m\sigma}} \sum_{0 \leq \ell \leq i}  \Big(  \|\partial_{\theta}^i \partial_{\omega}^j (\bbi - \bbp)f\|_{L^2}^{2} + \|\partial_{\theta}^{\ell} \partial_{\omega}^{j} (\bbi - \bbp)f\|_{\mu}^{2}  \Big)  \lesssim  \frac{\kappa}{\sqrt{m\sigma}} \sum_{0 \leq \ell \leq i}   \|\partial_{\theta}^{\ell} \partial_{\omega}^{j} (\bbi - \bbp)f\|_{\mu}^{2}.
\end{aligned}
\end{align}
We combine the estimates \eqref{F-50-1} and \eqref{F-50-2} to get
\begin{align}
\begin{aligned} \label{F-50-4}
|\mathcal{I}_{83}^{1}|
& \lesssim \frac{\epsilon}{m} \|\partial_{\theta}^i \partial_{\omega}^j (\bbi - \bbp)f\|_{L^2}^{2} + \frac{m \kappa^2} {\epsilon \sigma^2} \sum_{0 \leq \ell \leq i} \|\partial_{\theta}^{\ell} \partial_{\omega}^{j-1} (\bbi - \bbp)f\|_{L^2}^{2} +  \frac{\kappa}{\sqrt{m\sigma}} \sum_{0 \leq \ell \leq i}   \|\partial_{\theta}^{\ell} \partial_{\omega}^{j} (\bbi - \bbp)f\|_{\mu}^{2}.
\end{aligned}
\end{align}

\noindent $\diamond$ (Estimate on $\mathcal{I}_{83}^{2}$):  Direct calculations yield
\begin{align*}
\begin{aligned}
\mathcal{I}_{83}^{2}
& = - \frac{\kappa}{m}  \sum_{0 \leq \ell \leq i} \binom{i}{\ell} \int_{\bbt} \langle \partial_{\theta}^{i- \ell} \mathcal{S}[\sqrt{M} f]\partial_{\theta}^{\ell}  \partial_{\omega}^{j+1} (\bbi - \bbp)f, \partial_{\theta}^i \partial_{\omega}^j (\bbi - \bbp)  f \rangle \,d\theta \\
& =  \frac{\kappa}{m}  \sum_{0 \leq \ell \leq i} \binom{i}{\ell} \int_{\bbt} \langle \partial_{\theta}^{i- \ell} \mathcal{S}[\sqrt{M} f]\partial_{\theta}^{\ell}  \partial_{\omega}^{j} (\bbi - \bbp)f, \partial_{\omega} \big( \partial_{\theta}^i \partial_{\omega}^j (\bbi - \bbp)  f\big)  \rangle \,d\theta.
\end{aligned}
\end{align*}
Now we use Lemma \ref{L6.1} $(i)$ and Cauchy's inequality to obtain
\begin{align}
\begin{aligned} \label{F-50-3}
|\mathcal{I}_{83}^{2}|
&\lesssim  \frac{\kappa}{m}  \sum_{0 \leq \ell \leq i} \|\partial_{\theta}^{i- \ell} \mathcal{S}[\sqrt{M} f]\|_{L^{\infty}} \sqrt{\frac{m}{\sigma}} \int_{\bbt} \Big| \langle \partial_{\theta}^{\ell}  \partial_{\omega}^{j} (\bbi - \bbp)f,\sqrt{\frac{\sigma}{m}} \partial_{\omega} \big( \partial_{\theta}^i \partial_{\omega}^j (\bbi - \bbp)  f\big)  \rangle \Big| d\theta\\
& \lesssim \frac{\kappa}{\sqrt{m \sigma}} \sum_{0 \leq \ell \leq i} \Big( \|\partial_{\theta}^{\ell}  \partial_{\omega}^{j} (\bbi - \bbp)f\|_{L^2}^2 + \|\partial_{\theta}^{i}  \partial_{\omega}^{j} (\bbi - \bbp)f\|_{\mu}^2 \Big) \\
&  \lesssim \frac{\kappa}{\sqrt{m \sigma}} \sum_{0 \leq \ell \leq i} \|\partial_{\theta}^{\ell}  \partial_{\omega}^{j} (\bbi - \bbp)f\|_{\mu}^2.
\end{aligned}
\end{align}
Combining the estimates \eqref{F-50}, \eqref{F-50-4}, and \eqref{F-50-3} gives
\begin{align*}
\begin{aligned}
|\mathcal{I}_{83}|
& \lesssim \frac{\epsilon}{m} \|\partial_{\theta}^i \partial_{\omega}^j (\bbi - \bbp)f\|_{L^2}^{2} + \frac{m \kappa^2} {\epsilon \sigma^2} \sum_{0 \leq \ell \leq i} \|\partial_{\theta}^{\ell} \partial_{\omega}^{j-1} (\bbi - \bbp)f\|_{L^2}^{2}  +  \frac{\kappa}{\sqrt{m\sigma}} \sum_{0 \leq \ell \leq i}   \|\partial_{\theta}^{\ell} \partial_{\omega}^{j} (\bbi - \bbp)f\|_{\mu}^{2}.
\end{aligned}
\end{align*}

\vspace{0.5cm}

\noindent $\bullet$ (Estimate on $\mci_{84}$): By a straightforward calculation, we estimate
\begin{align}
\begin{aligned} \label{F-53}
|\mci_{84}| &= \Big| \frac{\kappa}{2\sigma} \sum_{0 \leq \ell \leq i} \binom{k}{\ell} \int_{\mathbb{T}} \big\langle  \partial_{\theta}^{i-\ell} \mathcal{S}[\sqrt{M} f]  \partial_{\theta}^{\ell} \partial_{\omega}^j [\omega - \nu, \bbp]f, \partial_{\theta}^i \partial_{\omega}^j (\bbi - \bbp)f \big\rangle \,d\theta\Big| \\
& \lesssim \frac{\kappa}{\sigma} \sum_{0 \leq \ell \leq i} \|\partial_{\theta}^{i-\ell} \mathcal{S}[\sqrt{M} f]\|_{L^\infty} \Big( \sqrt{\frac{\sigma}{m}} \|\partial_{\theta}^i \partial_{\omega}^j (\bbi - \bbp)f\|_{L_{2}}^2 + \sqrt{\frac{m}{\sigma}} \|\partial_{\theta}^{\ell} \partial_{\omega}^j [\omega - \nu, \bbp ]f\|_{L_{2}}^2 \Big)\\
&\lesssim \frac{\kappa}{\sqrt{m\sigma}} \|\partial_{\theta}^i \partial_{\omega}^j (\bbi - \bbp)f\|_{L^2}^2 + \frac{\kappa}{\sigma} \sqrt{\frac{m}{\sigma}} \sum_{0 \leq \ell \leq i} \|\partial_{\theta}^{\ell} \partial_{\omega}^j [\omega - \nu, \bbp ]f\|_{L^2}^2.
\end{aligned}
\end{align}
Then it follows from Lemma \ref{L5.6} $(iii)$ and Lemma \ref{L5.5} that
\begin{align*}
\begin{aligned}
 \|\partial_{\theta}^{\ell} \partial_{\omega}^j [\omega - \nu, \bbp ]f\|_{L_{\omega, \nu}^2}^2 &\leq \int_{\bbr^2} |\partial_{\theta}^{\ell}f_1|^2 |\partial_{\omega}^j \big( (\omega - \nu)\chi_1 \big)|^2 \,d\omega d\nu + \frac{\sigma}{m} \int_{\bbr^2} |\partial_{\theta}^{\ell}f_1|^2 | |\partial_{\omega}^{j}  \chi_0|^2 \,d\omega d\nu \\
 &+  \int_{\bbr^2} \big|\langle \chi_1, (\omega - \nu)\partial_{\theta}^{\ell} (\bbi - \bbp)f \rangle \big|^2 |\partial_{\omega}^{j}  \chi_0|^2 \,d\omega d\nu \\
 & \lesssim \big(\frac{m}{\sigma} \big)^{j-1} |\partial_{\theta}^{\ell}f_1|^2 + \frac{\sigma}{m}|\partial_{\theta}^{\ell}f_1|^2 \times  \big(\frac{m}{\sigma} \big)^{j} + \frac{\sigma}{m}| \|\partial_{\theta}^{\ell} (\bbi - \bbp)f\|_{L_\mu^2}^2 \times \big(\frac{m}{\sigma} \big)^{j} \\
 & \lesssim  \big(\frac{m}{\sigma} \big)^{j-1} |\partial_{\theta}^{\ell}f_1|^2 +  \big(\frac{m}{\sigma} \big)^{j-1}  \|\partial_{\theta}^{\ell} (\bbi - \bbp)f\|_{L_\mu^2}^2,
\end{aligned}
\end{align*}
where we used the following estimate for the second inequality.
\begin{align*}
\begin{aligned}
\big|\langle \chi_1, (\omega - \nu)\partial_{\theta}^{\ell} (\bbi - \bbp)f \rangle \big| &\leq \bigg(\int_{\bbr^2} \chi_1^2 d\omega d\nu \bigg)^{\frac{1}{2}}  \bigg( \frac{\sigma}{m} \int_{\bbr^2} \frac{m}{\sigma} (\omega -\nu)^2 |\partial_{\theta}^{\ell} (\bbi - \bbp)f|^2 d\omega d\nu \bigg)^{\frac{1}{2}} \\
& \leq \sqrt{ \frac{\sigma}{m}} \|\partial_{\theta}^{\ell} (\bbi - \bbp)f\|_{L_\mu^2}.
\end{aligned}
\end{align*}
Then we have
\begin{align}
\begin{aligned} \label{F-53-1}
 \|\partial_{\theta}^{\ell} \partial_{\omega}^j [\omega - \nu, \bbp ]f\|_{L^2}^2
 & \lesssim  \big(\frac{m}{\sigma} \big)^{j-1} \|\partial_{\theta}^{\ell}f_1\|_{L^2}^2 +  \big(\frac{m}{\sigma} \big)^{j-1}  \|\partial_{\theta}^{\ell} (\bbi - \bbp)f\|_{\mu}^2.
\end{aligned}
\end{align}
Now we substitute \eqref{F-53-1} into  \eqref{F-53} to obtain
\begin{align*}
\begin{aligned}
|\mci_{84}|
&\lesssim \frac{\kappa}{\sqrt{m\sigma}} \|\partial_{\theta}^i \partial_{\omega}^j (\bbi - \bbp)f\|_{L^2}^2 + \frac{\kappa}{\sigma} \sqrt{\frac{m}{\sigma}} \sum_{0 \leq \ell \leq i} \Big(  \big(\frac{m}{\sigma} \big)^{j-1} \|\partial_{\theta}^{\ell}f_1\|_{L^2}^2 +  \big(\frac{m}{\sigma} \big)^{j-1}  \|\partial_{\theta}^{\ell} (\bbi - \bbp)f\|_{\mu}^2 \Big) \\
& \lesssim \frac{\kappa}{\sqrt{m\sigma}} \|\partial_{\theta}^i \partial_{\omega}^j (\bbi - \bbp)f\|_{L^2}^2 +\frac{\kappa}{\sqrt{m\sigma}} \big(\frac{m}{\sigma} \big)^{j}  \sum_{0 \leq \ell \leq i} \Big(  \|\partial_{\theta}^{\ell}f_1\|_{L^2}^2 + \|\partial_{\theta}^{\ell} (\bbi - \bbp)f\|_{\mu}^2 \Big).
\end{aligned}
\end{align*}

\vspace{0.5cm}

\noindent $\bullet$ (Estimate on $\mci_{85}$): Similarly as before, we use Lemma \ref{L6.1} $(i)$ and Cauchy's inequality to obtain
\begin{align}
\begin{aligned} \label{E-69}
|\mci_{85}| &= \bigg| \frac{\kappa}{m} \sum_{0 \leq \ell \leq i} \binom{k}{\ell} \int_{\bbt} \partial_{\theta}^{i-\ell} \mathcal{S}[\sqrt{M} f] \big\langle \partial_{\theta}^{\ell} \partial_{\omega}^j [\partial_{\omega}, \bbp]f, \partial_{\theta}^i \partial_{\omega}^j (\bbi - \bbp)f \big\rangle \,d\theta \bigg| \\
&\lesssim \frac{\kappa}{m} \sum_{0 \leq \ell \leq i} \|\partial_{\theta}^{i-\ell} \mathcal{S}[\sqrt{M} f]\|_{L^{\infty}} \int_{\bbt} \Big| \big\langle \partial_{\theta}^{\ell} \partial_{\omega}^j [\partial_{\omega}, \bbp]f, \partial_{\theta}^i \partial_{\omega}^j (\bbi - \bbp)f \big\rangle\Big| d\theta \\
& \lesssim \frac{\kappa}{m} \sum_{0 \leq \ell \leq i}  \Big( \sqrt{\frac{m}{\sigma}} \|\partial_{\theta}^i \partial_{\omega}^j (\bbi - \bbp)f\|_{L_{2}^2}^2 + \sqrt{\frac{\sigma}{m}} \|\partial_{\theta}^{\ell} \partial_{\omega}^j [\partial_{\omega}, \bbp ]f\|_{L_{2}^2}^2  \Big) \\
&\lesssim \frac{\kappa}{\sqrt{m\sigma}} \|\partial_{\theta}^i \partial_{\omega}^j (\bbi - \bbp)f\|_{L^2}^2 + \frac{\kappa}{m} \sqrt{\frac{\sigma}{m}} \sum_{0 \leq \ell \leq i} \|\partial_{\theta}^{\ell} \partial_{\omega}^j [\partial_{\omega}, \bbp ]f\|_{L^2}^2.
\end{aligned}
\end{align}
To estimate the second R.H.S. term in inequality \eqref{E-69}, we use the fact that
\begin{align*}
\begin{aligned}
& \Big| \big\langle \chi_0, \partial_{\omega} \partial_{\theta}^{\ell} (\bbi - \bbp)f \big\rangle \Big| \leq \sqrt{\frac{m}{\sigma}} \| \partial_{\theta}^{\ell} (\bbi - \bbp)f\|_{L_{\mu}^{2}}, \quad \Big| \big\langle \chi_1, \partial_{\omega}\partial_{\theta}^{\ell} (\bbi - \bbp)f \big\rangle \Big| \leq \sqrt{\frac{m}{\sigma}} \| \partial_{\theta}^{\ell} (\bbi - \bbp)f\|_{L_{\mu}^{2}},
\end{aligned}
\end{align*}
together with Lemma \ref{L5.6} $(v)$ and Lemma \ref{L5.5} to get
\begin{align*}
\begin{aligned}
 \|\partial_{\theta}^{\ell} \partial_{\omega}^j [\partial_{\omega}, \bbp ]f\|_{L_{\omega, \nu}^{2}}^2 &\lesssim \big(\frac{m}{\sigma}\big)^2  |\partial_{\theta}^{\ell} f_1|^2  \int_{\bbr^2}\Big|\partial_{\omega}^{j}\big( (\omega - \nu)\chi_1 \big) \Big|^2 d\omega d\nu +   \frac{m}{\sigma}  |\partial_{\theta}^{\ell} f_1|^2   \int_{\bbr^2}\big|\partial_{\omega}^{j} \chi_0 \big|^2 d\omega d\nu \\
&  +\int_{\bbr^2}\Big| \big\langle \chi_0, \partial_{\omega} \partial_{\theta}^{\ell} (\bbi - \bbp)h \big\rangle \Big|^2 |\partial_{\omega}^{j}  \chi_0|^2 d\omega d\nu +  \int_{\bbr^2} \Big|\big\langle \chi_1, \partial_{\omega}\partial_{\theta}^{\ell} (\bbi - \bbp)h \big\rangle \Big|^2  |\partial_{\omega}^{j}  \chi_1|^2 d\omega d\nu \\
& \lesssim \big(\frac{m}{\sigma}\big)^{j+1} |\partial_{\theta}^{\ell} f_1|^2 +  \big(\frac{m}{\sigma}\big)^{j+1}  \| \partial_{\theta}^{\ell} (\bbi - \bbp)f\|_{L_{\mu}^{2}}^2.
\end{aligned}
\end{align*}
This yields
\begin{align}
\begin{aligned} \label{E-69-1}
 \|\partial_{\theta}^{\ell} \partial_{\omega}^j [\partial_{\omega}, \bbp ]f\|_{L^{2}}^2
 \lesssim \big(\frac{m}{\sigma}\big)^{j+1} \|\partial_{\theta}^{\ell} f_1\|_{L^2}^2 +  \big(\frac{m}{\sigma}\big)^{j+1}  \| \partial_{\theta}^{\ell} (\bbi - \bbp)f\|_{\mu}^2.
\end{aligned}
\end{align}
We substitute \eqref{E-69-1} into \eqref{E-69} to deduce
\begin{align*}
\begin{aligned}
|\mci_{85}|
&\lesssim \frac{\kappa}{\sqrt{m\sigma}} \|\partial_{\theta}^i \partial_{\omega}^j (\bbi - \bbp)f\|_{L^2}^2 + \frac{\kappa}{m} \sqrt{\frac{\sigma}{m}} \sum_{0 \leq \ell \leq i} \Big( \big(\frac{m}{\sigma}\big)^{j+1} \|\partial_{\theta}^{\ell} f_1\|_{L^2}^2 +  \big(\frac{m}{\sigma}\big)^{j+1}  \| \partial_{\theta}^{\ell} (\bbi - \bbp)f\|_{\mu}^2 \Big) \\
& \lesssim \frac{\kappa}{\sqrt{m\sigma}} \|\partial_{\theta}^i \partial_{\omega}^j (\bbi - \bbp)f\|_{L^2}^2 + \frac{\kappa}{\sqrt{m\sigma}} \big(\frac{m}{\sigma}\big)^{j}  \sum_{0 \leq \ell \leq i} \Big( \|\partial_{\theta}^{\ell} f_1\|_{L^2}^2 +  \| \partial_{\theta}^{\ell} (\bbi - \bbp)f\|_{\mu}^2 \Big).
\end{aligned}
\end{align*}

\vspace{0.5cm}

\noindent $\bullet$ (Estimate on $\mci_{86}$): We split $\mci_{86}$ into two terms:
\begin{align}
\begin{aligned} \label{F-58}
\mci_{86} &= \int_{\mathbb{T}} \big\langle \partial_{\theta}^{i+1} \partial_{\omega}^j [\bbp, \omega - \nu]f, \partial_{\theta}^i \partial_{\omega}^j (\bbi - \bbp)f \big\rangle d\theta \\
& + \int_{\mathbb{T}} \big\langle \partial_{\theta}^{i+1} \partial_{\omega}^j [\bbp,  \nu]f, \partial_{\theta}^i \partial_{\omega}^j (\bbi - \bbp)f \big\rangle d\theta =: \mci_{86}^1 + \mci_{86}^2.
\end{aligned}
\end{align}

\noindent $\diamond$ (Estimate on $\mathcal{I}_{86}^{1}$): We take $\ell = i+1$ in \eqref{F-53-1} to get
\begin{align*}
\begin{aligned}
 \|\partial_{\theta}^{i+1} \partial_{\omega}^j [\bbp, \omega - \nu ]f\|_{L^2}^2
 & \lesssim  \big(\frac{m}{\sigma} \big)^{j-1} \|\partial_{\theta}^{i+1}f_1\|^2 +  \big(\frac{m}{\sigma} \big)^{j-1}  \|\partial_{\theta}^{i+1} (\bbi - \bbp)f\|_{\mu}^2.
\end{aligned}
\end{align*}
This deduces
\begin{align}
\begin{aligned} \label{F-72-1}
|\mathcal{I}_{86}^{1}| &\lesssim \frac{\epsilon}{m} \|\partial_{\theta}^i \partial_{\omega}^j (\bbi - \bbp)f\|_{L^2}^{2} + \frac{m}{\epsilon} \|\partial_{\theta}^{i+1} \partial_{\omega}^j [\omega - \nu, \bbp ]f\|_{L^2}^2 \\
 & \lesssim  \frac{\epsilon}{m} \|\partial_{\theta}^i \partial_{\omega}^j (\bbi - \bbp)f\|_{L^2}^{2} +  \frac{m}{\epsilon} \big(\frac{m}{\sigma} \big)^{j-1} \Big( \|\partial_{\theta}^{i+1}f_1\|^2 +   \|\partial_{\theta}^{i+1} (\bbi - \bbp)f\|_{\mu}^2 \Big).
\end{aligned}
\end{align}

\noindent $\diamond$ (Estimate on $\mathcal{I}_{86}^{2}$): Note that
\begin{align*}
\begin{aligned}
 \bigg(\int_{\bbr} \nu g(\nu) \,d\nu \bigg)^2 &\leq  \int_{\bbr} g(\nu)\,d\nu  \int_{\bbr} (1+ \nu^2) g(\nu)\,d\nu = \|g\|_{\nu}^2, \\
 \Big|  \langle\chi_0, \nu \partial_{\theta}^{i+1}(\bbi - \bbp) f \rangle  \Big| & \leq \Big( \int_{\bbr^2} \nu^2 \chi_0^2 \,d\omega d\nu \Big)^{\frac{1}{2}} \Big(\int_{\bbr^2} |\partial_{\theta}^{i+1}(\bbi - \bbp) f |^2 \,d\omega d\nu \Big)^{\frac{1}{2}}\cr
 &  \leq \|g\|_{\nu} \|\partial_{\theta}^{i+1}(\bbi - \bbp) f \|_{L_{\omega, \nu}^{2}}, \\
  \Big|  \langle\chi_1, \nu \partial_{\theta}^{i+1}(\bbi - \bbp) f \rangle  \Big| & \leq \Big( \int_{\bbr^2} \nu^2 \chi_1^2 \,d\omega d\nu \Big)^{\frac{1}{2}} \Big(\int_{\bbr^2} |\partial_{\theta}^{i+1}(\bbi - \bbp) f |^2 \,d\omega d\nu \Big)^{\frac{1}{2}} \cr
  & \leq \|g\|_{\nu} \|\partial_{\theta}^{i+1}(\bbi - \bbp) f \|_{L_{\omega, \nu}^{2}}.
\end{aligned}
\end{align*}
We then use Lemma \ref{L5.6} $(iv)$ and Lemma \ref{L5.5} to find
\begin{align*}
\begin{aligned}
&\|\partial_{\theta}^{i+1} \partial_{\omega}^{j} [\bbp, \nu]f\|_{L_{\omega, \nu}^{2}}^{2} \cr
&\quad \lesssim  |\partial_{\theta}^{i+1}f_0|^2 \int_{\bbr^2} | \partial_{\omega}^{j} (\nu \chi_0)|^2 \,d\omega d\nu  + |\partial_{\theta}^{i+1}f_1|^2 \int_{\bbr^2} | \partial_{\omega}^{j} (\nu \chi_1)|^2 \,d\omega d\nu \\
&\quad  + |\partial_{\theta}^{i+1} f_0|^2 \bigg(\int_{\bbr} \nu g(\nu) \,d\nu \bigg)^2 \int_{\bbr^2} | \partial_{\omega}^{j}  \chi_0|^2 \,d\omega d\nu  + |\partial_{\theta}^{i+1} f_1|^2 \bigg(\int_{\bbr} \nu g(\nu) \,d\nu \bigg)^2 \int_{\bbr^2} | \partial_{\omega}^{j}  \chi_1|^2 \,d\omega d\nu \\
&\quad +  \int_{\bbr^2} \Big|  \langle\chi_0, \nu \partial_{\theta}^{i+1}(\bbi - \bbp) f \rangle  \Big|^2 | \partial_{\omega}^{j}\chi_0|^2 \,d\omega d\nu + \int_{\bbr^2} \Big|  \langle\chi_1, \nu \partial_{\theta}^{i+1}(\bbi - \bbp) f \rangle  \Big|^2 | \partial_{\omega}^{j}\chi_1|^2 \,d\omega d\nu \\
&\quad  \lesssim  \big( \frac{m}{\sigma} \big)^{j} \|g\|_{\nu}^{2} \Big( |\partial_{\theta}^{i+1}f_0|^2 + |\partial_{\theta}^{i+1}f_1|^2 + \|\partial_{\theta}^{i+1}(\bbi - \bbp) f \|_{L_{\omega, \nu}^{2}}^2 \Big).
\end{aligned}
\end{align*}
Thus, we have
\begin{align}
\begin{aligned} \label{E74}
\|\partial_{\theta}^{i+1} \partial_{\omega}^{j} [\bbp, \nu]f\|_{L^{2}}^{2}
\lesssim  \big( \frac{m}{\sigma} \big)^{j} \|g\|_{\nu}^{2} \Big( \|\partial_{\theta}^{i+1}f_0\|_{L^2}^2 + \|\partial_{\theta}^{i+1}f_1\|_{L^2}^2 + \|\partial_{\theta}^{i+1}(\bbi - \bbp) f \|_{L^{2}}^2 \Big).
\end{aligned}
\end{align}
We next use Cauchy inequality with $\epsilon > 0$ and \eqref{E74} to obtain
\begin{align}
\begin{aligned} \label{F-74-1}
|\mathcal{I}_{86}^{2}| &\lesssim \frac{\epsilon}{m} \|\partial_{\theta}^i \partial_{\omega}^j (\bbi - \bbp)f\|_{L^2}^{2} + \frac{m}{\epsilon} \|\partial_{\theta}^{i+1} \partial_{\omega}^j [ \bbp, \nu ]f\|_{L^2}^2 \\
 & \lesssim  \frac{\epsilon}{m} \|\partial_{\theta}^i \partial_{\omega}^j (\bbi - \bbp)f\|_{L^2}^{2} +  \frac{m}{\epsilon} \big(\frac{m}{\sigma} \big)^{j} \|g\|_{\nu}^{2} \Big( \|\partial_{\theta}^{i+1}f_0\|_{L^2}^2 +\|\partial_{\theta}^{i+1}f_1\|_{L^2}^2 +   \|\partial_{\theta}^{i+1} (\bbi - \bbp)f\|_{L^2}^2 \Big).
\end{aligned}
\end{align}
We combine the estimates \eqref{F-58}, \eqref{F-72-1}, and \eqref{F-74-1} to obtain
\begin{align*}
\begin{aligned}
|\mathcal{I}_{86}|
 & \lesssim  \frac{\epsilon}{m} \|\partial_{\theta}^i \partial_{\omega}^j (\bbi - \bbp)f\|_{L^2}^{2} +  \frac{m}{\epsilon} \big(\frac{m}{\sigma} \big)^{j-1}   \Big( \|\partial_{\theta}^{i+1}f_1\|_{L^2}^2 +   \|\partial_{\theta}^{i+1} (\bbi - \bbp)f\|_{L^2}^2 \Big) \\
 & +  \frac{m}{\epsilon} \big(\frac{m}{\sigma} \big)^{j} \|g\|_{\nu}^{2} \Big( \|\partial_{\theta}^{i+1}f_0\|_{L^2}^2 +\|\partial_{\theta}^{i+1}f_1\|_{L^2}^2 +   \|\partial_{\theta}^{i+1} (\bbi - \bbp)f\|_{L^2}^2 \Big).
\end{aligned}
\end{align*}
Now we substitute all the estimates of $\mci_{8i}$, $i =1, \cdots, 6$ into \eqref{F-49} and take $\epsilon > 0$ small enough and use Poincar\'e's inequality to obtain
\begin{align*}
\begin{aligned}
& \frac{d}{dt} \|\partial_{\theta}^i \partial_{\omega}^j (\bbi - \bbp) f\|_{L^2}^2 + \frac{1}{m} \|\partial_{\theta}^i \partial_{\omega}^j (\bbi - \bbp)f\|_{\mu}^2 \\
& \hspace{1cm} \lesssim \left(m + \frac{1}{\sigma} + \frac{m \kappa^2}{\sigma^2} \right)\|\partial_{\theta}^{i+1} \partial_{\omega}^{j-1} (\bbi - \bbp) f\|^2_{\mu} + \frac{\kappa}{\sqrt{m \sigma}} \|\partial_{\theta}^{i} \partial_{\omega}^{j} (\bbi - \bbp) f\|^2_{\mu} \\
& \hspace{1cm} +  \big(\frac{m}{\sigma}\big)^{j-1} \left(\frac{m}{\sigma} \frac{\kappa}{\sqrt{m \sigma}} + \frac{1}{\sigma} + m + \frac{m^2 \|g\|_{\nu}^2}{\sigma}  \right) \Big\{ \|\partial_{\theta}^{i+1}(f_0, f_1)\|_{L^2}^2  + \|\partial_{\theta}^{i+1} (\bbi - \bbp)f\|_{\mu}^2 \Big\}\cr
& \hspace{1cm} + \Big(\frac{1}{\sigma} + \frac{m\kappa^2}{\sigma^2}\Big) \|\partial_\omega^{j-1} (\bbi - \bbp)f\|_{L^2}^2 + \frac{\kappa}{\sqrt{m\sigma}}\|\partial_\omega^j  (\bbi - \bbp)f\|_{\mu}^2 \\
& \hspace{1cm}+ \frac{\kappa}{\sqrt{m\sigma}}(\frac{m}{\sigma})^j \left(\|f_1\|_{L^2}^2 + \| (\bbi - \bbp)f\|_{\mu}^2 \right).
\end{aligned}
\end{align*}
We notice that there exists a positive constant $C>0$ such that
\[
\max\left\{\frac{1}{\sigma} + \frac{m \kappa^2}{\sigma^2}, \,\frac{m}{\sigma}\frac{\kappa}{\sqrt{m \sigma}} + \frac{1}{\sigma} + \frac{m^2 \|g\|_{\nu}^2}{\sigma}\right\} \leq C m
\]
due to our main assumptions. This concludes
\begin{align*}
\begin{aligned}
& \frac{d}{dt} \|\partial_{\theta}^i \partial_{\omega}^j (\bbi - \bbp) f\|_{L^2}^2 + \frac{1}{m} \|\partial_{\theta}^i \partial_{\omega}^j (\bbi - \bbp)f\|_{\mu}^2 \\
& \hspace{1cm} \lesssim  m\|\partial_{\theta}^{i+1} \partial_{\omega}^{j-1} (\bbi - \bbp) f\|^2_{\mu} + m \big(\frac{m}{\sigma}\big)^{j-1} \Big\{ \|\partial_{\theta}^{i+1}(f_0, f_1)\|_{L^2}^2  + \|\partial_{\theta}^{i+1} (\bbi - \bbp)f\|_{\mu}^2 \Big\}\cr
& \hspace{1cm} +  \big( \frac{1}{\sigma}+ \frac{m\kappa^2}{\sigma^2} \big) \|\partial_\omega^{j-1} (\bbi - \bbp)f\|_{L^2}^2 + \frac{\kappa}{\sqrt{m\sigma}}\|\partial_\omega^j  (\bbi - \bbp)f\|_{\mu}^2 + \big( \frac{1}{m} + \frac{\kappa}{\sqrt{m\sigma}}\big) (\frac{m}{\sigma})^j \left(\|f_1\|_{L^2}^2 + \| (\bbi - \bbp)f\|_{\mu}^2 \right)\cr
& \hspace{1cm} \lesssim  m\|\partial_{\theta}^{i+1} \partial_{\omega}^{j-1} (\bbi - \bbp) f\|^2_{\mu} + m^{2j-1} \Big\{ \|\partial_{\theta}^{i+1}(f_0, f_1)\|_{L^2}^2  + \|\partial_{\theta}^{i+1} (\bbi - \bbp)f\|_{\mu}^2 \Big\}\cr
& \hspace{1cm} + m\|\partial_\omega^{j-1} (\bbi - \bbp)f\|_{L^2}^2 + \frac{\kappa}{\sqrt{m\sigma}}\|\partial_\omega^j  (\bbi - \bbp)f\|_{\mu}^2 + m^{2j-1}  \left(\|f_1\|_{L^2}^2 + \| (\bbi - \bbp)f\|_{\mu}^2 \right),
\end{aligned}
\end{align*}
for $1 \leq i + j \leq s$.

%
%



\end{document}